\documentclass[a4paper, 11pt]{article} 

\usepackage[svgnames]{xcolor} 

\usepackage[utf8]{inputenc} 
\usepackage[T1]{fontenc} 
\usepackage{PTSerif} 
\usepackage[nottoc]{tocbibind}

\usepackage{pgfplots}
\pgfplotsset{width=10cm,compat=1.9}

\usepackage{hyperref}
\usepackage{amsmath}
\usepackage{amssymb}
\usepackage{amsthm}
\usepackage{blindtext}
\usepackage[margin = 1 in]{geometry}
\usepackage{mathtools}
\usepackage{graphicx}
\usepackage{tikz-cd}
\usepackage{footnote}
\usepackage{fancyhdr}
\usepackage{hyperref}
\usepackage{stmaryrd}
\usepackage{comment}
\usepackage[normalem]{ulem}
\usepackage{quiver}
\usepackage{adjustbox}
\usepackage{pgfplots}
\pgfplotsset{width=10cm,compat=1.9}

\usepackage{todonotes}

\usepackage{mathrsfs}
\usepackage{cleveref}

\newtheorem{innercustomthm}{Theorem}
\newenvironment{customthm}[1]
  {\renewcommand\theinnercustomthm{#1}\innercustomthm}
  {\endinnercustomthm}

  \newtheorem{innercustomcor}{Corollary}
\newenvironment{customcor}[1]
  {\renewcommand\theinnercustomcor{#1}\innercustomcor}
  {\endinnercustomcor}

\newtheorem{innercustomprop}{Proposition}
\newenvironment{customprop}[1]
  {\renewcommand\theinnercustomprop{#1}\innercustomprop}
  {\endinnercustomprop}

\hypersetup{
    colorlinks,
    linkcolor=blue,
    urlcolor=blue
}

\newtheorem{theorem}{Theorem}[section]
\newtheorem{corollary}[theorem]{Corollary}
\newtheorem{lemma}[theorem]{Lemma}
\newtheorem{proposition}[theorem]{Proposition}
\newtheorem{construction}[theorem]{Construction}

\newtheorem{definition}[theorem]{Definition}

\newtheorem{remark}[theorem]{Remark}

\theoremstyle{definition}
\newtheorem{example}[theorem]{Example}

\graphicspath{ {images/} }

\newcommand{\?}{{\color{red} ???}} 
\newcommand{\mb}{\mathbb}
\newcommand{\mc}{\mathcal}
\newcommand{\mf}{\mathfrak}
\newcommand{\msf}{\mathsf}

\newcommand{\Q}{\mb{Q}}
\newcommand{\R}{\mb{R}}
\newcommand{\Z}{\mb{Z}}
\newcommand{\C}{\mb{C}}
\newcommand{\A}{\mb{A}}
\newcommand{\F}{\mb{F}}
\newcommand{\N}{\mb{N}}

\renewcommand{\L}{\mc{L}}
\renewcommand{\O}{\mc{O}}

\renewcommand{\P}{\mb{P}}
\newcommand{\G}{\mb{G}}

\newcommand{\m}{\mf{m}}

\newcommand{\eps}{\epsilon}
\newcommand*{\sheafhom}{\mathcal{H}\kern -.5pt om}

\newcommand{\et}{\text{\'et}}

\newcommand{\kbar}{\overline{k}}

\newcommand{\mP}{\mc{P}}

\newcommand{\Fq}{\mb{F}_q}

\newcommand{\Fqb}{\overline{\mb{F}_q}}
\newcommand{\Qlb}{\overline{\mb{Q}_l}}
\newcommand{\Qpb}{\overline{\mb{Q}_p}}

\DeclareMathOperator{\Gal}{Gal}
\DeclareMathOperator{\Lie}{Lie}

\DeclareMathOperator{\Hom}{Hom}

\DeclareMathOperator{\Spec}{Spec}

\DeclareMathOperator{\im}{im}

\DeclareMathOperator{\Res}{Res}
\DeclareMathOperator{\Aut}{Aut}
\DeclareMathOperator{\coker}{coker}

\DeclareMathOperator{\End}{End}

\DeclareMathOperator{\GL}{GL}

\DeclareMathOperator{\ch}{char}

\DeclareMathOperator{\sing}{sing}

\DeclareMathOperator{\dR}{dR}
\DeclareMathOperator{\cris}{cris}

\DeclareMathOperator{\supp}{supp}
\DeclareMathOperator{\Vol}{Vol}
\DeclareMathOperator{\len}{len}

\DeclareMathOperator{\red}{red}
\DeclareMathOperator{\Isoc}{Isoc}

\DeclareMathOperator{\cyc}{cyc}
\DeclareMathOperator{\mon}{mon}
\DeclareMathOperator{\univ}{univ}
\DeclareMathOperator{\Sym}{Sym}
\DeclareMathOperator{\std}{std}
\DeclareMathOperator{\Fr}{Fr}
\DeclareMathOperator{\geo}{geo}
\DeclareMathOperator{\Hyp}{Hyp}
\DeclareMathOperator{\Fil}{Fil}
\DeclareMathOperator{\Perv}{Perv}
\DeclareMathOperator{\intt}{int}
\DeclareMathOperator{\sss}{ss}

\let\mi\relax
\DeclareMathOperator{\mi}{mid}
\newcommand{\cm}{*_{\mi}}

\newcommand{\catname}[1]{{\normalfont\textbf{#1}}}

\newcommand{\bP}{\catname{P}}
\newcommand{\bD}{\catname{D}}
\newcommand{\bN}{\catname{N}}
 
\newcommand{\ol}{\overline}
 
\newcommand{\oyd}{\Omega^{n-1-q}_{(\ol{Y_m}, D)}}

\DeclareMathOperator{\FI}{\catname{F-Isoc}}
\DeclareMathOperator{\oFI}{\catname{F-Isoc}^\dag}

\renewcommand{\bar}{\overline}

\usepackage{sectsty}
\title{Convolution monodromy groups and the Shafarevich conjecture for hypersurfaces in tori}
\author{Caleb Ji}
\date{\today}

\begin{document}
\maketitle 

\setcounter{tocdepth}{1}

\begin{abstract}
    The Shafarevich conjecture for a class of varieties over a number field posits the finitude of those with good reduction outside a finite set of primes.  In the case of hypersurfaces in the torus $\mathbb{G}_m^n$, a natural class to consider are those with a fixed Newton polyhedron that are nondegenerate with respect to it.  Using an approach similar to that of Lawrence-Sawin for abelian varieties (arXiv:2004.09046), we prove the Shafarevich conjecture for certain classes of Newton polyhedra.  In the course of our proof we construct a fiber functor for the Tannakian category of perverse sheaves on the torus in positive characteristic and compute the weights of the Frobenius on it, which leads to the big monodromy results which are key to this method.
\end{abstract}

\tableofcontents

\section{Introduction}
This paper has two main goals: to further develop the theory of convolution monodromy groups on the torus, and to provide an application to the Shafarevich conjecture.  We begin by recalling the statement of the latter.  The Shafarevich conjecture for curves over a fixed number field $K$ states that there are only finitely many of a fixed genus which have good reduction outside a finite set of primes.  Its proof by Faltings in \cite{Faltings1983} led directly to another important finiteness result, namely Mordell's conjecture.  More recently, Lawrence and Venkatesh \cite{LV} have produced a new proof of Mordell's conjecture using p-adic Hodge theory and monodromy calculations, which did not go through the Shafarevich conjecture.  Interestingly, their method is however applicable to other versions of the Shafarevich conjecture which have seen renewed interest.    
Let $K$ be a number field with ring of integers $\O_K$ and let $S$ be a finite set of primes of $\O_K$.  Broadly speaking, in a Shafarevich-type conjecture one considers a family of algebraic varieties over $K$ and posits that there are only finitely many varieties in the family with good reduction outside $S$.  While there is not a single statement that encapsulates all instances in which this statement ought to be true, it is believed to hold in fairly wide generality; see for instance \cite{MR3673652}.  In the cases of the Shafarevich conjecture proven so far, the varieties in question are usually taken to be projective.  In additional to Faltings's original result for curves and abelian varieties \cite{Faltings1983}, examples of this include K3 surfaces \cite{she}, hyper-K\"ahler varieties \cite{Andre}, flag varieties \cite{Jav}, del Pezzo surfaces \cite{Scholl}, and hypersurfaces in abelian varieties \cite{finiteness}.  

In the present work, we will instead consider non-projective varieties, namely hypersurfaces in tori, which we assume to be irreducible.  Aside from the interesting challenge to find replacements for arguments that require projectivity, the motivation in studying such varieties stems from multiple additional sources.  First, subvarieties of the torus carry highly interesting cohomological information, including weights and mixed Hodge structures, with number theoretic applications; see for instance \cite{Batyrev}.  Additionally, the fact that the torus is an algebraic group implies that it gives rise to a Tannakian category of perverse sheaves under convolution of which its hypersurfaces gives objects of, as explained in \cite{FFK}.  Finally, one can often explicitly compute these aspects of hypersurfaces in tori in terms of their defining equations.  We will make use of all of these aspects in this work.  

As in every other instance of the Shafarevich conjecture, one must pick out a natural moduli space for which finiteness is both plausible and interesting.  In the case of hypersurfaces in tori, the result cannot hold unless we group them by their Newton polyhedron, which is defined to be the convex hull of the set of points corresponding to the monomials with nonzero coefficients in the defining polynomial.  Furthermore, we require the hypersurfaces to be nondegenerate with respect to their Newton polyhedron, which is a natural smoothness condition.  More precisely, this condition means that if $f_{\tau}$ is the restriction of the defining polynomial $f$ to the monomials appearing in any given face $\tau$ of the Newton polyhedron, then the equations 
\[
\frac{\partial f_{\tau}}{\partial x_1} = \cdots = \frac{\partial f_{\tau}}{\partial x_n} = 0 
\]
have no solution in $\G_m^n$.
Fixing a Newton polyhedron, the Shafarevich conjecture in this setting asks if there are finitely many such hypersurfaces with good reduction outside a finite set of primes of the underlying number field.  

Broadly speaking, our approach to this question consists of two parts: a monodromy calculation and an application of the method of \cite{LV} using period maps and p-adic Hodge theory.  For the first we use one of the main innovations of \cite{finiteness}, namely the use of Tannakian categories of perverse sheaves to compute monodromy groups.  Given a hypersurface in an algebraic group, one can consider its corresponding perverse sheaf and define a \textit{convolution monodromy group} (or \textit{Tannakian monodromy group}), which can then be related to its geometric monodromy group.  This is an invariant of great independent interest, as we will now explain. 

Building off work of Gabber-Loeser \cite{Gabber}, Katz \cite{KatzCE}, and Kramer-Weissauer \cite{Kramer}, the work of Forey, Fres\'an, and Kowalski \cite{FFK} constructs the Tannakian category of perverse sheaves on an algebraic group defined either over a finite field or its algebraic closure.  Roughly speaking, after taking the quotient by an appropriate category of negligible objects, the category of perverse sheaves is given the structure of a Tannakian category under sheaf convolution.  In the case of tori, much of the theory was initially developed in \cite{Gabber}, and found deep number-theoretic applications in \cite{KatzCE}.  In addition to developing methods for computing specific monodromy groups, this work of Katz related exponential sums to the trace functions of perverse sheaves on $\G_m$, thus computing their distribution in terms of appropriate convolution monodromy groups.  In a different direction, the work of Kramer-Weissauer \cite{Kramer} focused on the case of abelian varieties, and used the convolution mondromy group to prove various vanishing theorems.  
This theory was applied in \cite{finiteness} to compute convolution monodromy groups on abelian varieties, which were then used to show that certain geometric monodromy groups are large.  In the present work, we begin by building off of the work of Katz, generalizing some results found in \cite{KatzCE} to higher dimensional tori $\G_m^n$ and computing several explicit examples of convolution monodromy groups. 
  We use these in a similar way to \cite{finiteness}, along with a comparison theorem allowing us to go from geometric monodromy groups in characteristic $p$ to those in characteristic 0, to prove large geometric monodromy results in characteristic 0.  

These kinds of large monodromy results concern the cohomology objects $R^if_*(\Q_p)$ associated to a morphism $f\colon Y\rightarrow X$ over some number field, and are eventually used to give bounds on the rational points $X(K)$ or the integral points of a model $\mc{X}(\O_{K, S})$.  This process was initiated in \cite{LV} in their proof of Faltings's theorem.  They studied the Kodaira-Parshin family, which can be constructed as a proper family over a curve of genus at least 2, to show that the number of rational points on the base is finite.  However, their method applies to smooth proper morphisms $f\colon Y\rightarrow X$ over $\Z[S^{-1}]$ more generally, though the results they obtained only give Zariski non-density of the integral points of the base.  By taking $f$ to be the universal family for some moduli problem, this gives a new approach to the Shafarevich conjecture.  Any point of the base $x\in X$ determines a Galois representation $H_{\et}^*(Y_x, \Q_p)$ of the Galois group $G_K$, which by a result of Faltings can take one of only finitely many isomorphism classes after semisimplification.  On the other hand, the variation of these representations can be measured by passing to the associated $p$-adic representations of $G_{K_v}$ for an appropriate place $v$.  By results of $p$-adic Hodge theory, these are crystalline representations which are endowed with a Frobenius action and a Hodge filtration.  The variation of the Hodge structure is measured by a $p$-adic period map, which can then finally be related to the monodromy group of $R^if_*(\Q_p)$.  To make this last step precise, one needs to wrestle with several technical difficulties, including transcendence results and numerical conditions which may be rather complicated.  We recommend the introduction to \cite{LV} for a more detailed overview. 

This strategy toward the Shafarevich conjecture was extended in \cite{finiteness} by the use of the Tannakian method to compute monodromy groups in the case of hypersurfaces in a fixed abelian variety $A$.  In this work, the authors were able to achieve finiteness, rather than simply Zariski non-density, of the integral points of the base.  To achieve this, they needed to prove a large monodromy statement for all subvarieties of the moduli space of hypersurfaces in $A$.  They did this by considering the local systems of the form $R^if_*(\L_{\chi})$ for various characters $\chi$ of the fundamental group of $A$, rather than just the constant one.  As a whole these give more chances to achieve big monodromy results, though the technical details become more complicated.  They dealt with them by packaging all the cohomological information of these local systems into ``Hodge-Deligne systems," which can be thought of motives with various realizations given by different cohomology theories and compatibility relations between them.  This is also the language we will work with in this paper.    

\subsection{Results} 
We will now highlight and discuss the results obtained in this paper.

Consider the Tannakian category of perverse sheaves modulo negligible sheaves on the torus $\G_m^n$ over the algebraic closure of a finite field.  Our first main result is the determination of an explicit fiber functor for this category.  Let $p_i\colon \G_m^i\times \A^1\rightarrow \G_m^i$ denote the projections and let $j_{0!}\colon \G_m^{i+1}\rightarrow \G_m^i\times \A^1$ denote the extension by 0 morphisms.

\begin{customthm}{\ref{thm: ff}}
The functor   
\[
\omega(N) = Rp_{1*}j_{10!}Rp_{2*}j_{20!} \cdots Rp_{n*}j_{n0!}N 
\] 
is a fiber functor on $\bar{\bP}((\G_m^n)_{\bar{k}})$.   
\end{customthm} 

This extends the result of Deligne in \cite[Chapter 3]{KatzCE}, which is the case of $n=1$.  Having an explicit description of the fiber functor associated to an algebraic group is useful for proving a number of results about the corresponding Tannakian categories.  In our case we will use the fact that the Frobenius element gives an automorphism of this fiber functor, and thus an element of the Tannakian group.  We can then compute its weights when the perverse sheaf is set to be the one corresponding to a hypersurface.  
We remark that if we worked with the extension by zero into copies of $\P^1$ instead of $\A^1$ in the definition of the functor, the functor would just be given by the compactly supported cohomology of the hypersurface.  In this case, the weights were computed in \cite{DL} (supposing nondegeneracy with respect to the Newton polyhedron, which we will also assume), extending the fundamental results of Deligne \cite{Weil2}.  

We compute the weights for curves in $\G_m^2$ in terms of the Newton polygon of the defining equation in two distinct ways, the results of which are given in Theorem \ref{thm: curve weights} and Proposition \ref{prop: curve weights 2}.  
The case of higher-dimensional varieties is more difficult.  In section \ref{sec: weight examples} we give some specific examples of surfaces which we compute the weights for.  Finally, in section \ref{sec: higher weights} we compute some of the weight multiplicities of certain higher-dimensional hypersurfaces. 

The purpose of computing these weights, other than their intrinsic interest, is to use them to show that the convolution monodromy groups of the associated hypersurfaces are large.  By combining our weight calculations with results from representation theory, we are able to conclude that in some scenarios the convolution monodromy group is large, which essentially means that it contains a classical group.  Specifically, these results can be found in \Cref{ex: curve}, \Cref{prop: surface}, and \Cref{cor: prism monodromy}.  This leads us to deduce the following relation between the convolution monodromy group, which $G^*$ refers to below, and the geometric monodromy group associated to various characters of the fundamental group.   

\begin{customcor}{\ref{cor: big}}
    Assume that $Y_{\overline{\eta}}$ is not translation-invariant by any nonzero element of $\G_m^n$, that $G^*$ is a simple algebraic group with irreducible distinguished representation, and that $Y$ is not equal to a constant family of hypersurfaces translated by a section of $\G_m^n$. 

    Then for any primes $l$ and $p$, there is an embedding $\iota\colon K\hookrightarrow \C$ and a torsion character $\chi$ of $\pi_1^{\et}((\G_m^n)_{\iota})$ of order a power of $l$ such that for all $\Gal(\bar{K}/K)$-conjugates $(\iota', \chi')$ of $(\iota, \chi)$, the geometric monodromy group of $R^{n-1}f_*(g^*\L_{(\iota', \chi')})$ contains a conjugate of $G^*$ as a normal subgroup. 
\end{customcor}

This result is akin to Corollary 4.10 of \cite{finiteness} which is stated for abelian varieties.  However, to apply it in our scenario, we not only have to use new arguments to deal with the fact that our hypersurfaces are not proper, but also make an appropriate comparison between monodromy groups in characteristic $p$ and characteristic 0, as we began in characterstic $p$.  These arguments comprise the rest of section \ref{sec: big mon}. 

Section \ref{sec: HD} follows, in which we construct Hodge-Deligne systems, which were first introduced in \cite{finiteness}.  Again extra care has to be taken to account for the non-properness of our hypersurfaces.  However, we manage by defining the cohomological realizations through appropriate eigenspaces of the cohomology of suitable compactifications of our hypersurfaces.  The result of this and the following section is Theorem \ref{thm: zariski non-density}, which gives the Zariski non-density of integral points on a base given certain big monodromy conditions and numerical conditions on the adjoint Hodge numbers of the constructed Hodge-Deligne systems.  This is an analogue of \cite[Theorem 8.17]{finiteness}, packaging together the technical analysis of \cite[Sections 9 -- 11]{LV} (which illustrates when Zariski non-density can be concluded) into these numerical conditions.  We remark that these numerical conditions take a slightly different form to those in \cite[Theorem 8.17]{finiteness}, which has the unfortunate result of making them difficult to satisfy unless the dimension of the hypersurface is fairly large. 

Section 9 is devoted to computing the adjoint Hodge numbers associated to hypersurfaces in tori.  This essentially boils down to computing the mixed Hodge numbers of character sheaves on these hypersurfaces, which we accomplish with the help of \cite{DK}.  Then we begin the process of checking the numerical conditions, in which we can use previous work of the author \cite{Ji} which dealt with a very similar scenario.  With the aim of finding a Newton polyhedron that not only satisfies these conditions but also satisfies a big monodromy result we can show, we prove the following result. 

\begin{customprop}{\ref{prop: asy weights}}
    Fix $n\ge 500000$ and let $\Delta$ be a rectangular prism with side lengths $a_1, a_2, \ldots, a_n$ with a corner with volume $o(\Vol(\Delta))$ cut off.  Then as $\min_i(a_i)\rightarrow \infty$, for sufficiently large $m$ and any choice of character, the Hodge-Deligne weights satisfy 
    \[
h^{p,q} \sim A(n, q)\Vol(\Delta). 
    \] 
    Furthermore, they satisfy the numerical conditions of Theorem \ref{thm: zariski non-density}. 
\end{customprop} 

This allows us to give the following application to the Shafarevich conjecture.  

\begin{customthm}{\ref{thm: finiteness}}
    Fix a Newton polyhedron $\Delta$ with big monodromy.  Let $S$ be a finite set of primes of a number field $K$.  Then assuming the numerical conditions are satisfied, up to translation there are only finitely many smooth primitive hypersurfaces in $(\G_m^n)_K$ nondegenerate with respect to their Newton polyhedron $\Delta$ with good reduction outside $S$. 
\end{customthm} 

In particular, the theorem holds for the hypersurfaces with Newton polyhedra described in \Cref{prop: asy weights}.  

In fact, we expect that this conclusion holds for a much wider class of Newton polyhedra than the ones described in \Cref{prop: asy weights}.  
For one thing, the numerical conditions can reasonably be expected to hold in fairly high generality when the dimension is sufficiently large, but as the adjoint Hodge numbers can be difficult to get a precise handle on, such statements may be tricky to prove.  
It is also possible that, with a more involved argument, one can show that weaker numerical conditions suffice as in \cite{finiteness}, which would allow the dimension to be taken to be smaller.  
Finally, one could hope to compute more monodromy groups of hypersurfaces, either by developing more general ways of computing weights of the Frobenius or otherwise, which would also lead to more instances of the Shafarevich conjecture.  On the other hand, aside from this application it is encouraging to note that many of the relevant geometric results proven in the case of abelian varieties also have some analogue for tori as shown in this paper.  
This raises the possibility of proving similar results for semiabelian varieties as a potential question for future work. 

\subsection{Relation to previous work} 
As stated earlier, most of the basic results we prove about the Tannakian category of perverse sheaves on $\G_m^n$ are motivated by analogous statements for the case $n=1$ found in \cite{KatzCE}.  Moreover, the idea of computing weights of the fiber functor and using them to determine the monodromy group can also be found in \cite{KatzCE}.  On the other hand, the tools (such as the generic vanishing theorem) provided by the general theory of \cite{FFK} proved to be important at various points in this paper.  In light of the relation of the fiber functor we construct to compactly supported cohomology, our weight calculations can also be seen as a continuation of \cite{DL}, which computed the weights of compactly supported cohomology of hypersurfaces in the torus.  

Regarding the application to the Shafarevich conjecture, this paper owes the most to the works \cite{LV} and \cite{finiteness}.  The work \cite{LV} provided the basic framework of using period mappings to approach versions of the Shafarevich conjecture.  Aside from \cite{finiteness}, this strategy has also been taken up in several other papers to prove sparsity of special points in moduli spaces, such as \cite{ELV} and \cite{chiu}.  The work \cite{chiu} is particularly relevant because it also considers hypersurfaces with a given Newton polyhedron that may not be projective.  However the methods and results of that paper are largely disjoint from that of the present work, as \cite{chiu} applies to a much more general class of Newton polyhedra but prove a subpolynomial rather than a constant bound.

On the other hand, \cite{finiteness} generalized the approach of \cite{LV} and incorporated Tannakian methods to compute monodromy groups which we use in this paper.  In particular, many of the results in the latter half of this paper are analogues for the torus of statements for abelian varieties in \cite{finiteness}, and the statement of the general Theorem \ref{thm: zariski non-density} from which the applications are deduced is inspired by the form of \cite[Theorem 10.1]{LV} and \cite[Theorem 8.17]{finiteness}.

\subsection*{Acknowledgments} 
The author would like to thank Will Sawin for suggesting this problem and teaching him much of the necessary background to understand and work on it.  His expertise also proved very useful on several technical issues the author encountered throughout this paper.  The author also thanks Nick Katz for interesting comments on this paper.  The author was partially supported by National Science Foundation Grant Number DGE-2036197.

\section{The Tannakian category of perverse sheaves on the torus}
Let $k = \F_q$ be a finite field with finite extensions $k_n = \F_{q^n}$ and algebraic closure $\overline{k}$.  In \cite{Gabber} and \cite{KatzCE}, the authors define and study a certain Tannakian category of perverse sheaves on $\G_m/\bar{k}$.  To be more precise, the Tannakian category is constructed as the quotient of the category of perverse sheaves by a certain subcategory of negligible sheaves, with sheaf convolution as the tensor functor.  This theory is extended to the case of all commutative algebraic groups in \cite{FFK}.  In this paper, we are primarily interested in the case of higher dimensional tori $G = \G_m^n$. 

We will first review the basic theory from the foundational papers referred to above. 
Our main new contribution in this section is the construction of an explicit fiber functor on this Tannakian category, which we will study further in the following sections.    

\subsection{Sheaf convolution} 
Let $G$ be a connected commutative algebraic group over $k$ with multiplication map $m$.  For a prime $l\neq \ch k$, let $D^b_c(G_k)\coloneqq D^b_c(G_k, \bar{\Q_l})$ denote the bounded derived category of constructible $\bar{\Q_l}$-sheaves on $G_k$, and similarly for $D^b_c(G_{\bar{k}})\coloneqq D^b_c(G_{\bar{k}}, \bar{\Q_l})$.  In both cases, for $\mc{F}, \mc{G}\in D^b_c(G)\coloneqq D^b_c(G, \bar{\Q_l})$ we define the convolution products 
\[
\mc{F}*_* \mc{G}\coloneqq Rm_*(\mc{F}\boxtimes \mc{G}), \mc{F}*_! \mc{G}\coloneqq Rm_!(\mc{F}\boxtimes \mc{G})
\]
where $m$ is the multiplication on $G$.  Furthermore, we define middle convolution as follows: 
\[
\mc{F}*_{\mi} \mc{G}\coloneqq \im[\mc{F}*_!\mc{G}\rightarrow \mc{F}*_* \mc{G}].
\]

The category of perverse sheaves on $G$ is not stable under middle convolution.  However, one can define a subcategory of negligible sheaves so that the quotient category is stable.  In the case of $\G_m/\bar{k}$, in \cite{KatzCE} they are defined as the subcategory of perverse sheaves $N$ such that $\chi((\G_m)_{\bar{k}}, N)=0$. 
 To give a uniform definition for all $G$ we will follow \cite{FFK} and introduce the concept of a generic character. 

 \begin{definition}
      [generic character]
     \cite[Definition 1.21]{FFK}
     Let $G$ be a connected commutative algebraic group
of dimension $d$ over a finite field $\F_q$. Let $S$ be a subset of $\widehat{G}$, where $\widehat{G}\coloneqq \bigsqcup_n \widehat{G}(\F_{q^n})$ and $\widehat{G}(\F_{q^n})$ denotes the group of characters $\chi\colon G(\F_{q^n})\rightarrow \Qlb^*$.  Then $S$ is a  \textit{generic} set of characters if 
\[
|\widehat{G}(\F_{q^n})-S(\F_{q^n})| \ll q^{n(d-1)}
\]
for all integers $n\ge 1$. 
 \end{definition}

Recall that a character $\chi\in \widehat{G}(k)$ gives rise to a character sheaf $\L_\chi$ on $G$ as follows.  The Lang isogeny $L\colon G\rightarrow G$ defined by $L(x)=x^{-1}\Fr_k(x)$ is a Galois \'etale covering corresponding to a map $\pi_1^{\et}(G)\rightarrow G(k)$, and composing with $\chi^{-1}$ gives a character of $\pi_1^{\et}(G)$ which defines a character sheaf $\L_\chi$.  Note that the geometric Frobenius acts on the stalk of $\L_{\chi}$ at $x$ by $\chi(x)$, and thus $\L_{\chi}$ is a pure lisse sheaf of weight 0.  

Given any complex $M\in D^b_c(G_k, \bar{\Q_l})$, let $M_{\chi}=M\otimes \L_{\chi}$. 

We will largely be concerned with the case in which $k$ is a finite field, which we will assume for the rest of this section.  In that case we can consider both perverse sheaves on $G_k$ and on $G_{\bar{k}}$. 

\begin{definition}
    \cite[Section 3.2]{FFK} Let $k$ be a finite field.  Let $\bP(G)$ denote the full subcategory of sheaves on $G_{\bar{k}}$ whose objects are defined over some finite extension of the base field $k$.
\end{definition} 

\begin{remark}
    In \cite{KatzCE}, one considers the entire category of perverse sheaves on $(\G_m)_{\bar{k}}$.  For the purposes of this paper, the difference between this and $\bP(G)$ is not significant, because the explicit perverse sheaves we consider all arise algebraically, and the results we use from \cite{KatzCE} in reference to the bigger category hold for the smaller category as well.  
\end{remark}

 \begin{definition}
     [negligible sheaves] \cite[Definition 3.4]{FFK} 
     An object $M$ of $\bP(G)$ is \textit{negligible} if the set of characters $\chi\in \widehat{G}$ such that $H^0(G, M_\chi)=0$ is generic.  An object $N$ of $D^b_c(G)$ is \textit{negligible} if all its perverse cohomology sheaves are negligible.
 \end{definition} 

 \begin{remark}
 \label{remark: neg}
     If $G$ is a semiabelian variety, \cite[Proposition 3.21]{FFK} allows us to give an alternative characterization of negligible sheaves which agrees with that used by Katz \cite{KatzCE} in the case of $\G_m$.  The proposition states that if $M$ is a perverse sheaf on any $G$, then $M$ is negligible if and only if $\chi(G_{\kbar}, M_{\chi})=0$ for a generic set of $\chi$.  Furthermore, if $G$ is semiabelian, then $\chi(G_{\kbar}, M_{\chi})$ is independent of $\chi$.  Thus for such $G$, $M$ is negligible if and only if its Euler characteristic vanishes. 
 \end{remark}

 We let $\bN(G)\subset \bP(G)$ and $\bN_D(G)\subset D_b^c(G)$ be the subcategories of negligible sheaves.  We can now use middle convolution to define a tensor structure on the quotient categories $\bar{\bD}(G)\coloneqq D_b^c(G)/\bN_D(G)$ and on $\bar\bP(G)\coloneqq \bP(G)/\bN(G)$.  In fact the situation is even nicer, as by \cite[Lemma 3.8]{FFK}, middle convolution coincides with both operations $-*_!-$ and $-*_*-$ on the quotient.  The upshot is the following theorem.

 \begin{theorem}
     \cite[Theorem 3.10, 3.15]{FFK} The canonical forget support morphisms $M*_! N \rightarrow M *_* N$ induce isomorphisms in $\bar{\bD}(G)$,
and define by passing to the quotient a convolution bifunctor denoted 
\[
*\colon \bar{\bD}(G)\times \bar{\bD}(G)\rightarrow \bar{\bD}(G).
\]
Furthermore, this descends to a bifunctor on the subcategory $\bar{\bP}(G)$ that makes it into a neutral Tannakian category.  
 \end{theorem}

Finally, given an element $M\in \bar{\bP}(G)$, we let $\langle M\rangle$ denote the Tannakian subcategory of $\bar{\bP}(G)$ consisting of all subquotients of convolution powers of $M\oplus M^\vee$.  Then $\langle M\rangle$ is a Tannakian subcategory which in various cases will be useful to study. 

 \subsection{Generic vanishing theorems} 
The generic vanishing theorem is a key result about generic characters that we will use several times in this paper.  We recall its general statement, as well as a more specific version for tori. 

\begin{theorem}
\label{thm: generic}
    [Generic vanishing] \cite[Theorem 2.1]{FFK} 
     Let $G$ be a connected commutative algebraic group over a finite field $k$
and let $M$ be a perverse sheaf on $G$. The set of characters $\chi\in\widehat{G}$ satisfying 
\begin{align*}
    & H^i(G_{\bar{k}}, M_{\chi}) = H^i_c(G_{\bar{k}}, M_{\chi}) = 0 \quad \text{for all } i\neq 0 \\ 
    & H^0(G_{\bar{k}}, M_{\chi}) = H^0_c(G_{\bar{k}}, M_{\chi}) 
\end{align*}
is generic.
\end{theorem}

 We abbreviate translates of algebraic cotorus by \textit{tac}.  

\begin{theorem}
\label{thm: generic vanishing}
    [Corollary 2.13 of \cite{FFK}] Let $T$ be a torus of dimension $d$ over $k$.  Let $M\in\mP(T)$.  For $-d\le i\le d$, the set
    \[
\{\chi\in \widehat{T}| H^i(T_{\kbar}, M_\chi)\neq 0\}
    \] 
    is contained in a finite union of tacs of $T$ of dimension $\le d-|i|$. 
\end{theorem}

\subsection{A fiber functor} 
Our goal now is to construct an explicit fiber functor on $\bar{\bP}((\G_m^n)_{\bar{k}})$, which for convenience we will abbreviate by $\bar{\bP}$.  In the case of $n=1$, there already is one provided by a result of Deligne as follows.  Let $j_0\colon \G_m\rightarrow \A^1$ be the inclusion.  Then the functor 
\[
M\mapsto H^0(\A^1, j_{0!}M)
\]
is a fiber functor on $\bar{\bP}((\G_m)_{\bar{k}})$ \cite[Theorem 3.1]{KatzCE}.  

For the higher dimensional case, we begin by defining the inclusion maps 

\[
j_{k0}, j_{k\infty}: \G_m^n \rightarrow \G_m^{k-1} \times \A^1 \times \G_m^{n-k}
\]

where for $j_{k0}$, the factor of $\A^1$ is viewed as $\P^1\backslash\infty$ while for $j_{k\infty}$, the factor of $\A^1$ is viewed as $\P^1\backslash 0$.  By a slight abuse of notation we will use the notation for the same inclusion where the domain is slightly altered, specifically to denote the maps 
\[
j_{k0}, j_{k\infty}: \G_m^k \rightarrow \G_m^{k-1}\times \A^1 \quad \text{and} \quad \G_m^{k-1} \times \A^1 \rightarrow \G_m^{k-1} \times \P^1. 
\]

Again by a slight abuse of notation, denote the following projections 
\[
p_k: \G_m^{k-1}\times \A^1 \rightarrow \G_m^{k-1}, \quad \G_m^{k-1}\times \P^1 \rightarrow \G_m^{k-1}. 
\]

\begin{theorem}
\label{thm: ff}
The functor   
\[
\omega(N) = Rp_{1*}j_{10!}Rp_{2*}j_{20!} \cdots Rp_{n*}j_{n0!}N 
\] 
is a fiber functor on $\bar{\bP}((\G_m^n)_{\bar{k}})$.    
\end{theorem}

\begin{remark}
This definition makes sense a posteriori by Lemma \ref{lemma: vanishing}, which implies that the complex $\omega(N)$ vanishes in all nonzero degrees.  This also shows that in the case of $n=1$, $\omega(-)$ coincides with the fiber functor of Deligne referred to above.
\end{remark} 

The proof of Theorem \ref{thm: ff} occupies the remainder of the section.  The three things that need to be proved is that $\omega(-)$ is faithful (Proposition \ref{prop: faithful}), exact (Proposition \ref{prop: exact}), and a tensor functor (Proposition \ref{prop: tensor}). 

\begin{remark}
    The same methods can likely be used to show that $Rp_{n*}j_{n0!}$ is a faithful exact tensor functor from $\bar{\bP}((\G_m^n)_{\bar{k}})$ to $\bar{\bP}((\G_m^{n-1})_{\bar{k}})$. This suggests the question of whether one can use this method as an intermediate step towards constructing fiber functors for other algebraic groups, such as semiabelian varieties.  (This remark arose from a comment by Nick Katz.)
\end{remark}

\subsubsection{Exactness and faithfulness} 
We recall the following lemma from \cite{KatzCE}.

\begin{lemma}
\label{lemma: katz dim}
[Lemma 3.4 of \cite{KatzCE}] For any perverse sheaf $N$ on $\G_m$, the groups $H^i(\A^1, j_{0!}N)$ vanish for $i\neq 0$, and 
\[
\dim H^0(\A^1, j_{0!}N) = \chi(\G_m, N)= \chi_c(\G_m, N). 
\]
\end{lemma} 

In \cite{KatzCE} this was proven using the classification of perverse sheaves on $\A^1$.  Here we prove a generalization of this lemma to $\G_m^n$ using a different method. 

\begin{lemma}
\label{lemma: vanishing}
For any perverse sheaf $N$ on $\G_m^n$, the groups $G_i\coloneqq H^i(\A^1, j_{10!}Rp_{2*}j_{20!} \cdots Rp_{n*}j_{n0!}N )$ vanish for $i\neq 0$, and 
\[
\dim H^0(\A^1, j_{10!}Rp_{2*}j_{20!} \cdots Rp_{n*}j_{n0!}N ) = \chi(\G_m^n, N)= \chi_c(\G_m^n, N). 
\]
\end{lemma} 

\begin{proof} 
The presheaf given by extending $N$ by 0 at 0 on one component and then taking the pushforward at $\infty$ coincides with that obtained by these operations in the opposite order.  Taking the associated sheaf and deriving, we see that $Rj_{i\infty*}j_{i0!}N = j_{i0!}Rj_{i\infty*}N$ as sheaves on $\P^1$.  We may write 
\begin{align}
    G_i &= H^i(\P^1, Rj_{1\infty *}j_{10!}Rp_{2*}Rj_{2\infty *}j_{20!}\cdots Rp_{n*}Rj_{n\infty *}j_{n0!}N) \label{eq1} \\ 
    &= H^i(\P^1, j_{10!}Rj_{1\infty *}Rp_{2*}j_{20!}Rj_{2\infty *}\cdots j_{n0!} Rp_{n*}Rj_{n\infty *}N). 
\end{align}  

Note that here, $p_k$ is affine.  Then by \cite[Theorem 4.1.1, Corollary 4.1.3]{BBD}, we have that $Rp_{k*}$, $Rj_{k\infty *}$, and $j_{k0!}$ are right t-exact for the perverse t-structure, so $G_i\in {}^pD^{\le 0}$.  By definition this implies that $G_i$ is 0 for positive degrees.  It remains to show that $G_i$ is 0 for negative degrees, which we will reduce to this case using Verdier duality. 

By Verdier duality, the second expression of $G_i$ is equal to 
\begin{align}
    G_i &= H^{-i}\big(\P^1, D(j_{10!}Rj_{1\infty *}Rp_{2*}j_{20!}Rj_{2\infty *}\cdots Rp_{n*} j_{n0!} Rj_{n\infty *}N)\big) \\ 
    &= H^{-i}(\P^1, Rj_{10*}j_{1\infty !}Rp_{2*}Rj_{20*}j_{2\infty !}\cdots Rp_{n*} Rj_{n0*}j_{n\infty !}DN). 
\end{align}   

The last equality holds because $D\circ Rp_{k*}=Rp_{k!}\circ D = Rp_{k*}\circ D$, as here $p_{k}$ is proper.  As this form is the same as that of (\ref{eq1}), we see that it vanishes for $-i >0$, which means that $G_i$ vanishes for negative degrees as desired. \\ 

Next, we need to prove the equality of the dimension of $G_0$ with the Euler characteristic and compactly supported Euler characteristic of $N$.  As in the proof of Lemma 3.4 of \cite{KatzCE}, we use the result of Laumon \cite{Lau} that says that $\chi$ and $\chi_c$ of a $\Q_l$-sheaf on a separated variety over an algebraically closed field.  Thus we have the chain of equalities 
\begin{align*}
&\chi(\G_m^n, N) = \chi_c(\G_m^n, N) = \chi_c(\G_m^{n-1}\times \A^1, j_{n0!}N) = \chi(\G_m^{n-1}\times \A^1, j_{n0!}N) \\ 
= ~&\chi(\G_m^{n-1}, Rp_{n*}j_{n0!}N) = \chi_c(\G_m^{n-1}, Rp_{n*}j_{n0!}N)) = \cdots \\ 
= ~ &\cdots \\
= ~& \chi(\G_m, Rp_{2*}j_{20!} \cdots Rp_{n*}j_{n0!}N ). 
\end{align*} 
This last quantity is equal to 
\[
\chi_c(\G_m, Rp_{2*}j_{20!} \cdots Rp_{n*}j_{n0!}N ) = \chi_c(\A^1, j_{10!}Rp_{2*}j_{20!} \cdots Rp_{n*}j_{n0!}N ), 
\]
which by the vanishing of $G_i$ for $i\neq 0$ is equal to $\dim H^0(\A^1, j_{10!}Rp_{2*}j_{20!} \cdots Rp_{n*}j_{n0!}N )$, as desired. 
\end{proof}

\begin{proposition}
\label{prop: faithful}
The functor $\omega$ is faithful.
\end{proposition} 
\begin{proof}
Recall that by \ref{remark: neg}, a perverse sheaf is negligible if and only if its Euler characteristic vanishes.  Then by Lemma \ref{lemma: vanishing}, we have that $\omega$ applied to a perverse sheaf vanishes precisely if the perverse sheaf is negligible.  This implies the faithfulness of $\omega$ as a functor from $\bar{\bP}$ to vector spaces.  
\end{proof}

\begin{proposition}
\label{prop: exact}
The functor $\omega$ is exact.
\end{proposition} 
\begin{proof}
Recall that an exact sequence in $\bar{\bP}$ is given by a sequence in $\bP((\G_m^n)_{\bar{k}}$ 
\[
0\rightarrow A\xrightarrow{\alpha}B\xrightarrow{\beta}C\rightarrow 0 
\] 
where $\alpha$ is injective, $\beta$ is surjective, $\beta\circ\alpha = 0$, and $\ker\beta/\im\alpha$ is negligible.  We may break this up into the two exact sequences 
\begin{align*}
    &0\rightarrow A\xrightarrow{\alpha} B \rightarrow \coker\alpha \rightarrow 0 \\ 
    & 0 \rightarrow C \rightarrow \coker \alpha \rightarrow \ker\beta/\im\alpha \rightarrow 0.
\end{align*} 
Now the functor sending $N$ to $j_{10!}Rp_{2*}j_{20!} \cdots Rp_{n*}j_{n0!}N$ preserves distinguished  triangles, so we obtain long exact sequences of cohomology corresponding to the two exact sequences (viewed as distinguished triangles in the derived category) above.  By Lemma~\ref{lemma: vanishing} these are actually short exact sequences, and since $\ker\beta/\im\alpha$ is negligible we immediately obtain the desired short exact sequence 
\[
0\rightarrow \omega(A)\rightarrow \omega(B)\rightarrow \omega(C)\rightarrow 0, 
\] 
as desired. 
\end{proof}

\subsubsection{Tensor functor}
In this section, $K, M, N$ will denote objects in $\bar{\bP}$.  We need to construct a bifunctorial isomorphism $\omega(K\cm N) \cong \omega(K)\otimes \omega(N)$.  We do so in a similar way to the construction for the case $n=1$ in \cite{KatzCE}, Chapter 30.  First we verify that the dimensions of these two vector spaces are the same; then we construct maps between them that compose to the identity. \\

\begin{lemma}
The natural maps 
\[
K*_!N\rightarrow K\cm N\rightarrow K*_*N
\] 
induce isomorphisms 
\[
\omega(K*_!N)\rightarrow \omega(K\cm N)\rightarrow \omega(K*_*N). 
\]
\end{lemma} 
\begin{proof}
Lemma 3.8 from \cite{FFK} ensures that the cone of the natural morphism $K*_!N\rightarrow K*_*N$ lies in $\bN$.  Then by Lemma~\ref{lemma: vanishing}, the long exact sequence of cohomology gives the desired isomorphisms.  
\end{proof} 

This implies the desired equality of dimensions as follows. 
\begin{lemma}
We have 
\[
\dim(\omega(K)\otimes \omega(N)) = \dim\omega(K*_{\mi} N). 
\]
\end{lemma} 
\begin{proof}
By Lemma~\ref{lemma: vanishing} we just need to show that $\chi(\G_m^n, K)\times \chi(\G_m^n, N)=\chi(\G_m^n, K*_{\mi}N)$.  But we have 
\[
\chi(\G_m^n, K*_{\mi}N)=\chi(\G_m^n, K*_{*}N)=\chi(\G_m^n, Rm_*(K\boxtimes N)) = \chi(\G_m\times \G_m, K\boxtimes N), 
\] 
which equals $\chi(\G_m^n, K)\times \chi(\G_m^n, N)$ by the K\"unneth formula.
\end{proof} 

\begin{proposition}
    \label{prop: tensor}
    The functor $\omega$ is a tensor functor.
\end{proposition}
\begin{proof}
     We would like to obtain bifunctorial maps 
\[
\omega(K)\otimes \omega(N) \rightarrow \omega(K*_!N)\cong \omega(K\cm N)\cong \omega(K*_* N)\rightarrow \omega(K)\otimes\omega(N)
\] 
which compose to the identity; by the previous lemmas, dimension considerations will give the desired isomorphism. 
  
Set 
\[
\mc{K} = Rj_{1\infty*}j_{10!}Rp_{2*}Rj_{20!}Rp_{3*}Rj_{30!}\cdots Rp_{n*}Rj_{n0!}K 
\]
and similar for $\mc{N}$, so that by the K\"unneth formula we have 
\[
\omega(K)\otimes \omega(N) = H^*(\P^1\times \P^1, \mc{K}\boxtimes\mc{N}). 
\]

The rest of the argument for this proceeds exactly as in Chapter 30 of \cite{KatzCE}, where the result is proven for $n=1$.
 Indeed, the only properties of $\mc{K}$ and $\mc{N}$ used in the remainder of the proof are that they both vanish at 0 and that $D\mc{K}$ and $D\mc{N}$ both vanish at $\infty$.  These still hold when $n>1$, so the proof carries through as desired.  (In fact, if we replace the perverse sheaves $K, N$ on $\G_m$ used in their proof with  $Rp_{2*}j_{20!} \cdots Rp_{n*}j_{n0!}K, Rp_{2*}j_{20!} \cdots Rp_{n*}j_{n0!}N$ in our setting, the proof carries through word-for-word.)
\end{proof}

\subsection{Basic properties of the fiber functor}
We will now record a few basic properties of the fiber functor we have constructed, beginning with a computation of its dimension.

\begin{proposition}
\label{prop: dim}
Let $i\colon X\rightarrow \G_m^n$ be the inclusion of a hypersurface nondegenerate with respect to its Newton polyhedron $\Delta$, working over any algebraically closed field $\bar{K}$.  Then 
\[
\omega(i_*\Q_l[n-1])=n!\Vol(\Delta(C)).\]
\end{proposition} 

This follows from the following theorem from \cite{DL}.
\begin{theorem} 
[Theorem 2.7 of \cite{DL}] 
Let $G(x_1, \ldots, x_n)$ be a Laurent polynomial over a field $K$ which is 0-nondegenerate with respect to $\Delta(G)$.  Then 
\[
\chi((\G_m)^n_{\overline{K}}\cap G^{-1}(0), \Q_l) = (-1)^{n-1}n!\Vol(\Delta(G))
\]
where $\chi$ denotes the Euler characteristic with respect to $l$-adic cohomology ($l\neq \ch K$). 
\end{theorem} 

\begin{proof}
[Proof of Proposition \ref{prop: dim}]  The previous theorem gives us that $\chi(\G_m^2, i_*\Q_l) = (-1)^{n-1}n!\Vol(\Delta(C))$.  Then by Lemma \ref{lemma: vanishing}, we have that 
\[
\dim (\omega(i_*\Q_l[n-1])) = (-1)^{n-1}\chi(\G_m^n, i_*\Q_l[1]) = n!\Vol(\Delta(C)).
\]
\end{proof}

It will be useful at various points of the paper to analyze the fiber functor in a few different ways.  The following statement gives an alternate expression for the fiber functor, which is valid by proper base change.  

\begin{proposition}
    \label{prop: ff 1} 
    There is an isomorphism 
    \[
\omega(N) \cong R\Gamma((\P^1)^n, Rj_{1\infty*}j_{10!}Rj_{2\infty*}j_{20!}\cdots Rj_{n\infty*}Rj_{n0!}N).
    \]
\end{proposition}
For example, in the case $n=2$ we have the following diagram
\[\begin{tikzcd}
	C & {\G_m^2} & {\G_m\times \A^1} & {\G_m\times \P^1} & {\A^1\times \P^1} & {\P^1\times \P^1} \\
	&&&&& {\P^1} \\
	&&&&& {*}
	\arrow["i", from=1-1, to=1-2]
	\arrow["{j_{20}}", from=1-2, to=1-3]
	\arrow["{j_{2\infty}}", from=1-3, to=1-4]
	\arrow["{j_{10}}", from=1-4, to=1-5]
	\arrow["{j_{1\infty}}", from=1-5, to=1-6]
	\arrow["{q_2}", from=1-6, to=2-6]
	\arrow["{q_1}", from=2-6, to=3-6]
\end{tikzcd}\]
and the isomorphism 
\[
\omega(i_*\Q_l[1]) = \mc{H}^1(Rp_{1*}j_{10!}Rp_{2*}j_{20!}i_*\Q_l) \cong
\mc{H}^1(Rq_{1*}Rq_{2*}Rj_{1\infty*}j_{10!}Rj_{2\infty*}j_{20!}i_*\Q_l), 
\]
with the other $\mc{H}^i$ terms being 0.

\section{Weights of the fiber functor for curves}  
Recall that the elements of a Tannakian group give automorphisms of any of its fiber functors.  
Since the Frobenius acts on the fiber functor constructed in the previous section, it gives an element of the Tannakian group.  
We may restrict to a particular representation and compute the eigenvalues of the Frobenius on it.  
As we will see in a future section, this places restrictions on the Tannakian subcategory associated to that representation which we can use to compute the associated Tannakian group. 

The representations we are interested in come from the perverse sheaves defined by a translate of the constant sheaf on some subvariety of the torus.  
In this section we will tackle the case of $n=2$, which corresponds to curves in the two dimensional torus.  A curve $i\colon C\hookrightarrow \G_m^2$ is given by an irreducible polynomial $f\in \bar{k}[x, y]$ with $f\neq x, y$. 
 Consider the perverse sheaf $P=i_*\Q_l[1]$and the maps $C\xhookrightarrow{i}\G_m\times \G_m\xrightarrow{j_2}\G_m\times\A^1\xrightarrow{p_2}\G_m$; then the fiber functor is given by $K\coloneqq \omega(P) = Rp_{2*}j_{2!}i_*\Q_l[1]$. 
 Say $C$ is defined by the polynomial $f$.  We will assume that $C$ is non-degenerate with respect to its Newton polygon.  We will also assume that the $p=\ch k$ is greater than the degree of $f$.

\subsection{Relating the weights of the fiber functor to those of a simpler object}
The fiber functor applied to the perverse sheaf $i_*\Q_l[1]$ is given by the cohomology group 
\[
H^0(\A^1, j_{0!} Rp_{2*}j_{1!}i_*\Q_l[1]) = H^1(\A^1, j_{0!} Rp_{2*}j_{1!}i_*\Q_l).
\] 
To compute the weights of the Frobenius on it, we will first reduce it to a simpler object. 

\begin{lemma}
\label{lem: reduce 1}
    The cohomology group $H^1(\A^1, j_{0!} Rp_{2*}j_{1!}i_*\Q_l)$ is an extension of the two components $H^1(\A^1, j_{0!} p_{2*}j_{1!}i_*\Q_l)$ and $H^0(\A^1, j_{0!} R^1p_{2*}j_{1!}i_*\Q_l)$. 
\end{lemma}  
\begin{proof}
We have a distinguished triangle
\[
p_{2*}j_{1!}i_*\Q_l \rightarrow Rp_{2*}j_{1!}i_*\Q_l \rightarrow R^1p_{2*}j_{1!}i_*\Q_l[-1] \xrightarrow{[+1]} p_{2*}j_{1!}i_*\Q_l[1] 
\]
Because $j_{0!}$ is exact, it keeps it distinguished.  The corresponding long exact sequence gives 
\begin{align*}
&\cdots \rightarrow H^{-1}(\A^1, j_{0!}R^1p_{2*}j_{1!}i_*\Q_l) \\ 
&\rightarrow H^1(\A^1, j_{0!}p_{2*}j_{1!}i_*\Q_l)\rightarrow H^1(\A^1, j_{0!}Rp_{2*}j_{1!}i_*\Q_l) \rightarrow H^0(\A^1, j_{0!}R^1p_{2*}j_{1!}i_*\Q_l) \\  &\rightarrow H^2(\A^1, j_{0!}p_{2*}j_{1!}i_*\Q_l) \rightarrow\cdots 
\end{align*}
which just becomes the short exact sequence 
\[0\rightarrow H^1(\A^1, j_{0!}p_{2*}j_{1!}i_*\Q_l)\rightarrow H^1(\A^1, j_{0!}Rp_{2*}j_{1!}i_*\Q_l) \rightarrow H^0(\A^1, j_{0!}R^1p_{2*}j_{1!}i_*\Q_l) \rightarrow 0,\]
as desired. 
\end{proof}

We deal with $H^1(\A^1, j_{0!} p_{2*}j_{1!}i_*\Q_l)$ first.  Let $n_0$ be the number of solutions $t\in \G_m$ to $f(t, 0)=0$. 

\begin{proposition}
The weights of $H^1(\A^1, j_{0!}p_{2*}j_{1!}i_*\Q_l)$ consist of those of $H^1(\A^1, j_{0!}p_{2*}j_{1*}i_*\Q_l)$ along with $n_0(C)$ of weight 0. 
\end{proposition}  

\begin{proof} 
For $V$ \'etale over $\G_m\times \A^1$, the group $j_{1!}i_*\Q_l(V)$ consists of the sections $s$ whose support is proper over $V$.  This coincides with $j_{1*}i_*\Q_l(V)$ if and only if $V$ does not map onto a point of the form $(t, 0)\in C$.  Thus we see that the cokernel $\mc{G}$ of the natural morphism $p_{2*}j_{1!}i_*\Q_l\rightarrow p_{2*}j_{1*}i_*\Q_l$ is a direct sum of skyscraper sheaves over the points $t\in \G_m$ satisfying $f(t, 0)=0$.  At such a point $t$, the cokernel is given by the cokernel of 
\[
\bigoplus_{(t, y)\in C, y\neq 0} \Q_l\rightarrow \bigoplus_{(t, y)\in C} \Q_l,
\] 
so $\mc{G} = \bigoplus_{t|(t, 0)\in C} \Q_l$. 

The fiber functor is given by embedding these into $\A^1$ by $j_{0!}$ and taking the first cohomology group.  The associated long exact sequence gives 
\begin{align*}
    0&\rightarrow H^0(\A^1, j_{0!}p_{2*}j_{1!}i_*\Q_l) \rightarrow H^0(\A^1, j_{0!}p_{2*}j_{1*}i_*\Q_l)\rightarrow H^0(\A^1, j_{0!}\mc{G}) \\ 
    &\rightarrow H^1(\A^1, j_{0!}p_{2*}j_{1!}i_*\Q_l) \rightarrow H^1(\A^1, j_{0!}p_{2*}j_{1*}i_*\Q_l)\rightarrow 0. 
\end{align*}

We will prove in Proposition \ref{prop: H0} that $p_{2*}j_{1*}i_*\Q_l[1]$ is perverse.  By Lemma \ref{lemma: katz dim}, extending a perverse sheaf on $\G_m$ by 0 into $\A^1$ and taking the cohomology group of a given degree other than 0 gives 0.  Thus we have a short exact sequence
\begin{align*}
    0&\rightarrow  H^0(\A^1, j_{0!}\mc{G}) 
    \rightarrow H^1(\A^1, j_{0!}p_{2*}j_{1!}i_*\Q_l) \rightarrow H^1(\A^1, j_{0!}p_{2*}j_{1*}i_*\Q_l)\rightarrow 0. 
\end{align*} 

The weights of $H^0(\A^1, j_{0!}\mc{G})$ are all 0, so the proposition follows.
\end{proof}

 We now turn to dealing with $H^0(\A^1, j_{0!}R^1p_{2*}j_{1!}i_*\Q_l)$.  Let $n_{\infty}$ be the number of solutions $t\in \G_m$ to $f(t, \infty)=0$.  Let $j_{\infty}$ be the inclusion of $\G_m\times \A^1$ into $\G_m\times \P^1$ and $q$ the projection down to $\G_m$ as shown below.  
 
\[\begin{tikzcd}
	C & {\G_m^2} & {\G_m\times \A^1} & {\G_m\times \P^1} \\
	&& {\G_m}
	\arrow["i", from=1-1, to=1-2]
	\arrow["{j_1}", from=1-2, to=1-3]
	\arrow["{j_{\infty}}", from=1-3, to=1-4]
	\arrow["q", from=1-4, to=2-3]
	\arrow["{p_2}", from=1-3, to=2-3]
\end{tikzcd}\] 

\begin{proposition}
\label{prop: curve weights}
The weights of the fiber functor $H^0(\A^1, j_{0!} Rp_{2*}j_{1!}i_*\Q_l[1])$ consist of those of $H^1(\A^1, j_{0!}p_{2*}j_{1*}i_*\Q_l)$ along with $n_0$ of weight 0 and $n_{\infty}$ of weight 2. 
\end{proposition} 

\begin{proof}
By the discussion above, it suffices to compute the weights of $H^0(\A^1, j_{0!}R^1p_{2*}j_{1!}i_*\Q_l)$.  We can compute the stalk of $R^1p_{2*}j_{1!}i_*\Q_l$ at a point $t\in\G_m$ as follows.  By factoring $p_2=q\circ j_{\infty}$, the Grothendieck spectral sequence gives an exact sequence
\[
0\rightarrow R^1q_*(j_{\infty *}j_{1!}i_*\Q_l)\rightarrow R^1p_{2*}j_{1!}i_*\Q_l\rightarrow q_*R^1j_{\infty *}(j_{1!}i_*\Q_l)\rightarrow R^2q_*(j_{\infty *}j_{1!}i_*\Q_l) \rightarrow \cdots.
\]
Take a point $t\in \G_m$.  For $i=1, 2$, by the proper base change theorem we have 
\[
 (R^iq_*(j_{\infty *}j_{1!}i_*\Q_l))_t = H^i(\P^1, t^*j_{\infty *}j_{1!}i_*\Q_l) = 0, 
\]
because the higher cohomology of skyscraper sheaves vanish.  Thus we are only interested in the component $q_*R^1j_{\infty *}(j_{1!}i_*\Q_l)$.  Again by proper base change, the stalk of this sheaf at $t\in \G_m$ is given by $H^0(\P^1, t^*R^1j_{\infty *}(j_{1!}i_*\Q_l))$.  Note that $R^1j_{\infty*}(j_{1!}i_*\Q_l)$ is only supported on the points $(t, \infty)$ with $f(t, \infty)\in C$.  Fix such a point $z = (y, \infty)$.  Then as $U$ ranges over the \'etale neighborhoods of $z$, we have
\[
(R^1j_{\infty*}j_{1!}i_*\Q_l)_{z} = \varinjlim_{U} H^1(U \times_{\G_m\times \P^1} (\G_m\times \A^1), j_{1!}i_*\Q_l)  \cong \varinjlim_{U} H^1(U \times_{\G_m\times \P^1} C, \Q_l). 
\]

Passing the limit into $U\times C$, we get that $\varinjlim_{U}U\times C$ is given by $\Spec \Fqb[x_1, x_2]_{z}^h \times C$.  The \'etale local rings of the closed points of a smooth curve are all isomorphic to $\O_{\A^1, 0}^h$.  Here, $C$ is a smooth curve punctured at $z$. Thus, passing the limit inside, we have $\varinjlim_{U}H^1(U\backslash\{x\}, \Z/l) = H^1(\Spec \O_{\A^1, 0}^h\backslash\{0\}, \Z/l) = H^1(\Spec\operatorname{Fr}(\O_{\A^1, 0}^h), \Z/l)$.  The prime-to-$p$ part of $\Spec\operatorname{Fr}(\O_{\A^1, 0}^h)$ is given by $\prod_{p'\neq p}\Z_{p'}(1)$, so we see that this cohomology group is $\Z/l(-1)$.  Doing the same thing for $\Z/l^n$-coefficients and taking the limit, we get that $\varinjlim_{U}H^1(U\backslash\{x\}, \Q_l) = \Q_l(-1)$. \\ 

Since there are $n_\infty$ points of this form and $\Q_l(-1)$ has weight 2, the proposition follows from the previous discussion. 

\end{proof}

\subsection{Weights of the related sheaf}
We can use a result of Katz to compute the weights of $H^1(\A^1, j_{0!}p_{2*}j_{1*}i_*\Q_l)$.  Indeed, \cite[Theorem 16.1]{KatzCE}, stated as Theorem \ref{thm: katz-weights} below, gives a formula for the weights of $\omega(j_* \mc{F}[1])$ for $\mc{F}$ a lisse sheaf on some open $U\subset \G_m$.  The complete conditions are that $N = j_*\mc{F}[1]$ is perverse, arithmetically semisimple, and pure of weight 0.   We would like to show that  $(p_{2*}i_*\Q_l(1/2)/\Q_l)[1]$ is of this form.  By construction, it is arithmetically semisimple and pure of weight 0.  It remains to show that $p_{2*}i_*\Q_l$ is equivalent to the pushforward of the pullback over the open set in $\G_m$ which it is lisse over.  \\ 

We note that $p_{2*}\circ i_*:C\rightarrow\G_m$ is \'etale outside of the ramification points of $C$.  Let the unramified locus in the image be given by $j: U\hookrightarrow\G_m$ with complement $i:Z\hookrightarrow\G_m$, and let $\mc{F} = j^*(p_{2*}i_*\Q_l)$. 

\begin{proposition}
\label{prop: H0} 
We have $p_{2*}i_*\Q_l = j_*\mc{F}$. 
\end{proposition} 
\begin{proof}
We have an exact sequence of sheaves on $\G_m$: 
\[
0\rightarrow i_*i^!p_{2*}i_*\Q_l\rightarrow p_{2*}i_*\Q_l \rightarrow j_*j^*(p_{2*}i_*\Q_l) = j_*\mc{F}. 
\] 
Clearly $p_{2*}i_*\Q_l$ does not have any nonzero sections supported only on $Z$, so the first term is 0. The cokernel of the final map is a sheaf $\mc{H}$ which can only be supported on ramification points $x\in Z$.  At such a point, we have that 
\[
\mc{H}_x = \coker (\varinjlim_{V\ni x}\Gamma(V\times C, \Q_l)  \rightarrow \varinjlim_{V\ni x}\Gamma(V\backslash\{x\} \times C, \Q_l)).  
\] 

Note that $V\backslash\{x\}\times C$ is obtained by removing a finite number of points from $V\times C$, which does not change the set of its connected components in the Zariski topology.  This implies that the given map is surjective so $\mc{H}_x=0$, as desired. 
\end{proof}

Now take $N=j_*\mc{F}(1/2)[1]$.  This is a perverse sheaf on $\G_m/\F_q$ that is a constant summand $\Q_l[1]$ direct sum an arithmetically semisimple piece and pure of weight 0 (the half-Tate twist is included for this purpose).  Let $\mc{F}(0)$ and $\mc{F}(\infty)$ denote the inertia representations at 0 and $\infty$ attached to $\mc{F}$.  We look at the unipotent summand in the tame parts of these representations and obtain a Jordan decomposition: 
\[
\mc{F}(0)^{unip}=\bigoplus_{i=1}^{d_0}Unip(e_i), \quad \mc{F}(\infty)^{unip}=\bigoplus_{j=1}^{d_\infty}Unip(f_j).
\]

Now we can apply the following result from \cite{KatzCE}.
\begin{theorem}
\label{thm: katz-weights}
 [Theorem 16.1 of \cite{KatzCE}] For $N$ perverse, arithmetically semisimple, and of weight 0, the action of $Frob_{k, 1}$ on $\omega(N)=H^0(\A^1\otimes_k \overline{k}, j_{0!}N)$ has exactly $d_0$ eigenvalues of weight $<0$; their weights are $-e_1, \ldots, -e_{d_0}$.  It has exactly $d_{\infty}$ eigenvalues of weight $>0$; their weights are $f_1, \ldots, f_{d_{\infty}}$.  All other eigenvalues, if any, have weight zero.
\end{theorem}  

In our setting, we want the weights of $H^1(\A^1, j_{0!}p_{2*}j_{1*}i_*\Q_l)$.  To apply the previous theorem, we need to take out the constant summand $\Q_l[1]$ when computing the unipotent pieces.  We do not need to add anything back in because the constant summand does not contribute to the fiber functor.  Furthermore, because we are starting with $N$ of weight 1 (as it is placed in degree -1), we would have to apply the half-Tate twist to make it of degree 0.  Alternatively, we can just remember to add 1 to each of the weights.  Thus we will compute the unipotent pieces and subtract one $Unip(1)$ from both the $0$ and $\infty$ parts and apply the theorem with one added to all the weights to compute the desired weights.  

\subsection{Weights in terms of the Newton polygon}
In order to compute the unipotent pieces at 0 and $\infty$, we must determine the local monodromy action of a generator at both these points.  At 0, we know that $1\in I(0)/P(0)$ corresponds to $\sigma\in \Gal(\overline{\F_q(x)}/\F_q(x))$ that on each $\Fq(x^{1/m})/\Fq(x)$, sends $x^{1/m}$ to $\zeta_mx^{1/m}$ for $(m, p)=1$.  Thus we see that the action of 1 is given by a permutation matrix with an $m$-cycle for each solution to $f$ which solves $y$ as a polynomial in $x^{1/m}$.  At $\infty$, the same thing happens when we replace $x$ with $1/x$.  We can explicitly carry out these computations using the Newton polygon. \\ 

We say that $C$ is 0-nondegenerate with respect to its Newton polygon $\Delta(C)$ if there are no solutions to 
\[
\dfrac{\partial f}{\partial x} = \dfrac{\partial f}{\partial y} = 0 
\] 
on any of the faces of its Newton polygon, including $\Delta(C)$. \\ 

We may and do multiply by appropriate factors of $x$ and $y$ so that every exponent of $f$ is nonnegative and there is a term of the form $x^a$ and a term of the form $y^b$ (e.g. a nonzero constant term would satisfy both of these requirements).  Because we assume $C$ to be nondegenerate with respect to its Newton polygon, we can use the slopes of its Newton polygon to compute the ramification degrees.  We assume that the characteristic $p$ is large enough so that it does not divide any of the numbers appearing in the slopes of the Newton polygon.  For example, we can take $p>\deg f$.  Finally, we plot the Newton polygon by putting the $y$-exponents on the $x$-axis and the $x$-exponents on the $y$-axis (this gives the correct orientation of the Newton polygon when solving for $y$ in terms of $x$). \\ 

The main property of Newton polygons that we will use is the following.  Define the slope sequence of a Newton polygon to be the increasing sequence of slopes corresponding to the edges of the Newton polygon, beginning with the bottom vertex on the $y$-axis going to the bottom vertex with largest $x$-coordinate.  Define the inverse slope sequence to be the slope sequence obtained when the Newton polygon is reflected about the $x$-axis and shifted up appropriately.  Let $S_0(\Delta)$ and $S_\infty(\Delta)$ denote the sides of $\Delta$ in the slope sequence and inverse slope sequence respectively.  As in \cite{DL}, define the volume of a side to be its length normalized so that the intersection of a fundamental domain of its affine space with the lattice $\Z^2$ has unit length.  (This definition extends readily to higher dimensions.)  In particular, an edge represented by a vector of the form $(a, b)$ has volume $\gcd(a, b)$.  Finally, denote the length of a side $E_i$ along the $x$-axis (corresponding to the $y$-exponents) by $a(E_i)$.

\begin{proposition}
\label{prop: Newton general}
Assume that $C$ is 0-nondegenerate with respect to its Newton polygon and that $p>\deg f$.   Then the unipotent part of the local monodromy over 0 is given by $\sum_{E_i \in S_0(\Delta(C))}^k\Vol(E_i) - 1$ copies of $Unip(1)$, and over $\infty$ it is given by $\sum_{E_j \in S_{\infty}(\Delta(C))}^m \Vol(E_j) - 1$ copies of $Unip(1)$.   
\end{proposition}
\begin{proof}
We obtain the solutions to $y$ in terms of power series of $x^{1/n}$ using the method described above.  For each edge in $S_0(\Delta(C))$ whose endpoints form a vector $(a, b)$, we obtain $a$ distinct (by 0-nondegeneracy) solutions of $y$ in terms of $x^{\gcd(a, b)/a}$.  In matrix form, this gives $\gcd(a, b)$ permutation matrices of dimension $a/\gcd(a, b)\times a/\gcd(a, b)$ of the form 
$ \begin{bmatrix}
0 & 1 & 0 & \cdots & 0 \\
0 & 0 & 1 & \cdots & 0 \\
\cdots & \cdots & \ddots & \cdots & 0 \\
0 & 0 & 0 & \cdots & 1 \\
1 & 0 & 0 & \cdots & 0 
\end{bmatrix}$.  These are diagonalizable, with roots of unity on the diagonal.  We see that each of these contributes one $Unip(1)$, so this gives the result for the local monodromy over 0.  The case of the local monodromy over $\infty$ is similar.
\end{proof}

Putting together Proposition \ref{prop: Newton general}, Theorem \ref{thm: katz-weights}, Proposition \ref{prop: dim}, and the previous discussion, we arrive at the following conclusion.  

\begin{theorem}
\label{thm: curve weights}
Assume that $C$ is 0-nondegenerate with respect to its Newton polygon and that $p>\deg f$.  Let $w_i$ denote the multiplicity of $i$ in the weights of the fiber functor.  Then the numbers $w_i$ are given as follows: 
\begin{align}
   w_0 &= n_0 + \sum_{E_i \in S_0(\Delta(C))}^k\Vol(E_i) - 1, \\ 
    w_1 &= 2\Vol(\Delta(C)) - w_0 - w_2, \\ 
    w_2 &= n_{\infty}+\sum_{E_i \in S_{\infty}(\Delta(C))}^m \Vol(E_i)- 1. 
\end{align}
\end{theorem}

\begin{proof}
We apply Theorem \ref{thm: katz-weights} to $N$ such that $N$ direct sum the constant summand $\Q_l[1]$ is $j_{0!}p_{2*}j_{1*}i_*\Q_l[1]$.  Proposition \ref{prop: Newton general} gives us the values $e_i$ and $f_j$ appearing in the unipotent parts of the inertia representations at $0$ and $\infty$.  Namely, we have $\sum_{E_i \in S_0(\Delta(C))}^k\Vol(E_i)$ copies of $Unip(1)$ over 0, and $\sum_{E_j \in S_{\infty}(\Delta(C))}^m \Vol(E_j)$ copies of $Unip(1)$ over $\infty$.  These give $\sum_{E_i \in S_0(\Delta(C))}^k\Vol(E_i) - 1$ of weight 0 and $\sum_{E_j \in S_{\infty}(\Delta(C))}^m\Vol(E_j) - 1$ of weight 2, and some unspecified number of weight 1.  By Proposition \ref{prop: curve weights}, the weights of the fiber functor consist of these weights along with $n_0$ additional of weight 0 and $n_{\infty}$ additional of weight 2.  But we know that the total number of weights by Proposition \ref{prop: dim} is $2\Vol(\Delta(C))$.  This allows us to compute the number of weight 1, and the result follows.
\end{proof} 

\subsection{Examples}
We illustrate the results of this section with several examples.

\begin{example}
In this example we will explicitly carry out the computations described previously in the case of a specific curve.  Let $C$ be defined by the polynomial $f = x^4y^3 + 3x^2y^2 + x^2 + y$. \\  

\begin{tikzpicture}
\begin{axis}[
    axis lines=middle,
    xmin=-1, xmax=6,
    ymin=-1, ymax=6,
    xtick=\empty, ytick=\empty
]
\addplot [only marks] table {
3 4
2 2 
0 2   
1 0
};
\addplot [domain=0:3, samples=2, dashed] {2/3*x+2};
\addplot [domain=0:1, samples=2]
{-2*x+2};
\addplot [domain=1:3, samples=2]
{2*x-2};
\end{axis}
\end{tikzpicture} 

The Newton polygon (with the $y$-exponents plotted on the $x$-axis) is illustrated above.  For the side going from $(0, 2)$ to $(1, 0)$ we set $Y=y/x^2$ and multiply $f$ by $x^2$ to obtain the polynomial $1+Y+3x^4Y^2+x^8Y^3$.  Solving for $Y$ as a power series in terms of $x$, we obtain a unique solution by Hensel's lemma.  This gives $y$ as a polynomial of $x$, so the monodromy generator at 0 acts trivially.  For the other side going from $(1, 0)$ to $(3, 4)$ we set $Y=yx^2$ and multiply $f$ by $x^2$ to obtain the polynomial $Y^3+3Y^2+Y+x^4$.  This gives two additional solutions for $y$ in terms of $x$, and the monodromy generator at 0 acts trivially here too.  Together, these solutions give 3 copies of $Unip(1)$.   \\

At $\infty$, it suffices to set $x' = 1/x$ and carry out the same process as above using $x'$ instead of $x$.  This gives the polynomial $g = y^3+3x'^2y^2+x'^2+x'^4y$.  Setting $Y=y/x'^{1/3}$ and dividing $g$ by $x'$, we get $Y^3 + 3Y^2x'^{5/3} + Yx'^{10/3}+x'$.  This gives three solutions for $y$ in terms of $x'^{1/3}$.  This gives one copy of $Unip(1)$. \\ 

Thus, after dealing with the constant summand, we see that in this scenario $H^1(\A^1, j_{0!}p_{2*}j_{1*}i_*\Q_l)$ gives $2$ weights of weight 0 and none of weight 2.  The total number of weights is $2\Vol(\Delta)=4$, so in this case we have $w_0=2, w_1=2, w_2=0$. 
\end{example}

\begin{example}
Let $C$ be defined by the polynomial $f = x^3+xy+y^3$.  By Theorem \ref{thm: curve weights}, we have $w_0=1, w_1=0, w_2=2$.
\begin{tikzpicture}
\begin{axis}[
    axis lines=middle,
    xmin=-1, xmax=6,
    ymin=-1, ymax=6,
    xtick=\empty, ytick=\empty
]
\addplot [only marks] table {
3 0
1 1 
0 3
};
\addplot [domain=0:3, samples=3, dashed] {-x+3};
\addplot [domain=0:1, samples=2]
{-2*x+3};
\addplot [domain=1:3, samples=2]
{-1/2*x+3/2};
\end{axis}
\end{tikzpicture} 
 \\ 
\end{example}

\section{Weights of the fiber functor for higher dimensional varieties}
In this section we describe another approach to computing the weights of the fiber functor that uses results of \cite{DL}.  We begin with describing some of these results in the next subsection which apply to any dimension.  Then we use them to recompute the weights of the fiber functor for curves, and some surfaces and higher dimensional varieties.  While our computations for higher dimensional varieties is incomplete, in some scenarios we achieve what is necessary to compute the convolution monodromy group.  

The basic strategy is to use Proposition \ref{prop: ff 1} to view the fiber functor applied to the perverse sheaf corresponding to a variety as a cohomology group of a certain complex on $(\P^1)^n$.  One can then piece together the weights by computing them on a stratification of $(\P^1)^n$.  We note that the main difficulty lies in treating the bottom stratum where all coordinates are 0 or $\infty$, because the defining polynomial may have singularities at these points.

\subsection{Weights of compactly supported cohomology} 
Our starting point is the following result of \cite{DL}.  Given a hypersurface in the torus over a finite field $k$ defined by $f$, one can ask for the generating function given by the weight multiplicities of $f^*\L_{\psi}$ where $\L_{\psi}$ is the Artin-Schreier sheaf corresponding to a nontrivial additive character $\psi\colon k\rightarrow \C^\times$.  That is, let 
\[
E(\G_m^n, f) = \sum_{w=0}^ne_wT^w
\]
where $e_w$ is the number of eigenvalues of absolute value $q^{w/2}$ of the Frobenius action on \allowbreak $H^n_c((\G_m^n)_{\kbar}, f^*\L_{\psi})$.  Given a convex polyhedron $\Delta$ with integral vertices, \cite[Formulas 1.7.1, 1.7.2]{DL} define a polynomial $E(\Delta)$ and $e(\Delta)$ which give a recursive formula for $E(\G_m^n, f)$ and $e_n$ respectively.  

\begin{theorem}
[Theorem 1.8 of \cite{DL}] 
Suppose that $f:\G_m^n\rightarrow \A^1$ is nondegenerate with respect to $\Delta = \Delta_{\infty}(f)$, and that $\dim \Delta = n$.  Then $E(\G_m^n, f)=E(\Delta)$ and $e_n=e(\Delta)$. 
\end{theorem}  

We will use explicit formulas for $E(\Delta)$ in low dimensions, given in \cite[(8.3)]{DL}.  Using this, we can compute the weights of the compactly supported cohomology of the smooth hypersurface in $\G_m^n$ defined by $f$ as follows.  We will be interested in the alternating signed sum by degree of the weights, as we will relate it to the fiber functor in the next subsection. 

\begin{proposition}
The alternating signed sum of the weights of the compactly supported cohomology of a smooth hypersurface defined by $f$ are given by those corresponding to the Artin-Schreier sheaf $(fw)^*\L_\psi$, where $w$ is another variable and $\psi$ is a non-trivial character, added to the alternating signed weights of $H^i_c(\G_m^n, \Q_l)$, with every weight shifted down by 2. 
\end{proposition}
\begin{proof}
Consider the following diagram. 
\[\begin{tikzcd}
	{X\times \A^1} & {\G_m^n\times\A^1} & {\A^1} \\
	X & {\G_m^n}
	\arrow[from=1-1, to=1-2]
	\arrow[from=1-1, to=2-1]
	\arrow[from=2-1, to=2-2]
	\arrow["g", from=1-2, to=2-2]
	\arrow["{f(t)w}", from=1-2, to=1-3]
\end{tikzcd}\]
First, we claim that $Rg_!(fw)^*\L_{\psi}=i_*\Q_l[-2](-1)$.  Indeed, at some $t\in\G_m^n$ for which $f(t)\neq 0$, by the proper base change theorem we have $t^*R^ig_!(fw)^*\L_\psi = H^i(\A^1, \L_{f(t)\psi})=0$, since $\L_{f(t)\psi}$ is a nontrivial Artin-Schreier sheaf.  On the other hand, taking $X=V(f)$, by proper base change we have that the pullback of $Rg_!(fw)^*\L_{\psi}$ to $X$ is just the compactly supported pushforward of $\Q_l$ from $X\times \A^1$ to $X$, which is $i_*\Q_l[-2](-1)$ as desired. \\ 

Thus, we have $H^*_c(\G_m^n, i_*\Q_l) = H^{*-2}_c(\G_m^n, Rg_!(fw)^*\L_{\psi}) = H^{*-2}_c(\G_m^n\times \A^1, (fw)^*\L_\psi)$.  Now consider the stratification $\G_m^{n+1}\xrightarrow{j}\G_m^n\times \A^1\xleftarrow{i}\G_m^n\times\{0\}$.  We have $j^*(fw)^*\L_{\psi}$ is just the pullback of the Artin-Schreier sheaf associated to $\psi$ by the polynomial $fw$, considered as a polynomial in $n+1$ variables.  We have $i^*(fw)^*\L_{\psi}=\Q_l$.  Thus we have a long exact sequence 
\begin{align}
    0&\rightarrow H^0_c(\G_m^{n+1}, (fw)^*\L_\psi)\rightarrow H^0_c(\G_m^{n}\times \A^1, (fw)^*\L_\psi)\rightarrow H^0_c(\G_m^{n}, \Q_l) \\ 
    &\rightarrow H^1_c(\G_m^{n+1}, (fw)^*\L_\psi)\rightarrow H^1_c(\G_m^{n}\times \A^1, (fw)^*\L_\psi)\rightarrow H^1_c(\G_m^{n}, \Q_l) \\ 
    &\rightarrow \cdots
\end{align} 
We wish to compute the alternating signed weights of $H^*(\G_m^n, i_*\Q_l) = H_c^{*+2}(\G_m^n\times \A^1, (fw)^*\L_\psi(1))$.  The result follows from the long exact sequence. 
\end{proof} 

Let us make explicit what these results imply for small $n$. \\ 

In the case $n=2$, we wish to compute the signed weights of $H_c^{*+2}(\G_m^2\times \A^1, (fw)^*\L_\psi)$.  We have that
\[
H^i_c(\G_m^2, \Q_l) = \begin{cases}
\Q_l & i = 2 \\ 
\Q_l(-1)^2 & i = 3 ~ \text{(weight 2)} \\ 
\Q_l(-2) & i = 4 ~ \text{(weight 4)} \\ 
0 & \text{ otherwise}
\end{cases}
\]

We also need to compute the signed weights of $H^i_c(\G_m^3, (fw)^*\L_\psi)$.  By \cite[Theorem 1.3]{DL}, these are 0 for $i\neq 3$, and for $i=3$ are given by the following.  For a convex polyhedral cone $\sigma$ of dimension $n\le 4$, according to \cite[(8.3)]{DL}, we have the following formulas for $\G_m^3$. 
\begin{align}
    e_3 &= 6V_3-2V_2+V_1-(F_0(1)-2)V_0, \\  
    e_2 &= 2V_2-2V_1+(2F_0(1)-3)V_0, \\ 
    e_1 &= V_1-F_0(1)V_0, \\ 
    e_0 & =V_0. 
\end{align} 

Note that in our scenario $F_0(1)$ is just the number of vertices in the original Newton polygon, because every vertex lifted up to $w=1$ in the new coordinate and connected to the origin.  Call this number $m$.  We can express $V_0, V_1, V_2, V_3$ in terms of the corresponding volumes in the original polygon, which we will denote with $U_i$.  Namely, we have 

\begin{align}
    V_3 &= U_2/3, \\  
    V_2 &= U_1/2, \\ 
    V_1 &= U_0, \\ 
    V_0 &= 1. 
\end{align} 
This gives 
\begin{align}
    e_3 &= 2U_2-U_1+2, \\  
    e_2 &= U_1-3, \\ 
    e_1 &= 0, \\ 
    e_0 &= 1. 
\end{align} 

Finally, combining with the weights coming from $H^i_c(\G_m^2, \Q_l)$ and shifting down the weights by 2, we get the following 
\begin{proposition}
\label{prop: coh curve weights}
The alternating signed weights of $H^i_c(\G_m^2, i_*\Q_l)$, beginning with negative the weight for $i=0$, plus the weight for $i=1$, etc., are given by the following numbers. 
\begin{align}
    f_2 &= -1, \\ 
    f_1 &= 2U_2-U_1+2, \\ 
    f_0 &= U_1-1. 
\end{align} 
\end{proposition}

For $n=3$, (8.3) of \cite{DL} gives the following formulas for $\G_m^4$.  Let $F_0(i)$ be the number of $i$-dimensional faces of $\Delta$ containing the origin.  Then 
\begin{align}
    e_4 &= 24V_4-6V_3+2V_2-(W_1-2V_1)+(F_0(3)-3)V_0, \\ 
    e_3 &= 6V_3-4V_2+(2W_1-3V_1)-(F_0(3)+F_0(2)-6)V_0, \\ 
    e_2 &= 2V_2-W_1+(F_0(2)+F_0(1)-4)V_0, \\ 
    e_1 &= V_1-F_0(1)V_0, \\ 
    e_0 &= V_0, 
\end{align}  
where 
\[
W_1\coloneqq \sum_{\substack{\tau \text{ face of } \Delta \\ 0\in \tau, \dim \tau = 1}} F_\tau(3)\Vol(\tau). 
\]

Expressing $V_i$ through the volumes $U_i$ of the original polygon, we have 
\begin{align}
    V_4 &= U_3/4, \\ 
    V_3 &= U_2/3, \\  
    V_2 &= U_1/2,  \\ 
    V_1 &= U_0, \\ 
    V_0 &= 1. 
\end{align} 

Furthermore, we have $F_0(1)=U_0=V, F_0(2) = E, F_0(3)=F$ where $V, E, F$ are the number of vertices, edges, and faces of the original polyhedron. 
This gives 
\begin{align}
    e_4 &= 6U_3 - 2U_2 + U_1 +2U_0 + F - W_1 - 3, \\ 
    e_3 &= 2U_2-2U_1 - 3U_0 - E - F + 2W_1+6, \\  
    e_2 &= U_1 +U_0 +E - W_1 -4, \\ 
    e_1 &= 0, \\ 
    e_0 &= 1. 
\end{align}

Moreover, the weights of $H^i_c(\G_m^3, \Q_l)$ are as follows. 
\[
H^i_c(\G_m^3, \Q_l) = \begin{cases}
\Q_l & i = 3 \\ 
\Q_l(-1)^3 & i = 4 ~ \text{(weight 2)} \\ 
\Q_l(-2)^3 & i = 5 ~ \text{(weight 4)} \\ 
\Q_l(-3) & i = 6 ~ \text{(weight 6)} \\ 
0 & \text{ otherwise}
\end{cases}
\] 

Combining these results gives the following proposition. 
\begin{proposition}
\label{prop: coh surface weights}
The alternating signed weights of $H^i_c(\G_m^3, i_*\Q_l)$, beginning with positive the weight for $i=0$, minus the weight for $i=1$, etc, are given by the following numbers. 
\begin{align}
    f_4 &= 1, \\ 
    f_3 &= 0, \\ 
    f_2 &= 6U_3 - 2U_2 + U_1 +2U_0 + F - W_1 - 6, \\ 
    f_1 &= 2U_2- 2U_1 -3U_0 + 2W_1+6, \\  
    f_0 &= U_1 +U_0 +E- W_1 -1. \\ 
\end{align}
\end{proposition}

\subsection{Curves revisited} 
We will now use this approach to give another computation in the case of curves.
Recall that the situation of a curve in $\G_m^2$ is illustrated by the following diagram. 
\[\begin{tikzcd}
	C & {\G_m^2} & {\G_m\times \A^1} & {\G_m\times \P^1} & {\A^1\times \P^1} & {\P^1\times \P^1} \\
	&&&&& {\P^1} \\
	&&&&& {*}
	\arrow["i", from=1-1, to=1-2]
	\arrow["{j_{20}}", from=1-2, to=1-3]
	\arrow["{j_{2\infty}}", from=1-3, to=1-4]
	\arrow["{j_{10}}", from=1-4, to=1-5]
	\arrow["{j_{1\infty}}", from=1-5, to=1-6]
	\arrow["{q_2}", from=1-6, to=2-6]
	\arrow["{q_1}", from=2-6, to=3-6]
\end{tikzcd}\]

Let $K=Rj_{1\infty*}j_{10!}Rj_{2\infty*}j_{20!}i_*\Q_l$.  By Proposition \ref{prop: ff 1}, the fiber functor applied to the perverse sheaf associated to the curve is given by $H^1(\P^1\times \P^1, K)$. 

 We have a stratification given by $\G_m^2\xhookrightarrow{j} \P^1\times \P^1\xleftarrow{i}Z$, where $Z$ is given by $\G_m\times \{0, \infty\}\cup \{0, \infty\}\times \G_m \cup\{(0, 0), (0, \infty), (\infty, 0), (\infty, \infty)\}$.  We can calculate the pullback of the complex $K=Rj_{1\infty*}j_{10!}Rj_{2\infty*}j_{20!}i_*\Q_l$ restricted to these components. 

\subsubsection{Computing $K$ on each layer of the stratification}
 
First, it is clear that $j^*K$ is simply given by $i_*\Q_l$.  
 
 Because the stalk of the extension by 0 of an open embedding is 0 at points outside the open set, we have that the stalk of $K$ is 0 at all points where the first coordinate is equal to 0.  We denote the inclusions by $i_1:\G_m\times\{\infty\}\rightarrow \P^1\times \P^1, i_2:\infty\times \G_m\rightarrow \P^1\times \P^1, i_0:(\infty,  \infty)\rightarrow \P^1\times \P^1$. \\
 
 \textbf{At $\G_m\times\{\infty\}$ and  $\{\infty\}\times \G_m$} 
 
 At $\G_m\times\{\infty\}$, it is only supported on the points $(t, \infty)$ such that $f(t, \infty)=0$.  Denote these points by the inclusions $\alpha_1, \ldots \alpha_{n_1}$.  At such a point $z$, we have that as $U$ ranges over the \'etale neighborhoods of $z$ in $\G_m\times \P^1$, we have  
 \[
 (R^ij_{2\infty *}j_{20!}i_*\Q_l)_z = \varinjlim_{U} H^i(U\times_{\G_m\times \P^1}C, \Q_l).
 \]

For $i=0$ this is $\Q_l$, we showed in Proposition \ref{prop: curve weights} that for $i=0$ this is $\Q_l(-1)$, and for $i>1$ it is 0.  Applying $Rj_{1\infty*}j_{10!}$ to a complex on $\G_m\times \P^1$ and pulling back to $\G_m\times \{\infty\}$ is just the restriction, so indeed $\mc{H}^0(i_1^*K)=\bigoplus_{k=1}^{n_1} (\alpha_k)_*\Q_l$ and $\mc{H}^1(i_1^*K)=\bigoplus_{k=1}^{n_1} (\alpha_k)_*\Q_l(-1)$.  \\ 

At $\{\infty\}\times \G_m$, the calculation is essentially the same.  Indeed, denote the inclusions of the points $(\infty, t)$ with $f(\infty, t)=0$ by $\beta_1, \ldots, \beta_{n_2}$.  We can factor the inclusion $i_2$ as $\infty\times \G_m\xrightarrow{i_2}\P^1\times \G_m\xrightarrow{h}\P^1\times \P^1$.  Then by smooth base change on the squares represented by 
\[\begin{tikzcd}
	{\G_m\times \G_m} & {\G_m\times \P^1} \\
	{\P^1\times \G_m} & {\P^1\times \P^1}
	\arrow["{h'}", from=1-1, to=1-2]
	\arrow["{j_{1\infty}, j_{10}}", from=1-2, to=2-2]
	\arrow["{j_{1\infty}, j_{10}}"', from=1-1, to=2-1]
	\arrow["h"', from=2-1, to=2-2]
\end{tikzcd}\]

we see that $i_{2}^*h^*Rj_{1\infty*}j_{10!}Rj_{2\infty*}j_{20!}i_*\Q_l=i_2^*Rj_{1\infty*}j_{10!}h'^*Rj_{2\infty*}j_{20!}i_*\Q_l = i_2^*Rj_{1\infty*}j_{10!}i_*\Q_l$.  At this point the calculation is the exact same as the calculation of the restriction to $\G_m\times\{\infty\}$, and we see that $\mc{H}^0(i_2^*K)=\bigoplus_{k=1}^{n_2} (\beta_k)_*\Q_l$ and $\mc{H}^1(i_2^*K)=\bigoplus_{k=1}^{n_2} (\beta_k)_*\Q_l(-1)$. \\ 

\textbf{At $(\infty, 0)$ and $(\infty, \infty)$} 

We note that $P_1 = (\infty, 0)$ and $P_2 = (\infty, \infty)$ points may be singular.  If $\overline{C}$, the completion of $C$ in $\P^1\times \P^1$, contains $P_1$, then we consider the normalization of the curve.  Say this is given by $\overline{D}\rightarrow \overline{C}$ with $Q_1, \ldots, Q_{r_1}$ being the preimages of $P_1$.  We wish to compute the stalk of $R^ij_{1\infty *}j_{10!}Rj_{2\infty*}j_{20!}i_*\Q_l$ at $P_1$.  We know that, after passing to an appropriate open neighborhood of $P_1$, that $Rj_{2\infty*}j_{20!}i_*\Q_l$ is just given by $i_*\Q_l$.  Then as in the previous calculation, the stalk is given by 
\[
\varinjlim_{U}H^i(U\times_{\P^1\times \P^1}C, \Q_l) = \varinjlim_{U}H^i(U\times_{\P^1\times \P^1}D, \Q_l). 
\] 
This last expression is isomorphic to the direct sum of the stalks of the $i$th derived pushforward of $\Q_l$ at the $r_1$ points of the completed normalized curve $\overline{D}$, which we have seen to be $\Q_l$ in degree 0 and $\Q_l(-1)$ in degree 1.  The same argument applies for $P_2$, so we wish to compute these corresponding numbers $r_1$ and $r_2$.  They have a nice interpretation in terms of the Newton polygon.  Indeed, the normalization of a singularity at $(0, 0)$ is given by the different Puiseaux series corresponding to the slopes of the Newton polygon going from te $y$-axis to the $x$-axis.  The number of such Puiseaux series is given by the sum of the volumes of the associated sides.  Thus, for $(\infty, 0)$ we reflect the Newton polygon across the $x$-axis and shift it up to see which sides are counted.  For $(\infty, \infty)$ we do the same thing, except that we reflect it across both the $x$ and $y$-axes. 

\subsubsection{Computing the weights} 
We now compute the weights as follows.  
Using the long exact sequence associated to the short exact sequence 
\[
0\rightarrow j_!j^*K\rightarrow K\rightarrow i_*i^*K\rightarrow 0, 
\]
we can calculate the signed weights that give the fiber functor by doing so on the components of the stratification.  The part that comes from $j_!j^*K$ is the compactly supported cohomology of $S$, which we computed in Proposition \ref{prop: coh curve weights}. \\ 

Doing the same thing for $i_*i^*K$ part, we also have the compactly supported cohomology of $K$ restricted to $\G_m\times\{\infty\}$ and $\infty\times \G_m$.  We have a spectral sequence $E_2^{ij}=H^i_c(\G_m, \mc{H}^j(i_1^*K))\Rightarrow H^{i+j}_c(\G_m, i_1^*K)$.  But since $i_1^*K$ is a direct sum of skyscraper sheaves, only the $i=0$ terms are nonzero.  Doing the same for $i_2$, we see that in degree 0 we get $n_1+n_2$ weights of weight 0, and in degree 1 we get $n_1+n_2$ weights of weight 2. \\  

Finally, we have seen that the points $(\infty, \infty)$ and $(\infty, 0)$ contribute $r_1+r_2$ weight 0 in degree 0 and $r_1+r_2$ weight 2 in degree 1.  We arrive at the following result.  

\begin{proposition}
\label{prop: curve weights 2}
The weights of the fiber functor are given by the following numbers.
\begin{align}
    f_2 &= n_1+n_2+r_1+r_2 -1. \\ 
    f_1 &= 2U_2 - f_0-f_2 \\ 
    f_0 &= U_1 - n_1-n_2-r_1-r_2 - 1. 
\end{align}
\end{proposition}

Now we note that $r_1, n_2, r_2$ correspond to the parts of the inverse slope sequence of the Newton polygon with positive, 0, and negative slope respectively, while $n_1$ corresponds to the rightmost vertical side.  Thus, this agrees with the result calculated using the previous method in Theorem \ref{thm: curve weights}. 

\subsection{Generalities for surfaces}
\label{sec: gens}
We now consider the case of surfaces.  Here the perverse sheaf in question is of the form $i_*\Q_l[2]$, for some smooth surface $i:S\hookrightarrow \G_m^3$ nondegenerate with respect to its Newton polyhedron.  We take the same approach as for curves, and set $K=Rj_{1\infty*}j_{10!}Rj_{2\infty*}j_{20!}Rj_{3\infty*}j_{30!}i_*\Q_l$.   

 \subsubsection{Computing $K$ on each layer of the stratification}
 
 The stratification is given by $\G_m^3, \G_m^2\times \{0, \infty\} \cup \G_m\times \{0, \infty\}\times \G_m\cup \{0, \infty\}\times \G_m^2, \G_m\times \{0, \infty\}^2\cup \{0, \infty\}\times \G_m\times \{0, \infty\}\cup \{0, \infty\}\times \G_m, \{0, \infty\}^3$. \\ 
 
 On $\G_m^3$, the pullback of $K$ is just $i_*\Q_l$.  Furthermore, the stalk of $K$ is 0 at every point with 0 as the first coordinate, and at points in $\G_m\times 0\times \P^1$. 
 
\subsubsection{At $\G_m^2\times\{\infty\}$ et al.}  
 
 Now consider the pullback $i_{12}^*K$ where $i_{12}:\G_m^2\times\{\infty\}\rightarrow \G_m^2\times \P^1$ is the inclusion.  We are interested in computing $\mc{H}^i(i_{12}^*K)$.  In general, we have that
 \[
 \mc{H}^i(i_{12}^*K) = i_{12}^*( j_{1\infty*}j_{10!}j_{2\infty*}j_{20!}R^ij_{3\infty*}j_{30!}i_*\Q_l). 
 \]
 
 For $i=0$, we have that $\mc{H}^0(i_{12}^*K)$ is simply the sheaf $i_{3\infty*}\Q_l$ restricted to the curve $i_{3\infty}:C_{3\infty}\rightarrow \G_m^3$ defined by the equation $f(t_1, t_2, \infty)=0$.  The following propositions will show that $\mc{H}^1(i_{12}^*K) = i_{3\infty*}\Q_l(-1)$ and that $\mc{H}^i(i_{12}^*K)$ vanishes for larger $i$.
 
\begin{proposition}
\label{prop: curve inclusion}
Let $i_{3\infty}:C_{3\infty}\rightarrow \G_m^3$ be the inclusion of the curve $f(t_1, t_2, \infty)=0$ into $\G_m^3$.  Then $R^1j_{3\infty*}j_{30!}i_*\Q_l = i_{3\infty*}\Q_l(-1)$.  
\end{proposition}
\begin{proof}
The stalk of $q_*R^1j_{3\infty *}j_{30!}i_*\Q_l$ at each $t\in \G_m^2$ is zero unless $f(t_1, t_2, \infty)=0$.  Letting the curve defined by this equation be denoted $i_{3\infty}: C_{3\infty}\rightarrow \G_m^2\times \A^1$, we claim that this sheaf is given by $q_*i_{3\infty*}\Q_l(-1)$. \\
 
We will first show that there is an open cover of $C_{3\infty}$ given by $C'_{3\infty}\cup C''_{3\infty}$, such that both of them have \'etale neighborhoods in $S$ completed at $\infty$ that are isomorphic to \'etale neigborhoods of $C'_{3\infty}$ and $C''_{3\infty}$ inside $C'_{3\infty}\times \A^1$ and $C'_{3\infty}\times \A^1$, respectively.  Consider the two projections of $\overline{S}$, which is $S$ completed at $\infty$ in the last coordinate (which is the vanishing locus of $f$ in $\G_m^2\times \A^1$), onto $\G_m\times \A^1$.  If the first projection is \'etale on $U'\subset \overline{S}$ and the second is \'etale on $U''\subset \overline{S}$, then by the smoothness of $S$ we have that $C_{3\infty}\subset U'\cup U''$.  We set $C'_{3\infty}=C_{3\infty}\cap U'$.  Now the projection composed with the inclusion gives an \'etale morphism $C'_{3\infty}\rightarrow \A^1$.  Base changing this by the projection composed with the inclusion $U'\rightarrow \A^1$ gives an \'etale neighborhood of $U'$ given by $C'_{3\infty}\times_{\A^1}U'$.  Then this neighborhood contains $C'_{3\infty}$ via the diagonal, and is also \'etale over $C'_{3\infty}\times \A^1$.  This shows the claim. \\ 

Now we will compute $R^1j_{3\infty*}j_{30!}i_*\Q_l$ restricted to $C'_{3\infty}$ and $C''_{3\infty}$ and show they are both $\Q_l(-1)$.  Then because a Zariski-locally constant sheaf is constant, this will imply that $R^1j_{3\infty *}j_{30!}i_*\Q_l = i_{3\infty *}\Q_l(-1)$, from which the desired result follows.    \\

Note that $R^1j_{3\infty *}j_{30!}i_*\Q_l = i_{3\infty *}\Q_l(-1)$ is the sheafification of a presheaf that for $h:V\rightarrow \G_m^2\times \P^1$ \'etale, that does not map onto the $z=0$ part (which is ok since we are only concerned with \'etale neighborhoods of $C_{3\infty}$) sends 
\[
V\mapsto H^1(V\times_{\G_m^2\times \P^1}S, \Q_l). 
\]
For a fixed $U\subset C'_{3\infty}$, we consider the $V$ above that map onto $U$.  Then the \'etale neighborhoods  $U\times \A^1$ coming from the \'etale neighborhood $C_{3\infty}\times \A^1$ are cofinal in the category of all such $V$.  Therefore, taking the limit over all \'etale neighborhoods mapping onto $U$, the sheaf coincides with $R^1f_*\Q_l(U)$ where $f:C'_{3\infty}\times \G_m\hookrightarrow C'_{3\infty}\times \A^1$ is the inclusion.  It follows that we just need to compute the sheaf $R^1f_*\Q_l(U)$, which we know how to do.  Indeed, from the previous section we know that for $f_1:\G_m\rightarrow \A^1$, we have $(R^1f_{1*}\Q_l)_0 = \Q_l(-1)$.  Then by smooth base change on the square 
\[\begin{tikzcd}
	{C_{3\infty}\times \G_m} & {\G_m} \\
	{C_{3\infty}\times \A^1} & {\A^1}
	\arrow[from=1-1, to=1-2]
	\arrow[from=1-2, to=2-2]
	\arrow[from=1-1, to=2-1]
	\arrow[from=2-1, to=2-2]
\end{tikzcd}\]
we see that $R^1f_*\Q_l$ is indeed $\Q_l(-1)$ supported on the 0 section of $C'_{3\infty}\times \A^1$, as desired.  This implies that $R^1j_{3\infty *}j_{30!}i_*\Q_l = i_{3\infty *}\Q_l(-1)$.  
\end{proof} 

\begin{proposition} 
\label{prop: surface vanishing}
The sheaf $R^ij_{3\infty*}j_{30!}i_*\Q_l$ vanishes for $i\ge 2$.
\end{proposition}
\begin{proof}
  The sheaf $R^2j_{3\infty*}j_{30!}i_*\Q_l$ is only supported on points $z = (x, y, \infty)$ with $f(x, y, \infty)=0$. \\ 
For such a point $z$, we have 
\[
(R^ij_{3\infty*}j_{30!}i_*\Q_l)_{z} = \varinjlim_{U} H^i(U \times_{\G_m^2\times \P^1} (\G_m^2\times \A^1), j_{1!}i_*\Q_l)  \cong \varinjlim_{U} H^i(U \times_{\G_m^2\times \P^1} S, \Q_l). 
\]
For $i>2$, this is 0 because of the Artin vanishing theorem.  Now consider $i=2$.  These stalks are the same for each point on the curve $f(x, y, \infty) = 0$.  However, since $j_{3\infty}$ is affine, by Artin vanishing we know what the stalks of the $R^2j_{3\infty*}$ from the surface $S$ can be nonzero only on a dimension 0 subset.  Therefore these stalks cannot be nonzero, so they are all 0. 
\end{proof} 
 
 We now want to compute the corresponding restrictions $i_{13}^*K$ and $i_{23}^*K$ to $\G_m\times\{\infty\}\times \G_m$ and $\infty\times \G_m^2$.  We claim that the analogous results hold for them, except that the curve $C_{3\infty}$ is replaced with $C_{2\infty}$ and $C_{1\infty}$, defined by $f(x, \infty, z)=0$ and $f(\infty, y, z)=0$ respectively.  The proof of this is the same as that of the analogous result which was given in the case of a curve.  Namely, by the smooth base change theorem the restriction comes down to computing $R^ij_{2\infty*}j_{20!}i_*\Q_l$ and $R^ij_{1\infty*}j_{10!}i_*\Q_l$, which have just been done.  
 
 \subsubsection{At $\G_m\times\{\infty\}\times \{0\}, \G_m\times \{\infty\}\times\{\infty\}$ et al.}  
 \label{sec: sw3}
 
 Next comes the issue of computing the pullback $i_1^*K$ where $i_1:\G_m\times \{\infty\}\times\{ \infty\}\rightarrow \G_m\times (\P^1)^2$ is the inclusion, and the analogous pullback to $\G_m\times\{\infty\}\times 0$.  In particular we want to compute 
 \[
 \mc{H}^i(i_{1}^*K) = i_{1}^*( j_{1\infty*}j_{10!}R^i(j_{2\infty*}j_{20!}j_{3\infty*}j_{30!})i_*\Q_l).
 \]
 
 There are three cases for $\G_m\times\{\infty\}\times\{\infty\}$; first where there are no terms that have the largest $y$ and $z$ exponent, second when there is one such term, and third when there is more than one such term.  In the first, we have $f(t, \infty, \infty)=0$ for all $t\in \G_m$, in the second we have $f(t, \infty, \infty)\neq 0$ for all $t\in \G_m$, and in the third we have $f(t, \infty, \infty)=0$ for some finite number of possible $t\in\G_m$.  Similar reasoning applies to $\G_m\times\{\infty\}\times \{0\}$.  There is nothing to calculate in the second case.  
 
 \paragraph{The first case} 
 We get that for both $\G_m\times\{\infty\}\times\{\infty\}$ and $\G_m\times\{\infty\}\times \{0\}$, the pullback in this scenario will just be $\Q_l$ in degree 0 and $\Q_l(-1)$ in degree 1.

 \paragraph{The third case}

There is a spectral sequence 
\[E_2^{ab} = R^aj_{2\infty *}j_{20!}R^bj_{3\infty*}j_{30!}i_*\Q_l \Rightarrow R^{a+b}(j_{2\infty*}j_{20!}j_{3\infty*}j_{30!})i_*\Q_l.\]
Recall that we have computed that $R^0j_{3\infty*}j_{30!}i_*\Q_l$ is $i'_*\Q_l$ where $i'$ is the inclusion of the completion of the surface in $\G_m^2\times \A^1$, with $\A^1$ referring to $\G_m\cup \{\infty\}$, while $R^1j_{3\infty*}j_{30!}i_*\Q_l = i_{3\infty *}\Q_l(-1)$, and $R^ij_{3\infty*}j_{30!}i_*\Q_l = 0$ for larger $i$.  This implies that $E_2^{ab}$ is as follows. 
\[\begin{tikzcd}
	{i_{3\infty*}\Q_l(-1)} & {i_{23\infty *}\Q_l(-2)} \\
	{i''_*\Q_l} & {i_{2\infty*}\Q_l(-1)}
\end{tikzcd}\]

Here $i''$ is the inclusion of the completion of the surface in $\G_m\times (\P^1)^2$ except for the points in $\G_m^2\times \{0\}$, while $i_{2\infty}$ is the inclusion of $f(x, \infty, z) = 0$ in $\G_m\times (\P^1)^2$, and $i_{23\infty}$ is the inclusion of $f(x, \infty, \infty) = 0$ in $\G_m\times (\P^1)^2$.  This spectral sequence clearly degenerates.  Then restricting back to $\G_m\times\{\infty\}\times\{\infty\}$, we obtain the following.  Let $\alpha_1, \ldots, \alpha_{n_1}$ be the inclusions of the $n_1$ solutions to $f(x, \infty, \infty)=0$.  Then 
\[\mc{H}^0(i_1^*K) = \bigoplus_{k=1}^{n_1}\alpha_{k*}\Q_l, \mc{H}^1(i_1^*K) = \bigoplus_{k=1}^{n_1}\alpha_{k*}(\Q_l(-1))^2, \mc{H}^2(i_1^*K) = \bigoplus_{k=1}^{n_1}\alpha_{k*}\Q_l(-2).\] 

Next, we claim that the signed weights of the restriction to $i_1': \G_m\times\{\infty\}\times 0\rightarrow \G_m\times (\P^1)^2$ are 0.

Indeed, first note that the stalk of $R^kj_{2\infty *}j_{20!}R^aj_{3\infty *}j_{30!}i_*\Q_l$ at $(t, \infty, 0)$ is the same as that of $R^kj_{2\infty *}j_{30!}i_*\Q_l$.  Then if we let $i_{30}$ be the inclusion of $(t_1, t_2, 0)$ on this surface, we have a long exact sequence relating $R^kj_{2\infty*}j_{30!}\Q_l, R^kj_{2\infty*}\Q_l, R^kj_{2\infty*}i_{30*}\Q_l$ which, upon taking the stalk at $(t, \infty, 0)$, shows that the signed weights of the stalk of $R^kj_{2\infty*}j_{30!}\Q_l$ at those points is the difference of the latter two.  Both of those are $\Q_l$ in degree 0 and $\Q_l(-1)$ in degree 1, so we are done. \\ 

Using the same argument as for the previous layer of the stratification, we get analogous results for the restriction to $\{\infty\}\times \G_m\times\{\infty\}$ and $\{\infty\}\times\{\infty\}\times \G_m$.  Namely, 
\[
\mc{H}^0(i_2^*K) = \bigoplus_{k=1}^{n_2}\beta_{k*}\Q_l, \mc{H}^1(i_2^*K) = \bigoplus_{k=1}^{n_2}\beta_{k*}(\Q_l(-1))^2, \mc{H}^2(i_2^*K) = \bigoplus_{k=1}^{n_2}\beta_{k*}\Q_l(-2)
\] and 
\[\mc{H}^0(i_3^*K) = \bigoplus_{k=1}^{n_1}\gamma_{k*}\Q_l, \mc{H}^1(i_3^*K) = \bigoplus_{k=1}^{n_1}\gamma_{k*}(\Q_l(-1))^2, \mc{H}^2(i_3^*K) = \bigoplus_{k=1}^{n_1}\gamma_{k*}\Q_l(-2).\] 
Similarly, the restrictions to $\infty\times \G_m\times 0$ and $\infty\times 0\times \G_m$ also do not change the total signed weights.

\subsubsection{Computing the weights outside of the singular stratum}  
As in the case of a curve, we can compute the alternating signed weights of the fiber functor by doing so on each layer of the stratification.  We have computed the cohomology sheaves of the complex in question over every stratum except for the corners, which are more difficult.  For now, we will see how to put all this information together and do so except for the corners.  We represent the stratification as follows. 
\[
\adjustbox{scale=0.83,center}{%
\begin{tikzcd}
	{\G_m^3} & {(\P^1)^3} & {Z_1} & {Z_2} & {\{0,\infty\}\times\{0,\infty\}\times\{0,\infty\}} \\
	&& {\G_m^2\times\{0, \infty\} \text{ et al.}} & {\G_m\times\{0,\infty\}\times\{0,\infty\} \text{ et al.}}
	\arrow[hook, from=1-1, to=1-2]
	\arrow[hook', from=1-3, to=1-2]
	\arrow[hook, from=2-3, to=1-3]
	\arrow[hook', from=1-4, to=1-3]
	\arrow[hook', from=1-5, to=1-4]
	\arrow[hook, from=2-4, to=1-4]
\end{tikzcd}
}\]

\paragraph{On $\G_m^3$}
Proposition \ref{prop: coh surface weights} gives the result on $\G_m^3$ to be 
\begin{align}
    f_4 &= 1, \\ 
    f_3 &= 0, \\ 
    f_2 &= 6U_3 - 2U_2 + U_1 +2U_0 + F - W_1 - 6, \\ 
    f_1 &= 2U_2- 2U_1 -3U_0 -F -E + 2W_1+6, \\  
    f_0 &= U_1 +U_0 +E- W_1 -1. \\ 
\end{align}
Here $F_0(i)$ is the number of $i$-dimensional faces, $U_i$ is the sum of the volumes of the $i$-dimensional faces, and 
\[
W_1\coloneqq \sum_{\substack{\tau \text{ face of } \Delta \\ 0\in \tau, \dim \tau = 1}} F_\tau(3)\Vol(\tau). 
\] 

\paragraph{On the codimension 1 stratum}
On $\G_m^2\times\{\infty\}$, we need to compute the signed weights of $H^i_c(\G_m^2, i_{12}^*K)$.  We have a spectral sequence $E_2^{ab} = H^{a}_c(\G_m^2, \mc{H}^b(i_{12}^*K)) \Rightarrow H^{a+b}_c(\G_m^2, i_{12}^*K)$.  By our calculations of $\mc{H}^b(i_{12}^*K)$, this is given by the following. 
\[\begin{tikzcd}
	{H^0_c(C_{3\infty}, \Q_l(-1))} & {H^1_c(C_{3\infty}, \Q_l(-1))} & {H^2_c(C_{3\infty}, \Q_l)} \\
	{H^0_c(C_{3\infty}, \Q_l)} & {H^1_c(C_{3\infty}, \Q_l)} & {H^2_c(C_{3\infty}, \Q_l)}
	\arrow[from=1-1, to=2-3]
\end{tikzcd}\] 
Because we are just looking for the signed weights, this is given by the signed weights of the compactly supported cohomology of $C_{3\infty}$ added to the negative of itself with all weights shifted up by 2.  We recall the Proposition \ref{prop: coh curve weights} gives these to be 
\begin{align}
    f_2 &= 1, \\ 
    f_1 &= -2U_2+U_1-2, \\ 
    f_0 &= -U_1+1. 
\end{align} 
if we start with positive the weight for $i=0$.  Thus, here the contribution is 
\begin{align}
    f_4 &= -1, \\ 
    f_3 &= 2U_2-U_1+2, \\ 
    f_2 &= U_1, \\ 
    f_1 &= -2U_2+U_1-2, \\ 
    f_0 &= -U_1+1. 
\end{align} 

The same formulas hold for $C_{1\infty}$ and $C_{2\infty}$.  These correspond to the curves in $\G_m^2$ determined by the monomials with the largest $x, y, z$ exponents.  

\paragraph{On the codimension 2 stratum}
Next we need to compute the alternating signed weights on $\G_m\times\{\infty\}\times\{\infty\}$.  In the case that $K$ is supported on $n_1$ points, this gives $n_1$ of weight 1, $-2n_1$ of weight 2, and $n_1$ of weight 4. \\ 

In the case that it is supported on all of $\G_m$, we have that the pullback is $\Q_l$ in degree 0 and $\Q_l(-1)$ in degree 1.  Then the weights are given by $\Q_l$ in degree 1, $\Q_l(-1)^2$ in degree 2, and $\Q_l(-2)$ in degree 3.  The signed weights would thus be given by 
\begin{align}
    f_4 &= -1, \\ 
    f_2 &= 2, \\  
    f_0 &= -1, 
\end{align}

\subsection{Examples of specific surfaces} 
\label{sec: weight examples}
The previous results allow us to calculate the weights of a reasonably large class of surfaces.  As long as the defining polynomial $f$ does not have roots of the form $\infty\times \{0, \infty\} \times \{0, \infty\}$, then we can combine the results above to compute the desired weights.  This condition is satisfied if the face of the Newton polyhedron with largest $z$-coordinate is a rectangle with sides parallel to the $x$ and $y$ axes.  

\subsubsection{Rectangular prism}
Consider a polyhedron with Newton polyhedron given by $1+x^a+y^b+z^c + x^ay^b + y^bz^c + x^az^c + x^ay^bz^c$.  For $\G_m^3$, we have that $U_0 = 8, U_1 = 4a+4b+4c, U_2 = 2ab+2bc+2ca, U_3 = abc$.  

Here, $W_1$ refers to the polyhedron in one higher dimension with 1 added, so it is $\sum_v F_v(3) = 8\cdot 3 = 24$.  So this part gives 
\begin{align}
    f_4 &= 1, \\ 
    f_3 &= 0, \\ 
    f_2 &= 6abc - 4ab-4bc-4ca +4a+4b+4c - 8. \\ 
    f_1 &= 4ab+4bc+4ca-8a-8b-8c+12, \\  
    f_0 &= 4a+4b+4c-5.
\end{align}. 

Next, the contribution of the codimension 1 strata come from three of the faces, which gives  
\begin{align}
    f_4 &= -3, \\ 
    f_3 &= 2ab+2bc+2ca - 4a-4b-4c+6 \\ 
    f_2 &= 4a+4b+4c, \\ 
    f_1 &= -2ab-2bc-2ca + 4a+4b+4c-6, \\ 
    f_0 &= - 4a-4b-4c+3. 
\end{align}

Finally, the contribution of the codimension 2 strata come from the six edges.  The three of the edges corresponding to $\infty\times\{\infty\}\times \G_m, \infty\times \G_m\times\{\infty\}\times\{\infty\}\times \G_m$ give 
\begin{align}
    f_4 &= a+b+c, \\ 
    f_2 &= -2a-2b-2c, \\ 
    f_0 &= a+b+c. 
\end{align} 
The other three do not give any signed weights; see the argument used in the example of the pyramid below. \\ 

This gives a final total of 
\begin{align}
    f_4 &= a+b+c-2, \\ 
    f_3 &= 2ab+2bc+2ca - 4a-4b-4c+6, \\ 
    f_2 &= 6abc - 4ab-4bc-4ca +6a+6b+6c -8, \\ 
    f_1 &= 2ab+2bc+2ca -4a-4b-4c+6, \\  
    f_0 &= a+b+c-2.
\end{align}

\subsubsection{Pyramid} 
\label{sec: pyramid}
Consider a polyhedron with Newton polyhedron given by $1+y^a+z^b+y^{a}z^b + x^{-c}y^dz^{e}$.  Choose $c, d, e$ so that $0< d < a, 0< e < b$, and $c$ large so that the volume of every edge involving the vertex is 1.  We have that $U_0 = 5, U_1 = 2a+2b+4, U_2 = ab+a+b, U_3 = abc/3$.  Also, $W_1 = 4\cdot 3+4=16$ and $V = 5, E = 8, F = 5$. 

\paragraph{On $\G_m^3$} 
We get 
\begin{align}
    f_4 &= 1, \\ 
    f_3 &= 0, \\ 
    f_2 &= 2abc - 2(ab+a+b) + (2a+2b+4) + 2\cdot 5 +5 -16-6 = 2abc -2ab - 3. \\ 
    f_1 &= 2(ab+a+b) - 2(2a+2b+4) -15-5-8+2\cdot 16+6 = 2ab-2a-2b+2, \\  
    f_0 &= 2a+2b+4 + 5 +8-16-1 = 2a+2b.
\end{align}. 

\paragraph{On codimension 1} 
Next, the contribution of the codimension 1 strata comes from the top rectangular face and two of the edges.  The first corresponds to the zero locus of $1+y^a+z^b+y^az^b$ in $\G_m^2\times\{\infty\}$.  This gives the following weights.  
\begin{align}
    f_4 &= -1, \\ 
    f_3 &= 2ab-2a-2b+2 \\ 
    f_2 &= 2a+2b, \\ 
    f_1 &= -2ab+2a+2b-2, \\ 
    f_0 &= -2a-2b+1. 
\end{align}

For the side corresponding to $y^a+y^az^b$ in $\infty\times \G_m^2$, we obtain $b$ copies of $\G_m$ over which the pullback of $K$ is $\Q_l$ in degree 0 and $\Q_l(-1)$ in degree 1.  A similar result holds for the side corresponding to $z^b+y^az^b$, and together these two sides give the following weights. 
\begin{align}
    f_4 &= -a-b, \\  
    f_2 &= 2a+2b, \\ 
    f_0 &= -a-b. 
\end{align} 

\paragraph{On codimension 2} 
Finally, the contribution of the codimension 2 strata comes from $\infty\times\{\infty\}\times \G_m, \infty\times \G_m\times\{\infty\}, \infty\times 0\times \G_m$, and $\infty\times \G_m\times 0$. \\

The first is supported on $1+z^b$ in $\infty\times\{\infty\}\times \G_m$.  Before pushing forward the second time, the complex in degree 0 is $\Q_l$ on the surface and in degree 1 is $\Q_l(-1)$ on the intersection of the surface and $z=\infty$.  Thus each of these $a$ points contribute one weight 2, -2 weight 2, 1 weight 4.  The same reasoning holds for the restriction to $\infty\times \G_m\times\{\infty\}$, so the contribution here is given by 

\begin{align}
    f_4 &= a+b, \\ 
    f_2 &= -2a-2b, \\ 
    f_0 &= a+b.  
\end{align}

The stalk at the third divisor is supported on $1+z^b$ in $\infty\times 0\times \G_m$.  After applying $j_{20!}$ to the complex, near these points it is the constant sheaf supported on the surface minus its intersection with $(\infty, y, z)$ and $(x, 0, z)$.  Calling these divisors $D_1$ and $D_2$ respectively and the inclusions $j:S-D_1-D_2\rightarrow S-D_1$, $f:S-D_1\rightarrow S$, we wish to compute $R^if_*j_!j^*\Q_l$.  Using the exact sequence 
\[
0\rightarrow j_!j^*\Q_l\rightarrow \Q_l\rightarrow i_*i^*\Q_l\rightarrow 0, 
\]
we see that in the K-group these are given by the difference of the stalks of $R^if_*\Q_l$ and $R^if_*i_*i^*\Q_l$.  But these are the same, namely $\Q_l$ in degree 0 and $\Q_l(-1)$ in degree 1, so their difference is 0.  Thus there is no net contribution here to the signed weights.

\paragraph{Final total} 
This gives a final total of 
\begin{align}
    f_4 &= 0, \\ 
    f_3 &= 2ab-2a-2b+2, \\ 
    f_2 &= 2abc -2ab+2a+2b - 3, \\ 
    f_1 &= 0, \\  
    f_0 &= 1.
\end{align}.

\subsection{Computing the highest weight multiplicity of higher-dimensional polyhedra} 
\label{sec: higher weights}
\subsubsection{Setup}
The difficulty we encountered in computing the weights of the fiber functor on surfaces arose from the possible singularities arising from the corners in the compactification into $(\P^1)^3$.  This same issue arises when we move to higher dimensions, but we can avoid it by looking at Newton polyhedra which do not contain such points in their compactifications.  The prime example of such a polyhedron is an $n$-dimensional rectangular prism.  To go beyond these, we note that because the fiber functor is built from some extension-by-0 functors, we can allow for some corners to be included in the compactification.  This motivates us to consider the following polyhedron. 

Consider the case of a Newton polyhedron given by a rectangular prism with the corner near the origin sawed off; i.e. in the positive orthant with dimensions $b_1, \ldots, b_n$ except that the origin is replaced with $(a_1, 0, \ldots, 0), \ldots, (0, \ldots, 0, a_n)$ for $0 < a_i < b_i$.  We aim to compute the multiplicity of the highest weight, namely $f_{2n-2}$. \\ 

We first describe the plan of the calculation.  As before, we will consider the restriction of the complex $Rj_{1\infty *}j_{10!}Rj_{2\infty *}j_{20!}\cdots Rj_{n\infty *}j_{n0!}i_*\Q_l$ to each part of the stratification of $(\P^1)^n$ defined by setting some coordinates to $\{0, \infty\}$.  The key point is that when some of the coordinates are set to 0, the restriction of the complex to the resulting subvarieties do not contribute to the signed weights.  This was already seen in Section \ref{sec: sw3}, where in the case $n=3$ we computed the restriction of the complex to $\G_m\times \{\infty\}\times \{0\}$.  This allows us to proceed by computing the contributions from the subvarieties obtained by setting some coordinates to 0.

\begin{proposition}
    Let $X\xrightarrow{i} \G_m^n$ be a hypersurface nondegenerate with respect to its Newton polyhedron such that its compactification $\bar{X}\subset (\P^1)^n$ only contains corners where the first coordinate is 0.  Let $K$ be the complex $j_{10!}Rj_{1\infty*}\cdots j_{n0!}Rj_{n\infty*}(i_*\Qlb[n-1])$.  Then the restriction of the $K$ to the pieces of the stratification of $\bar{X}\subset (\P^1)^n$ where a coordinate is equal to 0 contributes nothing to the signed weights. 
\end{proposition} 
\begin{proof}
First, we note that $K'=j_{10!}Rj_{1\infty*}\cdots j_{n0!}Rj_{n\infty*}\Qlb$ has stalks equal to 0 on points with any coordinates set to equal to 0.  Let $S$ be the set of singularities of $\bar{X}$.  Then $\bar{X}\backslash S\hookrightarrow (\P^1)^n$ is smooth, so the restriction of $K$ to the points of $\bar{X}\backslash S\hookrightarrow (\P^1)^n$ with 0 as a coordinate is the same as that of $K'$, which is 0.  Finally, the fact that $K$ is constructed by $j_{10!}$ as the final functor applied implies that its restriction to $S$ is also 0, as in the conditions the first coordinate of every point in $S$ is 0.   
\end{proof}

Thus the complex is restricted to pieces of the stratification where some coordinates are set to $\infty$, and the contribution to the signed weights from these pieces come from the compactly supported cohomology of the corresponding variety.  The formula for this is given through recursive formulas of Denef-Loeser \cite{DL}, though they are generally difficult to compute.  However, it will be sufficient for our purposes to compute the highest weight multiplicity, which we can do by hand. 

\subsubsection{The example of $n=3$} 
In the case of $n=3$, the Newton polyhedron is given by 
\[
(1+x^a)(1+y^b)(1+z^c) - 1 + x^d+y^e+z^f
\] 
for $a>d, b>e,  c>f$.  
We recall that the signed weights begin with positive the weight of $H^0$.  Over $\G_m^3$, we have $f_4=1$.  We recall that for curves, we have $f_2=1$, but their sign will change when being pushed up to $f_4$. \\ 

In codimension 1, i.e. $\G_m^2\times\{\infty\}, \G_m\times\{\infty\}, \infty\times \G_m^2$, the complex restricts to that of a curve in each of the 3 cases.  This gives $f_4 = -3$. \\ 

In codimension 2, i.e. $\G_m\times \{\infty\}\times \{\infty\}, \G_m\times\{\infty\}\times \G_m, \infty\times \G_m^2$, we obtain $a$ points in each scenario, corresponding to the roots of $1+x^a, 1+y^b, 1+z^c$.  Note that by varying the corner cut off, these exponents can differ from each other.  This gives $f_4=a+b+c$. \\ 

Since there is nothing of dimension 0, this gives a total of $f_4 = a+b+c-2$. 

\subsubsection{In general}
\label{sec: truncated}
The Newton polyhedron is of the form 
\[
(1+x_1^{b_1})\cdots (1+x_n^{b_n}) - 1 + x_1^{a_1} +\cdots + x_n^{a_n}
\]
with $a_i<b_i$. \\ 

Recall that for $n$ even, we begin with negative the weight of $H^0$ and with $n$ odd, positive the weight of $H^0$. \\ 

On $\G_m^n$, we have a contribution to $f_{2n-2}$ of $(-1)^{n+1}$.  In general, if we set $k$ of these coordinates to be $\infty$ for $1\le k\le n-2$, we will get a variety that contributes a single weight.  Accounting for the sign, along with the initial weight coming from $\G_m^n$ this gives a subtotal of 
\[
(-1)^{n+1}\left(1 -\binom{n}{1} + \binom{n}{2} -\cdots +(-1)^{n-2}\binom{n}{n-2}\right)= -n+1. 
\]
Finally, setting $n-1$ of the coordinates to be $\infty$, we obtain equations of the form $1+x_j^{b_j}$ for each $1\le j\le n$, which contribute $(-1)^{n+1}(-1)^{n-1}(\sum_{j=1}^nb_j) = \sum_{j=1}^nb_j$ to $f_{2n-2}$.  Thus, the final total is $f_{2n-2} = \sum_{j=1}^nb_j-n+1$. 

\section{Computing the convolution monodromy groups} 
Using our computation of the weights of the fiber functor, we will now appeal to two known results to show that the convolution monodoromy group is large in several cases. 
\subsection{A criterion regarding weights}
The following result gives a criterion for a subgroup of $GL_R(\Qlb)$ to be large in terms of the eigenvalues of one of its elements.  Given a tuple $\lambda\in (\Qlb^*)^R$, define $\mu_R(\lambda)$ to be the partition of $R$ defined by the multiplicity of the absolute values of the elements of $\lambda$, after taking an embedding into $\C$.  

\begin{theorem}
[Theorem C.4.1 of \cite{HKR}]
\label{thm: A}
Let $r, R$ be positive integers and G be a connected reductive subgroup of $GL_R(\overline{\Q}_l)$.
Let $g \in G$ be an element and $\gamma\in (\overline{\Q}_l^*)^R$ be an eigenvector tuple of g. 
Suppose that $G$ is irreducible,
that $\gamma$ lies in $(\overline{\Q}_l^*)^R$, and that $c=\mu_R(\lambda)$ satisfies $\len(c) \le r + 1$ and 
$1 = c_{\len(c)} < c_{\len(c)-1}$ and
$c_2 \le r$. 
If $R > 72(r^2 + 1)^2$, then either $G = SL_R(\overline{\Q}_l)$ or $G = GL_R(\overline{\Q}_l)$.
\end{theorem}

We can apply this theorem to curves and surfaces, using our computations of their weights in the previous sections.  Furthermore, because of the asymmetry in the definition of the fiber functor with respect to 0 and $\infty$, the conclusion of the theorem may hold even if the conditions do not directly apply to the computation of the weights of the fiber functor.  For example, given a curve we may switch the role of $x$ and $y$, obtaining an isomorphic curve whose fiber functor may have different weights.  More generally, we may apply a transformation $x\mapsto x^ay^b, y\mapsto x^cy^d$ for $\begin{bsmallmatrix}a & b \\ c & d\end{bsmallmatrix}\in GL(2, \Z)$.  Such transformations send the curve $C$ to an isomorphic one in $\G_m^2$ defined by different equations, which may have a different Newton polygon and different weights.  Of course, the monodromy group remains the same.  We make a few observations as to the new weights which can be obtained in this way. 

\begin{itemize}
    \item Reflecting across the $x$-axis exchanges the monodromy at 0 and at $\infty$.  Modulo this, we can consider $\begin{bsmallmatrix}a & b \\ c & d\end{bsmallmatrix}\in SL(2, \Z)$. 
    
    \item A linear transformation preserves convex hulls; therefore the same sides will be used in the calculation of the local monodromy. 
    
    \item By factoring $SL(2, \Z)$ into $S$ and $T$, using the Euclidean algorithm we see that each side will also give the same number of cycles and thus the same number of $Unip(1)$s, unless it is vertical in which case it doesn't give anything. 
    
    \item The only other way for the weights to change is for a different contiguous set of edges to be used for calculating monodromy over 0.
\end{itemize}

One can check that all possible configurations of a triangle can be obtained.  In the following example, the original triangle does not satisfy the conditions of the theorem but some transformation of it does. 

\begin{example}
\label{ex: curve}
Consider a curve, 0-nondegenerate to its Newton polygon, whose Newton polygon is a triangle.  If the greatest common factors of the changes in $x$ and $y$ comprising the sides are $1$, $2$, and $3$ in any order and the area of the triangle is greater than 7204, then the conclusion of the theorem applies.  

For example, the vertices could be $(0, p_1), (2p_2, 0), (2p_2+ 60000, p_1 + 60003)$ where $p_1, p_2$ are distinct odd primes with $(p_1+60003, 60000)=2$ and $(2p_2-3, 60003)=1$ (which occurs with a positive density for pairs of primes). 
\end{example}

When it comes to surfaces, some of those with Newton polyhedra given by pyramids will satisfy the conditions of the theorem. 

\begin{proposition} 
\label{prop: surface}
    Consider $f=1+y^a+z^b+y^{a}z^b + x^{-c}y^dz^{e}$ nondegenerate with respect to its Newton polyhedron in the situation of Section \ref{sec: pyramid}.  Then if  
    \[
2ab-2a-2b+2 \ge 2, 2abc > (72(2ab-2a-2b+2)^2+1)^2, 
\]
the Tannakian mondoromy group is $SL_R(\Qlb)$ or $GL_R(\Qlb)$. 
\end{proposition} 
\begin{proof}
    This follows directly from applying Theorem \ref{thm: A} with $\len(c)=3$ to the results from Section \ref{sec: pyramid}. 
\end{proof}

\subsection{A result of Gabber} 
In the higher dimensional cases, we only achieved partial results in the weight computations.  However, in the case where the dimension of the representation is a prime, a result of Gabber allows us to sometimes conclude big monodromy.  

Let $V$ be an $n$-dimensional vector space over an algebraically closed field of characteristic 0.  Recall the following result of Gabber. 

\begin{theorem}
\label{thm: gabber}
    \cite[Theorem 1.6]{Katz90} 
    Let $\mf{g}$ be a semisimple Lie-subalgebra of $\End(V)$ which acts irreducibly on $V$.  Suppose that $\dim V$ is a prime $p$.  Then $\mf{g}$ is either $\mf{sl}(2)$ in $\Sym^{p-1}(\std)$, or $\mf{sl}(V)$ or $\mf{so}(V)$ or, if $n=7$, possibly $\Lie(G_2)$ in the seven-dimensional irreducible representation of $G_2$.
\end{theorem}  

To apply this theorem, consider $\Lie(G)_{ss}$ of the Tannakian groups $G$ of interest, namely $\langle i_*\Qlb[n-1]\rangle$ where $i$ is the inclusion of the hypersurface.  The representation $\omega(i_*\Qlb[n-1])$ is an irreducible $G$-representation which we assume to have prime dimension.  We wish to show the corresponding Lie algebra representation remains irreducible, which is equivalent to the representation being irreducible restricted to the identity component of $G$.  In our case, this follows from the following sequence of results.  

    


    



\begin{lemma}
\label{lem: punctual}
    Let $N\in \mc{P}_{arith}(\G_m^n)$ be arithmetically semisimple with finite geometric convolution monodromy group.  Then $N$ is punctual.
\end{lemma}
\begin{proof}
    We induct on $n$; the base case of $n=1$ is given by \cite[Theorem 6.4]{KatzCE}.  First, we claim that $Rp_{n*}j_{n0!}N$ has finite geometric monodromy group as a perverse sheaf on $\G_m^{n-1}$.  Indeed, the fiber functor on $\G_m^{n-1}$ yields the fiber functor on $\G_m^{n}$ when composed with $Rp_{n*}j_{n0!}N$, so any automorphism of it applied to $Rp_{n*}j_{n0!}N$ gives an automorphism of the fiber functor applied to $N$.  Thus for $N$ to have finite monodromy, $Rp_{n*}j_{n0!}N$ must have finite monodromy.  By induction it must be punctual.  Then by taking $H^0$ we see that the pushforward of $N$ along the projection $p_n$ must be supported on a finite number of points, and thus $N$ must be supported on a finite number of copies of $\G_m$.  By induction again, this implies $N$ must be punctual as desired.
\end{proof}

\begin{lemma}
\label{lem: semi irred}
    Let $V$ be the standard representation of the Tannakian group $G$ corresponding to $i_*\Qlb[n-1]$, where $i$ is the inclusion of the hypersurface $X=V(F)\hookrightarrow \G_m^n$.  If $X$ is not stable under translation by some point $P\in \G_m^n$ not equal to the identity, then the semisimple part of the Lie algebra acts irreducibly.
\end{lemma}
\begin{proof}
    Let $G^0$ be the identity component of $G$.  Then if $V$ is an irreducible representation of $G$, it must be a semisimple representation of $G^0$, and we claim that it suffices to check that it is in fact an irreducible representation of $G^0$.  Indeed, because $G^0$ is connected and has a finite semisimple representation, it is reductive \cite[Theorem 22.138]{MilneAG}.  This implies that it is the almost direct product (i.e. quotient of the direct product by a finite group) of a central torus and a semisimple group  \cite[Theorem 22.125]{MilneAG}.  Then if $V$ is an irreducible representation of $G^0$, the torus part acts by scalars, so $V$ must be irreducible restricted to $G^0$.  Then by the Lie correspondence, the semisimple part of the Lie algebra also acts irreducibly as desired.
    
    If $V$ is a reducible representation of $G^0$, then let $L = \Hom_{G^0}(V, V)$; this will then have dimension greater than 1 and strictly contain $\Hom_G(V, V)$, which has dimension 1.  Note that the $G$-action on $L$ factors through $G/G^0$, which is finite.  
    Thus the corresponding perverse sheaf, $\mc{F}_L$, has finite geometric convolution monodromy group.  Moreover, $L$ can be defined over some finite field, so by applying Lemma \ref{lem: punctual}, it is punctual.   This gives a nontrivial morphism $L\otimes V\rightarrow V$.  Note that the convolution of a constant sheaf on a hypersurface with a skyscraper sheaf is given by the translation of the hypersurface by the point defining the skycraper sheaf.  Thus there must be some point $P$ not equal to the identity for which there is a nontrivial morphism $[x\mapsto Px]_*i_*{\Qlb} \rightarrow i_*{\Qlb}$.  The locus where such a morphism is nonzero is open, and since $X$ is assumed to be irreducible, this implies that $X$ must be stable under translation by $P$. 
\end{proof}

Now we wish to apply Theorem \ref{thm: gabber} to the scenario of Section \ref{sec: truncated} where we dealt with truncated prisms.  Note that in this case the hypersurface is not invariant under non-trivial translation, because the existence of the corners in the Newton polyhedron ensure that translation by any nontrivial constant in $\G_m^n$ will not preserve the equation of the hypersurface up to a constant.  Then by \ref{lem: semi irred}, it remains to rule out the case that $\mf{g}_{\sss}=\mf{sl}(2)$ acting in $\Sym^{p-1}(\std)$. 
 Because $G$ normalizes its Lie algebra, in this case the Frobenius element must be in the normalizer of $\mf{sl}(2)$.  Thus the Frobenius is contained in the image of $\GL_2$ in $\Sym^{p-1}(\std)$, which is isomorphic to $\GL_2/\mu_{p-1}$.  This implies that the absolute value of the eigenvalues of the Frobenius are all the same or all different, which does not occur in the case of the truncated prisms computed in the previous section.  Thus in this case, we can conclude that $\mf{g}_{\sss}=\mf{sl}(V)$ or $\mf{so}(V)$.  

\begin{corollary}
\label{cor: prism monodromy}
    If the Newton polyhedron of a hypersurface $X$ is a truncated prism and the dimension is a prime, then the convolution monodromy group contains either $SL(V)$ or $SO(V)$. 
\end{corollary}

\section{Big monodromy from big convolution monodromy} 
\label{sec: big mon}
In this section we aim to show that geometric monodromy groups associated to a variety are large if the convolution monodromy group of the variety is large.  

\subsection{Fiber functors through characters and the arithmetic Tannakian group} 
For our calculations of monodromy groups, we are interested in restricting to the Tannakian subcategories generated by a single perverse sheaf.  While the fiber functor we have constructed can be used for both these subcategories as well as the entire category, we can construct more fiber functors for the subcategories through the use of characters of the fundamental group.  The basics of this theory have already been worked out in \cite[Chapter 2]{FFK}.  We briefly review the results here. 

We recall that Theorem \ref{thm: generic vanishing} implies that for a fixed perverse sheaf $M$, most characters have nice cohomological properties with respect to $M$.  We would like to say more, namely that they define a fiber functor on $\langle M\rangle$.  This is given by the following result.

\begin{theorem}
\label{thm: chf}
    \cite[Theorem 3.26]{FFK} If $G$ is a torus, an abelian variety or a unipotent group over a finite field $k$, then any semisimple
object of $\bP_{\intt}(G)$ is generically unramified. 
 That is, the subset of characters $\chi\in \widehat{G}$ such that the functor 
 \[
N\mapsto w_{\chi}(N) = H^0(G_{\bar{k}}, N_{\chi}) 
 \] 
 is a fiber functor is generic.
\end{theorem}
\begin{remark}
    The proof the theorem above shows that the set of $\chi$ appearing in the above theorem is contained within the set satisfying the generic vanishing theorem.
\end{remark}

Note that in the theorem above, we are working with perverse sheaves on $G_{\bar{k}}$.  However, the same theory holds over finite fields in that $N\mapsto H^0(G_{\bar{k}}, N_{\chi})$ gives a fiber functor on the Tannakian category of perverse sheaves over negligible sheaves on $G_k$ becasue all the necessary properties can be checked over the algebraic closure; see e.g. \cite[Lemma 2.5]{finiteness}. 

Fix a character $\chi$ of $\pi_1^{\et}((\G_m^n)_{\bar{k}})$.  If $\chi$ is $\Gal(\bar{k}/k)$-invariant, then it canonically descends to a character of $\pi_1^{\et}((\G_m^n)_k)$, and thus $\L_{\chi}$ also descends (see \cite[Definition 2.4]{finiteness} for details).

We define $\mc{P}^\chi,  \mc{P}_{\kbar}^\chi$ to be the categories of perverse sheaves on $\G_m^n$ and $(\G_m^n)_{\kbar}$ satisfying Theorem \ref{thm: chf}, respectively.  We let $\mc{N}^\chi,\mc{N}_{\kbar}^\chi$ denote $\mc{P}^\chi\cap \mc{N}$ and $\mc{P}_{\kbar}^\chi\cap \mc{N}_{\kbar}$ respectively.  Our goal now is to show that if we restrict to the subcategories defined by the pure perverse sheaves, we still obtain Tannakian categories with the geometric one being a normal subgroup of the arithmetic one.  For abelian varieties over fields of characteristic 0 this was achieved in Lemmas 2.7 - 2.9 of \cite{finiteness}.  Most of the same arguments work in our setting, except that we use pure perverse sheaves instead of the geometrically semisimple ones used in \cite{finiteness} because we do not have access to Kashiwara's conjecture, which has only been proved over characteristic 0. 

 \begin{lemma}
     The full subcategory of $\mc{P}^{\chi}/\mc{N}^{\chi}$ consisting of pure sheaves is a Tannakian subcategory of $\mc{P}^{\chi}/\mc{N}^{\chi}$.  Furthermore, the full subcategory of $\mc{P}_{\overline{k}}^{\chi}/\mc{N}_{\overline{k}}^{\chi}$ consisting of pullbacks of pure sheaves in $\mc{P}^{\chi}/\mc{N}^{\chi}$ is a Tannakian subcategory.
 \end{lemma}
 \begin{proof}
     All the axioms of Tannakian categories are straightforward  to check as in the proof of Lemma 2.7 on \cite{finiteness}, except that the convolution of two pure sheaves in $\mc{P}^{\chi}/\mc{N}^{\chi}$ is pure.  That is, if $M$ and $N$ are pure, then $M*N$ is equivalent to a pure sheaf modulo a negligible sheaf.  By Lemma 3.8 of \cite{FFK}, (we can pass to the algebraic closure), the cone of $M *_! N\rightarrow M*_*N$ is negligible.  Now if $p_1^*M\otimes p_2^*N$ is pure of weight $n$, then by Deligne's theorem $m_!(p_1^*M\otimes p_2^*N)$ is mixed of weights at most $n$ while $m_*(p_1^*M\otimes p_2^*N)$ is mixed of weights at least $n$, so we see that $M*N$ is pure as desired. 
 \end{proof}

    Now let $G_k$ be the convolution monodromy group of the full subcategory of $\mc{P}^\chi/\mc{N}^\chi$ consisting of pure perverse sheaves on $(\G_m^n)_{k}$ and let $G_{\bar{k}}$ be the convolution monodromy group of the full subcategory of $\mc{P_{\kbar}}^\chi/\mc{N_{\kbar}}^\chi$ consisting of summands of pullbacks to $(\G_m^n)_{\kbar}$ of pure perverse sheaves on $(\G_m^n)_{k}$.  Let $\Gal_k$ be the Tannakian group of the category of $p$-adic Galois representations over $k$.   Then the analogues of Lemmas 2.8 and 2.9 of \cite{finiteness} follow using the same proof, replacing the abelian variety $A$ with the torus $\G_m^n$ and geometrically semisimple perverse sheaves with pure perverse sheaves. 

\begin{lemma}
\label{lemma: normal}
    [Analogue of Lemma 2.8 in \cite{finiteness}] The group $G_{\overline{k}}$ is a normal subgroup of $G_k$, with quotient $\Gal_k$.  
\end{lemma}

The arithmetic convolution monodromy group $G_k$ is given by automorphisms of the fiber functor sending $K$ to $H^0((\G_m)^n_{\kbar}, K\otimes \L_\chi)$.  The Galois group acts on this functor and can thus be realized inside $G_k$.  By the previous lemma it is in the normalizer of $G_{\kbar}$.  Furthermore, the commutator subgroup of the identity component of $G_{\kbar}$ is a characteristic subgroup, so the Galois group normalizes it as well. 

\begin{lemma}
\label{lemma: bart}
    [Analogue of Lemma 2.9 in \cite{finiteness}]
    Let $\chi$ be a $\Gal(\overline{k}/k)$-invariant character of $\pi_1^{\et}((\G_m)^n_{\bar{k}})$ and let $\L_{\chi}$ be the associated character sheaf.  

    Let $K$ be a geometrically semisimple perverse sheaf on $A$ such that $H^i((\G_m)^n_{\kbar}, K\otimes \L_\chi)$ vanishes for $i\neq 0$.  Then the action of $\Gal(\kbar/k)$ on $H^0((\G_m)^n_{\kbar}, K\otimes \L_\chi)$ normalizes the geometric convolution monodromy group of $K$ and the commutator subgroup of its identity component.  
\end{lemma} 

This provides a link between the geometric and convolution monodromy groups, which we will exploit in the remainder of this section. 

\subsection{Setup}
Let $X$ be a variety and let $Y\subseteq X\times \G_m^n$ be a family over $X$ of smooth hypersurfaces in $\G_m^n$, all defined over $\kbar$, the algebraic closure of some finite field.  Let $f:Y\rightarrow X$ and $g:Y\rightarrow \G_m^n$ be the projections. We would like to show that the geometric monodromy group of $R^{n-1}f_*(g^*\L_{\chi_i})$ is large, for appropriate characters $\chi_i$ of $\pi_1^{\et}(\G_m^n)$.  Then we will deduce the corresponding statement over $\C$. 

\begin{remark}
    In \cite{finiteness}, the authors do the same thing for $\bigoplus_{i=1}^cR^{n-1}f_*(g^*\L_{\chi_i})$ for a fixed $c\in \N$, because they use it in their version of the numerical conditions that come later in the paper.  In this paper we address the numerical conditions differently, and are thus content to set $c=1$.  However, the same result can be proven for arbitrary $c$ with minimal changes.
\end{remark}

Let $G$ be the convolution monodromy group of the perverse sheaf $K = i_*\Qlb[n-1]$ on $(\G_m^n)_{\overline{\eta}}$.  Let $G^*$ be the commutator subgroup of the identity component of $G$.  Then we may restrict the distinguished representation of $G$ arising from $K$ to $G^*$, and consider $N(G^*)$ to be the normalizer of $G^*$ inside the group of automorphisms of this representation.  Thus $N(G^*)$ is realized as a subgroup of $GL(H^0((\G_m)^n_{\overline{k}}, K\otimes \L_{\chi}))$. \\ 

We wish to prove the following theorem, which is an analogue of Theorem 4.7 in \cite{finiteness}. 

\begin{theorem}
\label{thm: blah}
    Assume that $Y_{\overline{\eta}}$ is not translation-invariant by any nonzero element of $\G_m^n$, that $G^*$ is a simple algebraic group with irreducible distinguished representation, and that $Y$ is not equal to a constant family of hypersurfaces translated by a section of $\G_m^n$. 

    Then for all characters $\chi$ of $\pi_1^{\et}(\G_m^n)$, avoiding some finite
set of tacs of $\G_m^n$, the
following conditions are satisfied:  
\begin{itemize}
    \item $R^kf_*(g^*\L_{\chi_i})=0$ for $k\neq n-1$, and 
    \item the geometric monodromy group of $R^{n-1}f_*(g^*\mc{L}_{\chi})$ contains $G^*$ as a normal subgroup, where its representation restricted to $G^*$ is isomorphic to the distinguished representation of $G^*$. 
\end{itemize}
\end{theorem} 
 
The first condition can be achieved through Theorem \ref{thm: generic vanishing}.  The next subsections are devoted to the second condition.

\subsection{Controlling the geometric monodromy group in characteristic $p$} 
The geometric monodromy group of $R^{n-1}f_*(g^*\L_{\chi})$, which we denote by $G_{\geo}$, is given by the Zariski closure of the image of $\Gal(\overline{\kbar(\eta)}/\kbar(\eta))$ acting on the stalk at the geometric generic point of $X$.  By Lemma \ref{lemma: bart}, the action of the Galois group on $H^0((\G_m)^n_{\overline{\kbar(\eta)}}, K\otimes \L_{\chi})$, which is indeed the distinguished representation of $G^*$, normalizes the image of $G^*$.  This implies that $G_{\geo}$ is contained in $N(G^*)$.  \\ 

To show $G_{\geo}$ contains $G^*$, we will combine this fact with Lemma 4.5 of \cite{finiteness}, which we recall below.  

\begin{lemma} 
[Lemma 4.5 of \cite{finiteness}]
    Let $G^*$ be a simple algebraic group and $c$ be a natural number.  Fix an irreducible representation of $G^*$.  Let $N(G^*)$ be the normalizer of $G^*$ inside the group of automorphisms of this representation.  Then there is a finite list of irreducible representations of $N(G^*)^c$ such that a reductive subgroup $B$ of $N(G^*)$ contains $G^*$ if and only if $B$ has no invariants on any of these representations. 
\end{lemma} 

To apply this to $G_{\geo}$, we will show that the Galois group has no invariants on all representations which $G^*$ acts nontrivially on.  We do this by first characterizing such representations as certain pullbacks in Lemma \ref{lemma: 4.1}, which is an analogue of Lemma 4.1 of \cite{finiteness}.  Then we show that the Galois group has no invariants on these in Lemma \ref{lemma: 4.4}, which is an analogue of Lemma 4.4 of \cite{finiteness}.  To do the former, we will need a result on the structure of perverse sheaves with Euler characteristic 1 on the product $X\times \G_m^n$.  The classification in the case of $\G_m^n$ is in terms of hypergeometric complexes, which we now briefly review. 

Given $\psi$, a nontrivial $\Qlb$-valued additive character of a finite subfield $k_0$ of $\bar{k}$, and tame $\Qlb$-valued characters of $\pi_1((\G_m)_{\bar{k}})$, one defines the corresponding hypergeometric sheaf 
\[\Hyp(!, \psi, \chi 's, \rho 's) \in D^b_c((\G_m)_{\bar{k}}, \Qlb).\]  
Given $\lambda\in \kbar^*$, we define the multiplicative translate $\Hyp_{\lambda}(!, \psi, \chi 's, \rho 's)\coloneqq [x\mapsto \lambda x]_*\Hyp(!, \psi, \chi 's, \rho 's)$.  These are perverse sheaves \cite[Theorem 8.4.5]{Katz90}.
 For a discussion of the construction and basic properties of these sheaves, we refer the reader to \cite[Chapter 8]{Katz90}.  Katz proves in \cite[Theorem 8.5.3]{Katz90} states that if $K$ is a perverse simple object of $D^b_c(\G_m, \Qlb)$ whose Euler characteristic is 1, then $K$ is hypergeometric.  
 
 This was generalized to higher-dimensional tori $T$ in \cite[Section 8]{Gabber}.  Write $H(\psi; \chi)\coloneqq \L_{\chi}\otimes j^*\L_{\psi}[1])$ to be the hypergeometric sheaf on $\G_m$, which in Katz's notation would be written as $\Hyp(!, \psi; \chi; \emptyset)$.  Let $\delta_{\lambda}$ be the punctual sheaf $i_{\lambda *}\Qlb$ corresponding to the inclusion of a closed point $\lambda$ on $T$.  Given morphisms of tori $\pi_i\colon \G_m\rightarrow T$ for $1\le i\le m$, define the following hypergeometric complexes 
 \[
\Hyp(\mi, \lambda, (\chi_i), (\pi_i))\coloneqq \delta_{\lambda}*_{\mi} R\pi_{1*}(H(\psi; \chi_1))*_{\mi}\cdots *_{\mi}R\pi_{m*}(H(\psi; \chi_m)).
 \]

 Finally, we let $\Hyp_{\mi}(T)$ denote the category of $\mi$-hypergeometric complexes, those objects of $D^b_c(T, \Qlb)$ isomorphic to one of the $\Hyp(\mi, \lambda, (\chi_i), (\pi_i))$.
 This allows us the state the following classification result. 
 \begin{theorem} 
 \label{thm: gab}
     \cite[Theorem 8.2]{Gabber} Let $T$ be a torus over an algebraically closed field $k$ of characteristic $p$.  An object $A$ of $\Perv(T, \Qlb)$ is irreducible and satisfies $\chi(T, A)=1$ if and only if it is isomorphic to an object of $\Hyp_{\mi}(T)$.
 \end{theorem}

\begin{lemma}
\label{lemma: section}
    Let $X$ be a variety over an algebraically closed field $k$ of characteristic $p$ and let $K$ be an irreducible sheaf on $X\times \G_m^n$ which restricts to a perverse sheaf on $\G_m^n$ of dimension 1 at each $x\in X$.  Then for a nonempty open subset $U\subset X$, the restriction of $K$ to $U$ is a translate by a section of a pullback of a sheaf on $\G_m^n$.
\end{lemma}
\begin{proof}
    We first consider the case $n=1$.  Let $K_{\bar{k(X)}}$ denote the pullback of $K$ to the geometric generic fiber $(\G_m)_{\bar{k(X)}}$; denote.  If we show that $K_{\bar{k(X)}}$ is a translate of a pullback from $\G_m$, then the same must hold over an open set because these are in the derived category of constructible sheaves.  
    By Theorem \ref{thm: gab}, $K_{\bar{k(X)}}$ is of the form $\Hyp(\mi, \lambda, (\chi_i), (\pi_i))$ where $\lambda\in \bar{k(X)}^*$, $\psi$ is an additive character of a finite subfield, and the $\chi_i$ are tame characters of $\pi_1((\G_m)_{\bar{k(X)}})$.  
    For this to be a pullback translated by a section, it suffices to know that the characters are defined over $k$, and that the translate by $\lambda$ gives the desired section.  
    The additive character is defined over a finite subfield by definition, and the tame fundamental group of $\G_m$ is the same over $k$ and $\bar{k(X)}$.  Next, we need to know that $\lambda$ is actually contained in $k(X)$.  Note that the pullback of $K$ to the geometric generic fiber factors through $(\G_m)_{\overline{k}(X)}$, and thus the Galois group $\Gal(\overline{k(X)}/k(X))$ must act trivially on $K_{\bar{k(X)}}$.  Because the characters are invariant under this Galois action that implies that the translate of the sheaf by any Galois conjugate of $\lambda$ has to be the same.  Thus either the sheaf is constant in which case there is nothing to prove, or $\lambda\in k(X)$ as desired.  Then $\lambda$ defines the desired section over the open subset on which it is a regular function.  

\end{proof}

Now we prove an analogue of Lemma 4.1 of \cite{finiteness}.
\begin{lemma}
    \label{lemma: 4.1} 
    Assume that $G^*$
is a simple algebraic group with irreducible distinguished representation, and that Y is not equal to a constant family of hypersurfaces
translated by a section of $\G_m^n$.
Let $K_0$
be an irreducible perverse sheaf on $(\G_m^n)_{\overline{\eta}}$ in the Tannakian category generated
by $K$. Assume that $K_0$
is a pullback from $(\G_m^n)_{\Fqb}$ to $(\G_m^n)_{\overline{\eta}}$ of a perverse sheaf on $(\G_m^n)_{\Fqb}$. Then
$G^*$ acts trivially on the irreducible representation of $G$ corresponding to $K_0$.
\end{lemma}
\begin{proof}
    The proof of Lemma 4.1 in \cite{finiteness} carries over until the last paragraph, replacing the abelian variety $A$ with $\G_m^n$ and working over $\Fqb$ rather than over $\C$.  Using the same notation, by arguing by contradiction we have in our situation that $K_2* \mc{F} = K_1$ where $\mc{F}$ is a perverse sheaf with Euler characteristic 1.  By Lemma \ref{lemma: section}, for some non-empty open set $U\subset \G_m^n$ we have that $\mc{F}$ restricted to $X\times U$ is given by a complex on a  that is a translate by a section of $U$.  Then for any two points $x_1, x_2\in U$, we have that the perverse sheaf given by the inclusions of $Y_{x_1}$ and $Y_{x_2}$ are both given by the convolution of $Y$ restricted to another point in $U$, a fixed sheaf, and a section of $U$.  Thus looking at the supports, $Y_{x_1}$ and $Y_{x_2}$ are related by a translate by that section. This implies that over an open set, the image of $Y$ in $X\times \G_m^n$ lies in an orbit of a single hypersurface under the action of $\G_m^n$ by translation.  It thus suffices to show that this orbit is closed.  The closure of such an orbit in the space where we consider all hypersurfaces with the given Newton polyhedron (not just the ones nondegenerate with respect to their Newton polyhedron) will be a union of torus orbits.  All other torus orbits have lower dimension and thus have positive-dimensional stabilizers.  Such hypersurfaces have newton polyhedra with empty interiors, so they cannot appear in the closure.  Thus the original orbit is indeed closed and $Y$ is globally a translate by a section of $\G_m^n$, giving the desired contradiction. 
\end{proof}

Next we give an analogue of Lemma 4.4 of \cite{finiteness}).  

\begin{lemma}
\label{lemma: 4.4}
    Let $K\in D^b_c((\G_m)^n_{\overline{k}(\eta)}, \Qlb)$ be a perverse sheaf of geometric origin.  If no irreducible component of $K$ is a pullback from $A_{\kbar}$, then for all characters $\chi:\pi_1^{\et}((\G_m)^n_{\kbar(\eta)})\rightarrow \Qlb^\times$ outside a finite set of torsion translates of proper subtori of $\Pi(A)$, we have $H^i((\G_m)^n_{\overline{\kbar(\eta)}}, K\otimes \mc{L}_\chi)=0$ for $i\neq 0$ and $\left(H^0((\G_m)^n_{\overline{\kbar(\eta)}}, K\otimes \L_{\chi})\right)^{\Gal(\overline{\kbar(\eta)}/\kbar(\eta)}=0$.  
\end{lemma} 
\begin{proof}
    The proof of Lemma 4.4 of \cite{finiteness} applies here, except that we must replace a few arguments for abelian varieties with those for tori.  First, the generic vanishing theorem for abelian varieties used can be replaced with the corresponding result for tori, which is Theorem \ref{thm: generic vanishing}, to prove the first claim. 

    For the second claim, as in the proof of Lemma 4.4 of \cite{finiteness} we may assume $X$ is smooth by restricting to an open subset, say of dimension $m$, and spread $K$ out to a sheaf $K'$ over $\G_m^n\times X$ such that $K'[m]$ is perverse.  Let $\pi\colon \G_m^n\times X\rightarrow \G_m^n$ and $\rho\colon \G_m^n\times X\rightarrow X$ be the projections.
    
    By the proper base change theorem, the stalk of this complex at the generic point is given by $H^*((\G_m^n)_{\bar{\C(\eta)}}, K\otimes \L_{\chi})$, which by assumption is supported in degree 0.  Thus on an open set $R\rho_*(K'\otimes \L_{\chi})$ is lisse.  The Galois action matches the monodromy action on this subset, so if the Galois invariants on the stalk are nonzero then there is some monodromy-invariant complex $\mc{F}^0\subset R\rho_*(K'\otimes \L_{\chi})$.

    The key in the proof of Lemma 4.4 of \cite{finiteness} is to use that $\mc{F}^0$ is a direct summand, so that it produces a nonzero global section of $H^0(X, R\rho_*(K'\otimes \mc{L}_{\chi}))$.  This is achieved using the decomposition theorem there, which we do not have access to here because $\rho$ is not proper in this scenario.  Instead, we use the following argument. 

     By \cite[Theorem 2.11]{FFK},  for generic $\chi$ we have that $R\rho_*(K'\otimes \L_{\chi}) = R\rho_!(K'\otimes \L_{\chi})$ and is perverse.  Furthermore, since $K$ is of geometric origin and Deligne's theorems on the behavior of weights under $R\rho_*$ and $R\rho_!$, this complex is also pure and thus geometrically semisimple.  Since we are working over an algebraically closed field, this complex is indeed a sum of shifts of irreducible perverse sheaves, as desired.  The rest of the proof of Lemma 4.4 of \cite{finiteness} carries over.
\end{proof} 

Finally, we recall the statement of Lemma 4.5 of \cite{finiteness}.   
\begin{lemma}
    \label{lemma: 4.5} Let $G^*$ be a simple algebraic group and $c$ a natural number. Fix an irreducible representation of $G^*$.  Let $N(G^*)$ be the normalizer of $G^*$ inside the group of automorphisms of this representation.  Then there is a finite list of irreducible representations of $N(G^*)^c$ such that a reductive subgroup $B$ of $N(G^*)^c$ contains $(G^*)^c$ if and only if $B$ has no invariants on any of these representations. 
\end{lemma}

We will use this with $c=1$.  We can now prove Theorem \ref{thm: blah}. 
\begin{proof}
    [Proof of Theorem \ref{thm: blah}]  As stated earlier, Theorem \ref{thm: generic vanishing} ensures that for $\chi$ avoiding a finite set of tacs, we have $R^kf_*(g^*\mc{L}_{\chi})=0$ for $k\neq n-1$.  We now only consider such $\chi$.  

    Call this geometric monodromy group $G'$.  Recall that the geometric monodromy group of a sheaf is given by the Zariski closure of the Galois group of the function field of the generic point $k$ acting on the stalk at the geometric generic point.  
    Because $\chi$ is chosen to satisfy the generic vanishing theorems, by proper base change the stalk of $R^{n-1}f_*(g^*\mc{L}_{\chi})$ at the generic point is given by $H^0((\G_m)^n_{\bar{k(X)}}, K\otimes \L_\chi)$ where $K$ is the perverse sheaf given by the inclusion of the hypersurface into $(\G_m)_{\kbar}^n$ in degree $-n+1$. 
    This stalk can be interpreted as the arithmetic fiber functor and we are thus in the setting of Lemma \ref{lemma: bart}.
    Both $\Gal(\overline{k(X)}/k(X))$ and $G$ act on it, and by the lemma the former normalizes $G^*$, so $G'\subset N(G^*)$ as desired.   

    By Deligne's theorem, $G'$ is reductive.  Thus by Lemma \ref{lemma: 4.5}, there is a finite list of irreducible representations of $N(G^*)$ such that if the action of $G'$ on these representations has no invariants, then $G'$ contains $G^*$.  
    Because $i_*\Qlb[n-1]$ is pure, we can apply Lemmas \ref{lemma: normal} and \ref{lemma: bart}.  
    Then the action of $G_k$ on the associated representation factorizes through the normalizer of $G_{\overline{k}}$, and thus through $N(G^*)$.  
    Thus all representations of $N(G^*)$ correpsond to pure perverse sheaves in the Tannakian subcategory $\langle i_*\Qlb[n-1]\rangle $ of the arithmetic Tannakian category of pure perverse sheaves on $(\G_m)_{k, \eta}^n$ modulo negligible sheaves. 

    Because $G^*$ acts nontrivially on these representations, by Lemma \ref{lemma: 4.1}, none of these is a pullback from $(\G_m^n)_{\kbar}$.  Then by Lemma \ref{lemma: 4.4} the Galois group has no invariants on them outside some finite set of tacs, as desired.  
\end{proof}

\subsection{Lifting to characteristic 0 and application to an orbit of characters}

Now that we have a way of proving a big geometric monodromy result from a big convolution monodromy group over characteristic $p$, we will now lift this to characteristic 0.  We recall the following result. 
\begin{theorem}
    \cite[Theorem 8.18.2]    {Katz90}
    \label{thm: sp} 
 Let $S$ be a normal connected Noetherian scheme with generic point $\eta$, $X/S$ a smooth $S$-scheme with geometrically connected fibers, $l$ a prime number and $\mc{F}$ a lisse $\Qlb$-sheaf of rank $n\ge 1$ on $X$.  For each geometric point $s$ in $S$, define 
    \[
\Gamma(s)\coloneqq \text{ the image of } \pi_1(X_s) \text{ in } \GL(n, \Qlb).
    \] 
    Then 
    \begin{enumerate}
        \item the group $\Gamma(s)$ decreases under specialization, in the sense that if $t$ is a specialization of $s$, then $\Gamma(t)$ is conjugate in $\GL(n, \Qlb)$ to a subgroup of $\Gamma(s)$. 
        \item there exists an open neighborhood $V$ of $\eta$ in $S$ such that for any geometric point $s\in V$, $\Gamma(s)$ is conjugate in $\GL(n, \Qlb)$ to $\Gamma(\bar{\eta})$. 
    \end{enumerate}
\end{theorem} 

Now we consider the case when all varieties are defined over $\C$, and thus over some finite-type $\Z$-algebra $A$.  We may recover the previous situation by taking the specialization and base change to a finite field $\Spec A/\m$. 

\begin{corollary}
    Consider the situation of Theorem \ref{thm: blah} except with all varieties defined over $\C$.  Say this gives rise to the previous situation through reduction modulo a prime, and recall that we have defined $G^*$ through the convolution monodromy group in characteristic $p$.  Then for all characters $\chi$ of $\pi_1^{\et}(\G_m^n)$, avoiding some finite set of tacs of $\G_m^n$, the geometric monodromy group of $R^{n-1}f_*g^*\L_{\chi}$ contains a conjugate of $G^*$ as a normal subgroup.  
\end{corollary}
\begin{proof}
     Now all varieties are defined over $\C$, and thus over some finite-type $\Z$-algebra $A$.  It remains to show that the geometric monodromy group $G_{\geo, \C}$ of $R^{n-1}f_*(g^*\L_{\chi})$ contains conjugate of $G^*$ as a normal subgroup.  If we take the specialization and base change to a finite field $\Spec A/\m$, then we have shown that the resulting geometric monodromy group $G_{\geo}$ contains $G^*$ as a normal subgroup.  By Theorem \ref{thm: sp}, for some choice of $\m$ we have that $G_{\geo}$ and $G_{\geo, \C}$ are conjugate, so $G_{\geo, \C}$ contains a conjugate of $G^*$ as a normal subgroup as desired.    
\end{proof}

Finally, we will need to apply this statement to a set of characters comprising a Galois orbit.  Begin with $\G_m^n$ defined over a number field $K$, $X$ a smooth scheme over $\Q$, and $Y\subseteq X_K\times_K \G_m^n$ a family of hypersurfaces nondegenerate with respect to their fixed Newton polyhedron $\Delta$.  For an embedding $\iota\colon K\rightarrow \C$, we can define $(\G_m)^n_{\iota}$ and $Y_{\iota}$ as schemes over $\C$ through the embedding $\iota$. 

 We can now apply the previous results to these schemes.  Ensuring that the previous result holds for a $\Gal(\bar{K}/K)$-orbit of characters is done in Lemma 4.8, Lemma 4.9, and Corollary 4.10 of \cite{finiteness} for abelian varieties, but the exact same arguments apply for tori.    

\begin{corollary}
    \label{cor: big} Assume that $Y_{\overline{\eta}}$ is not translation-invariant by any nonzero element of $\G_m^n$, that $G^*$ is a simple algebraic group with irreducible distinguished representation, and that $Y$ is not equal to a constant family of hypersurfaces translated by a section of $\G_m^n$. 

    Then for any primes $l$ and $p$, there is an embedding $\iota\colon K\hookrightarrow \C$ and a torsion character $\chi$ of $\pi_1^{\et}((\G_m^n)_{\iota})$ of order a power of $l$ such that for all $\Gal(\bar{K}/K)$-conjugates $(\iota', \chi')$ of $(\iota, \chi)$, the geometric monodromy group of $R^{n-1}f_*(g^*\L_{(\iota', \chi')})$ contains a conjugate of $G^*$ as a normal subgroup. 
\end{corollary}


\begin{remark}
    We will apply this result to the cases where we prove $G^*$ to be big, so that $G^*$ is the only conjugate of itself. 
\end{remark}

\section{Constructing Hodge-Deligne systems} 
\label{sec: HD} 

\subsection{Summary}
Let $X$ be a smooth variety over $\Q$ and let $\G_m^n$ denote the $n$-dimensional torus over $K$.  Let $Y\subseteq X_K\times_K \G_m^n$ be a subscheme that is smooth and flat over $X$.  There is a finite set $S$ of places of $K$ lying over the places $S'$ of $\Q$ such that $Y$ spreads out to $f: \mc{Y}\hookrightarrow \mc{X}\times_{\Z[1/S]} \G_m^n$. \\ 

The goal of this section is to construct Hodge-Deligne systems, originally defined in Section 5 of \cite{finiteness}, associated to $X$ and $Y$ that will encode the various cohomological properties necessary to pass from big monodromy statements to finiteness.  In particular, we will define the pushforward of a Hodge-Deligne system under a smooth morphism which is not necessarily proper.  We will apply it to Hodge-Deligne systems arising form character sheaves $\L_\chi$.  This will allow us to construct a Hodge-Deligne system $V_I$ on $\mc{X}$ associated to an orbit of good characters.    Furthermore, we will prove a bound involving the dimension of the centralizer of the crystalline Frobenius of these Hodge-Deligne systems, which we will need for numerical conditions in the final argument. 

As in \cite{finiteness}, the purpose of these Hodge-Deligne systems is to relate the cohomological avatars of a variety, and in particular, the associated complex and $p$-adic period maps.  As originally worked out in \cite{LV}, this allows one to translate a big monodromy statement into finiteness results.  After constructing these Hodge-Deligne systems we will be able to cite -- with slight variations -- the results that facilitate such a translation.

\subsection{Review of Hodge-Deligne systems} 
Given a smooth variety over a number field $K$, we let $S$ be a finite set of primes of $K$ such that there is a smooth model $\mc{X}$ of $X$ over $\O_K[1/S]$.  Let $v$ be an unramified place of $K$ outside of $S$.    A Hodge-Deligne system on $\mc{X}$ at $v$ consists of: 
\begin{itemize}
    \item A singular local system $V_{\sing}$ of $\Q$-vector spaces on $X_{\C}$. 
    \item An \'etale local system $V_{\et}$ of $\Q_p$-vector spaces on $\mc{X}_{\et}\times_{\Spec \Z}\Spec \Z[1/p]$. 
    \item A vector bundle $V_{\dR}$ on $X$, an integrable connection $\nabla$ on $V_{\dR}$, and a descending filtration $\Fil^iV_{\dR}$ of $V_{\dR}$ by subbundles 
    \[
V_{\dR} = \Fil^{-M}V_{\dR} \supseteq \Fil^{-M+1}V_{\dR} \supseteq\cdots \supseteq \Fil^M V_{\dR} = 0  
    \]
    (not necessarily $\nabla$-stable), each of which is locally a direct summand of $V_{\dR}$. 
    \item A filtered F-isocrystal $V_{\cris}$ on $X_{K_v}$. 
\end{itemize} 

Furthermore, these satisfy various isomorphisms and axioms that one would expect from the constant Hodge-Deligne systems and their pushforwards along smooth proper morphisms; see \cite[Examples 5.8, 5.10]{finiteness}.  The full definition of a Hodge-Deligne system is given in \cite[Definition 5.2]{finiteness}. 

We will now sketch a construction of the associated Hodge-Deligne system $\mc{L}_\chi$ on $(\G_m^n)_L$.  In the following subsections we will construct Hodge-Deligne systems of the form $R^{n-1}f_*g^*\L_{\chi}$ on $\mc{X}$, which is what we will eventually use.  These will be constructed without explicit reference to the Hodge-Deligne systems $\L_{\chi}$ constructed here, morally we will consider them as the pushforwards of the pullbacks of these $\L_{\chi}$.  We will thus provide the description of the various realizations of $\L_{\chi}$ here without going into the details of their comparisons.  

\begin{example}
    
Let $\chi\in \pi_1^{\et}((\G_m^n)_K)$ be a character of order $r$; say it is defined over a field extension $L/K$.  We say that it is good if it satisfies the conditions of Theorem \ref{thm: blah}, and that $H^i(\G_m^n, \O_{\mc{Y}} \otimes \mc{L}_\chi) = H^i_c(\G_m^n, \O_{\mc{Y}} \otimes \mc{L}_\chi)$.  Our goal is to 

Such data includes a singular local system $V_{\sing}$, an \'etale local system $V_{\et}$, a filtered vector bundle with integrable connection $V_{\dR}$, and a filtered $F$-isocrystal $V_{\cris}$.  The first two come from the correspondence between characters of the (\'etale) fundamental group and (\'etale) local systems.  The vector bundle with integrable connection arises from the Riemann-Hilbert correspondence.  For the crystalline realization, we will instead use an F-unit isocrystal using the following result due to Tsuzuki, Katz, and Crew. 

\begin{theorem}
    [Theorem 3.9 of \cite{Kedlaya}] The category of unit-root objects in $\FI(X)$ is equivalent to the category of \'etale $\Q_p$-local systems on $X$.  The unit-root objects in $\oFI(X)$ form a full subcategory of $\FI(X)$ corresponding to the category of potentially unramified \'etale $\Q_p$-local systems on $X$. 
\end{theorem}

Because $\chi$ has finite order, the theorem applies to $\L_\chi$, and thus we obtain a unit-root overconvergent F-isocrystal $V_{\cris}$.  
\end{example}

There is another important notion introduced in conjunction with Hodge-Deligne systems in \cite[Section 5.3]{finiteness}, namely $H^0$-algebras.  These are essentially defined as algebra objects in the category of Hodge-Deligne systems; see \cite[Definition 5.14]{finiteness}.  Concretely, an $H^0$-algebra $\msf{E}$ over $\Spec \O_{K, S}$ can be described as follows.  
\begin{itemize}
    \item $\msf{E}_{\sing}$ is a $\Q$-algebra,
    \item $\msf{E}_{\dR} = E\otimes_Q K$,
    \item $\msf{E}_{\et} = E\otimes_Q \Q_p$,
    \item $\msf{E}_{\cris} = E\otimes_Q K_v$.
\end{itemize} 

In what follows, we will construct an $H^0$-algebra $\msf{E}_I$ on $\Spec \O_{K, S}$ for each orbit $I$ of characters and an $\msf{E}_I$-module $\msf{V}_I$ on $\mc{X}$.  

\subsection{Construction of a Hodge-Deligne system $R^{n-1}f_*g^*\L_\chi$ on $\mc{X}$} 
Given the composition $g\colon \mc{Y}\hookrightarrow \mc{X}\times \G_m^n\rightarrow \G_m^n$ with $f\colon \mc{Y}\rightarrow \mc{X}$ the universal family of hypersurfaces nondegenerate with respect to their fixed Newton polyhedron, we would like to construct Hodge-Deligne systems on $\mc{X}$ at $v$, where $v$ is a place of $K$ lying over a prime not below any place of $S$.  We will first construct one for each character $\chi$ of $\pi_1^{\et}(\G_m^n)$ corresponding to $R^{n-1}f_*g^*\L_{\chi}$.  

\subsubsection{Expressing cohomology objects as eigenspaces of the cohomology of smooth proper varieties}
Because the morphism $f\colon \mc{Y}\rightarrow \mc{X}$ is not proper, it is not clear that building the cohomology objects directly will give us desirable properties.  However, in our situation we can express these cohomology objects as the eigenspaces of the cohomology of smooth proper varieties. \\ 

 Given a fixed Newton polyhedron $\Delta$, there is a compactification of the torus $\G_m^n\hookrightarrow Z_\Delta$ introduced in \cite{Khovanskii1977} such that for any hypersurface $Y_x\subset \G_m^n$ nondegenerate with respect to its Newton polyhedron $\Delta$, the closure $\bar{Y_x}$ is smooth and transversally intersecting on the torus orbits of $Z_\Delta$.  Moreover, these hypersurfaces are sch\"on in the sense of \cite{Jenia}, which means the multiplication map of their compactifications are smooth. \\ 

For any $x\in \mc{X}$, let $m$ be a sufficiently large integer divisible by $r$ and consider $[m]^{-1}Y_x \subset \G_m^n$ and its compactification $\bar{[m]^{-1}Y_x} \subset Z$.  Call the boundary divisors $D = Z - \G_m^n$ and $D_x = \bar{[m]^{-1}Y_x} - [m]^{-1}Y_x$.  Since $[m]^{-1}Y_x$ carries an action of $\mu_m^n$, so does $H^{n-1}_{\et}([m]^{-1}Y_x, \L_\chi)$ by functoriality.  Given a tuple $\lambda\in (\Z/m\Z)^n$ and a vector space $V$ with an action of $\mu_m^n$, let $V_\lambda$ denote the corresponding eigenspace. \\ 

Note that as $m$ is taken to be sufficiently large and divisible by $m$, each $\chi$ determines such a $\lambda$ subject to the following commutative diagram.  

\[\begin{tikzcd}
	{\pi_1(\G_m^n)} & {\mu_m^n} \\
	& {\C^*}
	\arrow[two heads, from=1-1, to=1-2]
	\arrow["\chi"', from=1-1, to=2-2]
	\arrow["\lambda", from=1-2, to=2-2]
\end{tikzcd}\] 

Consider the torus action on each of the components $D_i$ of $D$; each of them must have a positive-dimensional stabilizer $G_i\subset \G_m^n$.  Note that $\mu_m^n$ is also a subset of this torus action $\G_m^n$; denote their nontrivial intersection by $H_i$.  Then for a generic choice of $\lambda\in (\Z/m\Z)^n$, we have that the elements $h_j\in \mu_m^n$ of $H_i$ do not all satisfy $(h_j^{\lambda_j})=1$.  But $h_j$ acts trivially on $D_i$.  Since $\lambda$ is generic we can pick one for which this holds for all components $D_i$.  Then $\lambda$ cannot be an eigenvalue for the action of $\mu_m^n$ on any $D_x$, so $H^{n-1}_{\et}(D_x, \Q_p)_\lambda=0$. \\ 

A similar argument can be made for other cohomology theories.  Furthermore, in many cohomology theories there is an excision exact sequence that relates the cohomology of a proper variety, an open subset, and the divisor complement.  This suggests that we may define the realizations through the $\lambda$-eigenspaces described above.  We shall only need to do so for the crystalline realization, but we will use these arguments for other cohomology theories as well to show the necessary comparison theorems. 

\subsubsection{Singular and \'etale realizations} 
We define $V_{\sing} = R^{n-1}f_*g^*(\L_{\chi})_{\sing}$ and $V_{\et} = R^{n-1}f_*g^*(\L_{\chi})_{\et}$.  Recall that $\chi$ is a good character, so $R^{n-1}f_* = R^{n-1}f_!$ in both scenarios.  Thus both $V_{\sing}$ and $V_{\et}$ are local systems.  It will be useful to express $V_{\et}$ in an alternate form to compare it with the crystalline realization.  

In the previous sections we have constructed a compactification as shown below.   
\[\begin{tikzcd}
	{\mc{Y}} & {\mc{X}\times \G_m^n} & {\mc{X}\times Z_{\Delta}} \\
	&& X
	\arrow[hook, from=1-1, to=1-2]
	\arrow[hook, from=1-2, to=1-3]
	\arrow[from=1-1, to=2-3]
	\arrow[from=1-3, to=2-3]
\end{tikzcd}\]
Considering the closure $\overline{\mc{Y}}\subset \mc{X}\times Z_\Delta$, we obtain the following composition. 
\[\begin{tikzcd}
	{[m]^{-1}\mc{Y}} & {\bar{[m]^{-1}\mc{Y}}} \\
	{\mc{Y}} & {\bar{\mc{Y}}} & {\mc{X}}
	\arrow["{[m]}"', from=1-2, to=2-2]
	\arrow["j"', hook, from=2-1, to=2-2]
	\arrow["{[m]}"', from=1-1, to=2-1]
	\arrow["{h_m}", from=1-2, to=2-3]
	\arrow["h"', from=2-2, to=2-3]
	\arrow["{j_m}", hook, from=1-1, to=1-2]
\end{tikzcd}\]

By a slight abuse of notation, we will also denote the character associated to $g^*\L_{\chi}$ by $\chi$, so we have a commutative diagram 
\[\begin{tikzcd}
	{\pi_1(\mc{Y})} & {\mu_m^n} \\
	& {\Qlb^*}
	\arrow[two heads, from=1-1, to=1-2]
	\arrow["\chi"', from=1-1, to=2-2]
	\arrow["\lambda", from=1-2, to=2-2]
\end{tikzcd}\] 

\begin{proposition}
    We have 
    \[
V_{\et}\cong (R^{n-1}h_{m*}(\Qlb))_{\lambda}. 
    \] 
\end{proposition} 
\begin{proof}
    We have $Rf_*g^*\L_{\chi} = (Rf_*[m]_*\Qlb)_{\lambda}$.  Let $D_m = \bar{\mc{Y}_m}\backslash \mc{Y}_m\hookrightarrow{i_m}\bar{\mc{Y}_m}$.  We know that $(R^{n-1}h_{m*}i_m^*\Qlb)_{\lambda}=0$.  We have a short exact sequence of sheaves 
    \[
0\rightarrow j_{m!}\Qlb \rightarrow \Qlb\rightarrow i_m^*\Qlb\rightarrow 0
    \] 
    on $\bar{\mc{Y}_m}$. 
Applying $R^{n-1}h_{m*}$ and taking $\lambda$-eigenspaces, we obtain 
\[
(R^{n-1}h_{m*}\Qlb)_{\lambda} \cong (R^{n-1}h_{m*}j_{m!}\Qlb)_{\lambda}\cong (Rf_*[m]_*\Qlb)_{\lambda} \cong Rf_*g^*\L_{\chi} 
\]
as desired,
using the isomorphisms above and the fact that $\chi$ is a good character.
\end{proof}

\subsubsection{de Rham realization} 
As a local system, we can realize $g^*\L_\chi$ as the sheaf of flat sections of a vector bundle with integrable connection, which we denote $(\mc{E}_\chi, \nabla)$.  Then we can associate to $R^{n-1}f_*g^*\L_\chi$  the relative de Rham cohomology of $Y$ with $\L_\chi$-coefficients.  Explicitly, this is given by 
\[
V_{\dR} = R^{n-1}f_*(\mc{E}_\chi \rightarrow \mc{E}_\chi\otimes \Omega^1_{\mc{Y}/\mc{X}}\rightarrow \cdots). 
\]
$V_{\dR}$ is equipped with the Gauss-Manin connection, which is integrable.

Using an argument similar to the \'etale case, we have the following description of $V_{\dR}$. 

\begin{proposition} 
      We have 
    \[
V_{\dR}\cong (R^{n-1}h_{m*}(\Omega^*_{\bar{[m]^{-1}\mc{Y}}/\mc{X}}))_{\lambda}. 
    \] 
\end{proposition}

\subsubsection{Crystalline realization}
Before constructing the crystalline realization, we first review some notions from $p$-adic cohomology. 
 The notion of a convergent isocrystal was defined in \cite{OgusII} to give a relative version of crystalline cohomology.  To be more precise, let $(V, k)$ be a complete discrete valuation ring with fraction field $K$ of characteristic 0 and residue characteristic $p$.  Given $X/V$ a scheme or $p$-adic formal scheme, the convergent site is defined via certain $p$-adic enlargements $(T, z_T)$ where $T$ is a flat $p$-adic formal $V$-scheme $z_T$ is a $V$-morphism $(T/p\O_T)_{\red}\rightarrow Z$.  A convergent isocrystal is a sheaf on the convergent site whose linearized transition morphisms, associated to morphisms of objects on the convergent site, are isomorphisms.  Given any smooth proper morphism $f\colon X\rightarrow Z$, one can define a convergent $R^qf_*\O_{X/K}$ on $Z/V$ with the desirable functorial properties.  Despite the lack of divided powers in the definition of the convergent site, this definition agrees with crystalline cohomology if $Z=\Spec k$.  \\ 

Similar to the case of crystalline cohomology, one can view convergent isocrystals as vector bundles with integral connections.  Another way to think about them comes from Theorem 2.15 of \cite{OgusII}.  Namely, if $S$ is a formally smooth $p$-adic formal $V$-scheme with reduced special fiber $S_0$, then there is a realization functor $E_{X_0}\mapsto (E_X, \nabla_X)$ that induces a fully faithful functor from convergent isocrystals on $X_0$ to coherent $\O_{X, K}$-modules with integrable connection.  Furthermore, in that paper Ogus constructs a convergent F-isocrystal from the higher pushforwards along a smooth proper morphism of the structure sheaf; the result we will use is the following. 

\begin{theorem}
    \cite[Theorem 3.1]{OgusII} 
    There exists a p-adically convergent isocrystal $R^qf_*\O_{X/K}$ on $Z/V$ with the following property: for each $p$-adic enlargement $(T, z_T)$ of $Z/W$, there is a natural isomorphism: 
    \[
\alpha_T\colon (R^qf_*\O_{X/K})_T\cong K\otimes R^qf_{T_1*}\O_{X_{T_1}/T}. 
    \]
    The pair $(R^qf_*\O_{X/K}, \alpha_{\bullet})$ is unique up to unique isomorphism. 
\end{theorem}

The reason we need to use eigenspaces of a smooth proper morphism for the crystalline realization is that in general, the pushforward of an $F$-isocrystal is not known to be an $F$-isocrystal.  However, in the smooth proper case Berthelot's conjecture (proven in the cases we need) gives the desired result.  Thus, to deal with the crystalline realization we will define the F-isocrystal by defining it  through the eigenspace of one obtained through a smooth proper morphism. \\ 

As before, given a good character $\chi\in \pi_1^{\et}((\G_m^n)_K)$ be a character of order $r$, we have defined a Hodge-Deligne system denoted $\L_\chi$ on $(\G_m^n)_L$.  All the structures of a Hodge-Deligne system are stable under pullback, so we obtain a Hodge-Deligne system $g^*\L_\chi$ on $\mc{Y}$.   Then we define the crystalline realization as 
\[
V_{\cris} = (R^{n-1}_*h_{m*}\O_{\bar{[m]^{-1}\mc{Y}}})_{\lambda}. 
\] 

This is an F-isocrystal by Theorem 3.1 of \cite{OgusII}.  Furthermore it follows the construction of Example 5.10 in \cite{finiteness}. 
 There the comparison theorems are shown between $V_{\cris}$ and $V_{\et}, V_{\dR}$, where the latter two are taken to be $R^{n-1}_*h_{m*}(\Q_p)_{\et}$ and $R^{n-1}_*h_{m*}(\Q_p)_{\dR}$ respectively.  These isomorphisms are compatible with taking $\lambda$-eigenspaces, and since we have shown that our definition of $V_{\et}$ and $V_{\dR}$ are equivalent to the $\lambda$-eigenspaces of these forms just mentioned, the comparison theorems hold in our scenario.

\subsection{Construction of a Hodge-Deligne system $R^{n-1}f_*\O$ on $\mc{X}$} 
The Hodge-Deligne systems constructed in the previous subsection live on $\mc{X}_{\O_{L, S}}$.  Following Lemma 5.29 of \cite{finiteness}, we will now construct a Hodge-Deligne system $V_I$ on $\mc{X}$ whose base change to $\O_{L, S}$ decomposes as a sum of those Hodge-Deligne systems.  Furthermore, we will construct an $H^0$-algebra $\msf{E}_I$ such that $\msf{V}_I$ is an $\msf{E}_I$-module.  \\ 

\begin{construction}
\label{construction}
[Construction of $\msf{E}_I$ and $\msf{V}_I$, ]
(cf. \cite[Lemma 5.29]{finiteness})
Let $\prod^{K/\Q}(\G_m^n)[r]$ be the set of pairs $(\iota, \chi)$ of embeddings $\iota\colon K\rightarrow L$ and $r$-torsion characters $\chi$ of $\pi_1(\G_m^n)$.  Then $\prod^{K/\Q}(\G_m^n)[r]$ carries an action of $\Gal(\Q)\times \Gal(\Q[\mu_r]/\Q)$.  Let $I$ be an orbit of this action that consists of good characters.  Then given the maps of the diagram below, we construct 
\[
V = R^{n-1}h_{m*}\O_{\Q} \cong R^{n-1}h_*g^*[r]_*\O_{\Q}. 
\] 
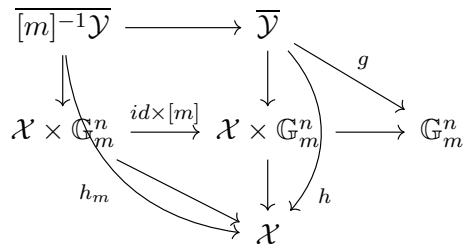
\[\begin{tikzcd}
	{\bar{[m]^{-1}\mc{Y}}} & {\bar{\mc{Y}}} \\
	{\mc{X}\times \G_m^n} & {\mc{X}\times \G_m^n} & {\G_m^n} \\
	& {\mc{X}}
	\arrow[from=1-1, to=2-1]
	\arrow[from=1-1, to=1-2]
	\arrow[from=1-2, to=2-2]
	\arrow["{id\times [m]}", from=2-1, to=2-2]
	\arrow[from=2-1, to=3-2]
	\arrow["h_m"', curve={height=30pt}, from=1-1, to=3-2]
	\arrow[from=2-2, to=3-2]
	\arrow["h"{pos=0.8}, curve={height=-24pt}, from=1-2, to=3-2]
	\arrow[from=2-2, to=2-3]
	\arrow["g", from=1-2, to=2-3]
\end{tikzcd}\]

By Example 5.10 of \cite{finiteness}, the pushforward by a smooth projective morphism of a Hodge-Deligne system produces a Hodge-Deligne system.  Thus $\msf{V}$ is a Hodge-Deligne system.  

Following Lemma 5.29 of \cite{finiteness}, we define $\msf{E}$ similarly, so that $\msf{V}$ is an $H^0$-algebra over $\mathsf{E}$.  Namely, 
\[
\msf{E} = (\pi_*\msf{O}_{\Q}|_{\G_m^n[r]})^{\vee}, \quad \msf{E}_I = \pi_*\msf{O}_{\Q[\mu_r]}
\] 
where $\pi\colon \Spec K\rightarrow \Spec \Q$ is the projection.
Then $\msf{E}$ splits into a direct sum of $\msf{E}_I$ where $I$ ranges over the orbits, including the ones that include characters that are not good.  We now define $\msf{V}_I = \msf{V}\otimes_{\msf{E}} \msf{E}_I$.  We have 

\[
\msf{V}_I|_{\O_{L, S}}\otimes_{\O_\Q} \O_{\Q[\mu_r]} = \bigoplus_{(\iota, \chi)\in I} R^{n-1}f_{\iota*}g_{\iota}^*\L_\chi. 
\]

Furthermore, Poincar\'e duality (since the characters are good) and the integral structure on $\msf{V}_{I, \sing}$ makes $\msf{V}_I$ into a polarized, integral Hodge-Deligne system. 
    \end{construction}

\subsection{Bounding the Frobenius centralizer}
Now we will give a bound on the Frobenius centralizer arising from this Hodge-Deligne system.  We will need a stronger bound than the one in Lemma 5.33 of \cite{finiteness}.  Instead, we will use a variant of the following result from $\cite{LV}$. 

\begin{lemma}
    [Lemma 10.4 of \cite{LV}] There exists an integer $L$ with the following property: 

    For any $y\in Y(\Z[S^{-1}])$, there exists a prime $l\le L, l\not\in S$ such that the semisimplification of $F_y^{crys, l}$ (and so also the crystalline Frobenius itself) satisfies 
    \[
\dim Z([F_y^{crys, l}]^{ss})\le \dim Z_{\textbf{G}'(\C)}(\varphi_0).
\]

On the left hand side, we take the centralizer inside $G\Aut(H_{crys}^{d, prim})$, to which the crystalline Frobenius -- and so also its semisimplification -- belongs. 
\end{lemma}

Our situation differs in the following ways.  We have a smooth morphism $f:Y\rightarrow X$ whose fiber is a hypersurface in $\G_m^n$.  We have a Hodge-Deligne system $V_I$ on $\mc{Y}$ denoted $V = g^*\L_\chi$ for a good character $\chi$.  Fix a point $x\in X(\Z[S^{-1}])$.  The fiber functor on the Tannakian subcategory of perverse sheaves on $Y_x$ generated by $g^*\L_\chi$ can be taken to be $R^{n-1}f_*(-)$, as $\L_\chi$ is good. By functoriality this gives a morphism $\phi_0\colon \G_m\rightarrow G(Y_x, \chi)$, the corresponding Tannakian group.  This is equivalent to a map to the geometric monodromy group $G(\C)$.  On the other hand, given a prime $l$, the filtered F-isocrystal of $V_I$ gives a an action of the crystalline Frobenius $F_x^{\cris, l}$ on $V_{\cris}$, which we recall to be defined as an eigenspace of the crystalline cohomology of the inverse image of a multiplication by $m$ map on the compactification of $\mc{Y}_x$. 

\begin{proposition}
\label{prop: centralizer bound}
    There exists an integer $L$ such that for any $x\in X(\Z[S^{-1}])$, there exists a prime $l\le L, l\not\in S$ such that the semisimplification of $F_x^{\cris, l}$ satisfies 
    \[
\dim Z([F_x^{\cris, l}]^{ss}) \le \dim Z_{G(Y_x, \chi)}(\phi_0). 
    \]
\end{proposition}

Indeed, the proof of Lemma 10.4 of \cite{LV} works in our scenario here, after making the appropriate substitutions mentioned above.  As in Section 10 of \cite{LV}, the right hand side is given by $h^0$ where $h^p$ denote the adjoint Hodge numbers on $G(Y_x, \chi)$.  We will calculate these adjoint Hodge numbers in the following section.




\section{Big monodromy and numerical conditions} 
The goal of this section is to put together the big monodromy statements and numerical conditions required to prove Zariski non-density of integral points on $\mc{X}$. 

To fix notation, we will apply the setting of Section 5.6 of \cite{finiteness}, which we now recall.  Let $E_0$ be a field and let $E$ be a finite \'etale $E_0$-algebra.  Let $\textbf{H}$ be a reductive group over $E_0$, which is one of $GL_N, GSp_N,$ or $GO_N$.  Let $V_{simp}$ be the standard representation of $\textbf{H}$.  Let $\textbf{G}^0$ be the Weil restriction $\textbf{G}^0 = \Res^E_{E_0}\textbf{H}_E$, and let $V = E\otimes_{E_0}V_{simp}$.  Let $\textbf{G}$ be the normalizer of $\textbf{G}^0$ in the group of $E_0$-linear automorphisms of $V$. 

 In what follows we wish to apply this formalism to the Hodge-Deligne system $\msf{V}_I$.  Thus we will take $E_0=\Q_p$ and $E = (\msf{E}_I)_{\et}$.  We recall that the latter is given by $\Q[\mu_r]\otimes_{\Q}\Q_p$.

\subsection{Big monodromy statement} 
Using the language of Hodge-Deligne systems, we can restate the big monodromy statement of \ref{cor: big} in more amenable terms.  Recall that $f\colon \mc{Y}\rightarrow \mc{X}$ is the universal family of hypersurfaces nondegenerate with respect to their fixed Newton polyhedron.  In \ref{cor: big}, we have selected a character $\chi_0$ all of whose elements in its $\Gal(\overline{\Q}/K)$-orbit have large monodromy.  Fix some embedding $\iota\colon K\rightarrow L$, where as before $L$ is chosen Galois over $\Q$, containing $K$ over with $\G_m^n[r]$ splits.  Then letting $I$ be the full $(\Gal_{\Q}\times \Gal_{\Q^{\cyc}/\Q})$- orbit containing $(\iota_0, \chi_0)$, we have constructed a Hodge-Deligne system $V_I$ on $\mc{X}$. \\ 

The differential Galois group $\textbf{G}_{\mon}$ of a Hodge-Deligne system $V$ is the differential Galois group of $V_{\dR}$.  It is also given by the Zariski closure of the monodromy group of the associated variation of Hodge structure, and thus functions as a measure of the geometric monodromy of the Hodge-Deligne system.

We will give the big monodromy statement in terms of 1-balanced subgroups.  We recall the definition below. 

\begin{definition}
    \cite[Definition 6.6]{finiteness} Let $H$ be one of the algebraic groups $SL_N, Sp_N$, or $O_n$, and let $G$ be a subgroup of $H^d$.  For $1\le i\le d$, let $\pi_i\colon G\rightarrow H$ be the coordinate projection, and suppose that each $\pi_i$ is surjective.  The index classes of $G$ are the equivalence classes of $\{1, \ldots, d\}$ under the equivalence relation $i\sim j$ if and only if the projection 
    \[
(\pi_i, \pi_j)\colon G\rightarrow H^2 
    \]
    is not surjective. (This is proven to be an equivalence relation in \cite[Lemma 6.5]{finiteness}.)

    For a positive integer $c$, we say that $G$ is $c$-\textit{balanced} if its index classes are all of the same size, and there are at least $c$ of them.  For a permutation $\sigma$ of the index set $\{1, \ldots, d\}$, we say that $G$ is \textit{strongly $c$-balanced} if it is $c$-balanced, each orbit of $\sigma$ contains elements of at least $c$ of the index classes of $G$, and $\sigma$ preserves the partition of $\{1, \ldots, d\}$ into index classes. 
\end{definition}  

\begin{remark}
    We will only use this in the case $c=1$; in that case being strongly $c$-balanced reduces to having index classes of the same size and $\sigma$ preserving the partition of index classes (in addition to the $\pi_i$ being surjective).
\end{remark} 

Finally, we also define $\textbf{G}^1$ and the $c$-(strongly) balancedness of an algebraic subgroup $G\subseteq \textbf{G}^0$ as in \cite[Definition 6.6]{finiteness}. 

We can now state an analogue of Lemma 6.8 of \cite{finiteness}; as in that paper, it is essentially just a restatement of the big monodromy result. 

\begin{lemma}
\label{lem: gmon}
    Pick $\chi_0$ as in Corollary \ref{cor: big} and let $I$ be the full $(\Gal_{\Q}\times \Gal_{\Q^{\cyc}/\Q})$- orbit containing $(\iota_0, \chi_0)$.  Let $V = V_I$ be the corresponding Hodge-Deligne system.  The Frobenius at $v$, an element of $\Gal_K$, acts on the set $I$; call this permutation $\sigma$.  Then the differential Galois group $\textbf{G}_{\mon}$ of $V$ (base-changed to $\Q_p$) is a strongly 1-balanced subgroup of $G^0$, with respect to $\sigma$. 
\end{lemma}

\begin{proof}
    Recall that $\textbf{G}_{\mon}\otimes \C$ is the Zariski closure of the monodromy group of the variation of Hodge structure of $V_I$.  By our construction of $V_I$, this is given as the monodromy group of $\bigoplus_{(\iota, \chi)\in I} R^{n-1}f_{\iota *}g_{\iota}^*\L_{\chi}$.  Since monodromy acts trivially on the indexing set and we have an appropriate bilinear pairing on the $R^{n-1}f_{\iota *}g_{\iota}^*\L_{\chi}$ when $\textbf{H} = GSp$ or $GO$, we have $\textbf{G}_{\mon}\subseteq \textbf{G}^0$.  Next, Corollary \ref{cor: big} implies that the geometric monodromy group of each individual $R^{n-1}f_{\iota *}g_{\iota}^*\L_{\chi}$ is all of $\textbf{H}(\C)$, so $\textbf{G}_{\mon}$ does indeed surject onto each factor $\textbf{H}$ of $\textbf{G}^0$.  In particular, the projection of the intersection with $\textbf{G}^1$ contains either $SL_N, Sp_n$, or $O_N$.  Finally, since the Galois group action on the pairs $(\iota, \chi)$ respects the equivalence relation in the definition of strongly $c$-balanced subgroups, $\sigma$ preservers the partition of $\{1, \ldots, d\}$ into index classes, and the result follows. 
\end{proof}

\subsection{Numerical conditions for Zariski non-density} 

Our goal is to formulate numerical conditions similar to Theorem 8.17 of \cite{finiteness} such that, if they are satisfied along with a big monodromy condition, we have a general non-density statement.  However, because we cannot leverage the $c$-balanced condition, the numerical conditions of Theorem 8.17 will not be satisfied in our context.  To account for this, as done in \cite[Theorem 10.1]{LV} we will use the improved bound on the centralizer given in Proposition \ref{prop: centralizer bound}.

\begin{lemma}
    [Lemma 8.16 of \cite{finiteness}] Assume e in the setting of Section 5.6 [of \cite{finiteness}].  Fix a $\textbf{G}^0$-bifiltered $\phi$-module $(V_0, \phi_0, F_0, \mf{f}_0)$ and another $\phi$-module $(V, \phi)$ with $\textbf{G}$-structure; suppose both $\phi$ and $\phi_0$ are semilinear over some $\sigma\in \Aut_{E_0}E$, and $F_0$ is balanced with respect to $\mf{f}_0$. 

    Let $\textbf{G}_{\mon}$ be a subgroup of $\textbf{G}^0$ , strongly $c$-balanced with respect to $\sigma$ for some positive integer $c$. 

    Suppose $F_0$ is uniform in the sense of Definition 5.46, and let $h^a = h^a_{simp}$ be the adjoint Hodge numbers $\textbf{H}$.  Let $t$ be the dimension of a maximal torus in $\textbf{H}$.  Suppose $e$ is a positive integer satisfying the following numerical conditions. 
    \begin{itemize}
        \item (First numerical condition.)
        \[
\sum_{a>0}h^a \ge \frac{1}{c}(e + \dim \textbf{H})
        \]
        and 
        \item (Second numerical condition.) 
        \[
\sum_{a>0}ah_a > T\left(\frac{1}{c}(e+\dim \textbf{H})\right) + T\left(\frac{1}{2}(h^0-t) +\frac{1}{c}(e+\dim \textbf{H})\right). 
        \]
    \end{itemize} 

    Let $\mc{H}=\textbf{G}_{\mon}/Q_0\cap \textbf{G}_{\mon}$ be the flag variety parameterizing filtrations on $\textbf{G}^0$ that are conjugate to $F_0$ under the conjugation of $\textbf{G}_{\mon}$.  Then the filtrations $F$ such that $(V, \phi, F)$ is of semisimplicity type $(V_0, \phi_0, F_0, \mc{f}_0)$ are of codimension at least $e$ in $\mc{H}$.  
\end{lemma}

\begin{remark}
   One of the major difficulties in proving this result in \cite{finiteness} is the need to deal with Galois representations which are not semisimple.  We do not need to alter their analysis of this point.  
\end{remark}

We now state the following variant of Theorem 8.17 of \cite{finiteness}. 

\begin{theorem}
\label{thm: zariski non-density}
    [cf. Theorem 8.17 of \cite{finiteness}]
    Let $X$ be a variety over $\Q$, let $S$ be a finite set of primes of $\Z$, and let $\mc{X}$ be a smooth model of $X$ over $\Z[1/S]$. 

    Let $E$ be a constant $H^0$-algebra on $\mc{X}$, and let $\textbf{H}$ be one of $GL_N, GSp_N$, or $GO_N$. Let $V$ be a polarized, integral $E$-module with $\textbf{H}$-structure, having integral Frobenius eigenvalues.  Let $h^a_{\iota}$ be the adjoint Hodge numbers of the $\textbf{H}$-filtration corresponding to an embedding $\iota\colon E\hookrightarrow \Qpb$.  Let $t$ be the dimension of a maximal torus in $\textbf{H}$.  Let $\textbf{G}^0$ and $\textbf{G}$ be as in this section.  Suppose there is a positive integer $c$ such that for any choice of embedding $\iota$, the following conditions hold. 

    \begin{itemize}
        \item (Big monodromy) The differential Galois group $\textbf{G}_{\mon}\subseteq \textbf{G}^0$ is strongly $c$-balanced with respect to Frobenius.

        \item (First numerical condition) 
        \[
    \sum_{a>0} h^a_{\iota} \ge \dim X + h^0.
        \]
        \item (Second numerical condition) 
        \[
    \sum_{a>0} ah^a_{\iota} > T_G(\dim X + h^0) + T_G(\dim X + \frac{3}{2}h^0).  
        \]
    \end{itemize}

Then the image of $\mc{X}(\Z[1/S])$ is not Zariski dense in $X$.
\end{theorem}

\begin{proof}
    This result is equivalent to \cite[Theorem 8.17]{finiteness} with $c=1$ with two modifications.  First, in that result the term $h^0$ is replaced with $\dim \textbf{H}$.  The term $\dim \textbf{H}$ arises from the inequality 
\[
\dim Z_{G^J}((\phi^J)^{ss})\le \dim Z_{\textbf{G}}(\phi)\le \dim \textbf{H}, 
\] 
where J is a subset of $\{1, \ldots, d\}$.  The result is proven for $\dim Z_{G^J}((\phi^J)^{ss})$, but it is sufficient in \cite{finiteness} to use the weaker bound involving $\dim \textbf{H}$.  The result would thus hold using $\dim Z_{\textbf{G}}(\phi)$ in place of $\dim \textbf{H}$ in the numerical conditions.  By Proposition \ref{prop: centralizer bound}, we have $\dim Z_{\textbf{G}}(\phi)\le h^0$, so we may replace $\dim \textbf{H}$ with $\dim h^0$ in the numerical conditions.   \\ 

    Second, in our scenario the adjoint Hodge numbers may depend on the embedding, as it may change the character; thus we do not assume as in the statement of \cite[Theorem 8.17]{finiteness} that the Hodge numbers are uniform.   Instead, we assume that the numerical conditions are met for every character used in the orbit.  To see this is sufficient, we note that these conditions are used in the proof of \cite[Theorem 8.16]{finiteness} (which is then used to prove \cite[Theorem 8.17]{finiteness}, where they project from $\textbf{G}_{\mon}\subseteq \textbf{H}^d$ onto $c$ of the $d$ factors of $\textbf{H}$, and then use \cite[Theorem 8.12]{finiteness} to conclude.  Here we take $c=1$ and by our assumption we have that Lemma 8.12 applies for every choice of character, so the argument goes through as desired.

\end{proof}

\section{Computation of adjoint Hodge numbers and verification of numerical conditions}
In this section we will compute the Hodge numbers of the de Rham realization of the Hodge-Deligne system on $\mc{X}$ constructed in the previous section.    

\subsection{The Hodge filtration on the cohomology}  
Proposition 1.12 of \cite{DK} allows us to compute the Hodge filtration in the following way.  Let $X$ be a smooth algebraic variety, and let $\overline{X}$ be a smooth compactification of $X$ such that the divisor $D=\bar{X}\backslash X$ has transversal intersections in $\bar{X}$.  We denote by $\Omega^p_{(\bar{X}, D)}$ the sheaf of germs of regular differential $p$-forms on $\bar{X}$ vanishing on $D$.  Then there is a spectral sequence  
\[
E_1^{p, q} = H^q(\bar{X}, \Omega^p_{(\bar{X}, D)})\Rightarrow H^{p+q}_c(X, \C)
\]
degenerating at $E_1$ and converging to the Hodge filtration on $H^*_c(X)$. \\ 

Recall that given a character $\chi$ of the torus corresponding to an eigenspace $\lambda$ of the action of $\mu_m^n$, we showed
   \[
V_{\dR}\cong (R^{n-1}h_{m*}(\Omega^*_{\bar{[m]^{-1}\mc{Y}}/\mc{X}}))_{\lambda}. 
    \] 
    
In our situation, we would like to compute the Hodge numbers of each stalk $Y_x$ associated to these characters.  For notational convenience, let $Y$ be a hypersurface in $\G_m^n$ nondegenerate with respect to its Newton polyhedron $\Delta$ and let $Y_m$ be the associated hypersurface given by taking $m$th powers.  Let $\textbf{P}=\textbf{P}_\Delta$ be the associated compactification of the torus and let $\overline{Y_m}$ denote the closure of $Y_m$ in $\textbf{P}$.  Then the corresponding Hodge numbers of $V_{\dR}$ are given by the dimension of the $\lambda$-eigenspace of the $(\mu_m)^n$-module $H^q(\overline{Y_m}, \Omega^{n-1-q}_{(\ol{Y_m}, D)})$ into its eigenspaces. 

To accomplish this, we will work in the Grothendieck group of $\mu_m^n$-modules.  Indeed, it will be more natural to compute \[\sum_{i\ge 0}(-1)^i[H^i(\overline{Y_m}, \Omega^{n-1-q}_{(\ol{Y_m}, D)})].\]

\begin{lemma}
    In the Grothendieck group of $\mu_m^n$-modules, we have 
     \[
     [H^q(\overline{Y_m}, \Omega^{n-1-q}_{(\ol{Y_m}, D)})] = (-1)^q\left(\sum_{i\ge 0}(-1)^i[H^i(\overline{Y_m}, \Omega^{n-1-q}_{(\ol{Y_m}, D)})] + (-1)^{n-1}[H^{2n-1-q}_c(\G_m^n, \C)]\right).
    \]
\end{lemma} 

\begin{proof}
    By Proposition 3.9 of \cite{DK} The Gysin homomorphism $H^i_c(Z, \C)\rightarrow H^{i+2}_c(\G_m^n, \C)$ is an isomorphism for $i>n-1=\dim Z$.  This isomorphism respects the Hodge filtration, and moreover $H^i_c(Z, \C)=0$ for $i<n-1$.  This means that $H^i(\ol{Y_m}, \oyd) = 0$ for $i<q$.  Furthermore, for $i>q$, we have $H^i(\ol{Y_m}, \oyd) = F^{n-1-q} H^{n+i-q+1}_c(Y_m, \C)/F^{n-q} H^{n+i-q+1}_c(Y_m^n, \C) = H^{i+2}_c(\G_m^n, \C)$.  However, we recall that the nonzero compact Hodge numbers of the torus are given by $h^{p,p}H^{n+p}(\G_m^n, \C) = \binom{n}{p}$ and are 0 for all others.  Thus the only nonzero part of the filtration of $H^{n+i-q+1}_c(\G_m^n, \C)$ occurs when $p = i-q+1$, and $i-q+1 = n-1-q$.     6

Then we have 
\begin{align*}
    [H^q(\overline{Y_m}, \Omega^{n-1-q}_{(\ol{Y_m}, D)})] &= (-1)^q\Bigg(\sum_{i\ge 0}(-1)^i[H^i(\overline{Y_m}, \Omega^{n-1-q}_{(\ol{Y_m}, D)})] + \sum_{i\ge q+1}(-1)^{i+1}[H^i(\overline{Y_m}, \Omega^{n-1-q}_{(\ol{Y_m}, D)})]\Bigg) \\ 
    &= (-1)^q\Bigg(\sum_{i\ge 0}(-1)^i[H^i(\overline{Y_m}, \Omega^{n-1-q}_{(\ol{Y_m}, D)})] ~ +  \\ 
    &  \hspace{2cm} \sum_{i\ge q+1}(-1)^{i+1}[F^{n-1-q} H^{n+i-q+1}_c(\G_m^n, \C)/F^{n-q} H^{n+i-q+1}_c(\G_m^n, \C)]\Bigg) \\ 
    &= (-1)^q\Bigg(\sum_{i\ge 0}(-1)^i[H^i(\overline{Y_m}, \Omega^{n-1-q}_{(\ol{Y_m}, D)})] + (-1)^{n-1}[H^{2n-1-q}_c(\G_m^n, \C)]\Bigg).
\end{align*} 
\end{proof}

We now compute $\sum_{i\ge 0}(-1)^i[H^i(\overline{Y_m}, \Omega^{n-1-q}_{(\ol{Y_m}, D)})]$ as a $\mu_m^n$-module.

\begin{lemma}
\label{lem: Laurent poly}
Let $L^*(\Delta)$ be the space of Laurent polynomials whose support lies strictly in the
interior of $\Delta$, considered as a $\mu_m^n$-module.  Then 
    \[(-1)^q\sum_{i\ge 0}(-1)^i[H^i(\overline{Y_m}, \Omega^{n-1-q}_{(\ol{Y_m}, D)})] = \sum_{i=0}^q (-1)^i\dbinom{n+1}{i}L((q+1-i)m\Delta).\]   
\end{lemma} 

\begin{proof}
Section 4.1 of \cite{DK} gives the following exact sequences  of coherent sheaves on ${\bf{P}}$: 
\begin{align*}
    & 0\rightarrow \Omega^p_{(\ol{Z}, D_Z)}\otimes_{\O_P}\O_P(-\ol{Z})\rightarrow \Omega^{p+1}_{({\bf{P}}, D)}\otimes_{\O_P}\O_{\ol{Z}} \rightarrow\Omega^{p+1}_{(\ol{Z}, D_Z)}\rightarrow 0. \\ 
    & 0\rightarrow \O_{\mb{P}}(-\Delta)\rightarrow \O_{\mb{P}}\rightarrow \O_{\ol{Z}}\rightarrow 0. 
\end{align*}  

These give 
\[
\sum_{i\ge 0}(-1)^i[H^i(\overline{Y_m}, \Omega^{n-1-q}_{(\ol{Y_m}, D)})] = \sum_{i, k\ge 0}(-1)^{i+k}[H^i({\bf{P}}, \Omega^{n-q+k}_{({\bf{P}}, D)}((k+1)\Delta)) - H^i({\bf{P}}, \Omega^{n-q+k}_{({\bf{P}}, D)}(k\Delta))]. 
\]
The same result is given in \cite{DK} except that they take the Euler characteristic of the above equality, as they are just interested in the dimension.  

This expression is simplified considerably by Proposition 2.10 in \cite{DK}, which states that for any $n$-dimensional $\Delta$, we have $H^k({\bf{P}}, \Omega^{p}_{({\bf{P}}, D)}(\Delta)) = \Lambda^p(\R^n)\otimes L^*(\Delta)$ for $k=0$ and is 0 otherwise, where $L^*(\Delta)$ is the space of Laurent polynomials whose support lies strictly in the interior of $\Delta$.  We note that the $(\Z/m)^n$-action on this splits up into characters according to the integral points in the interior of $\Delta$. 
 Applying this to $Y_m$, note first that its Newton polyhedron is given by $m\Delta$, where $\Delta$ is the Newton polyhedron of $Y$. We have  
 \begin{align}
   \sum_{i\ge 0}(-1)^i[H^i(\overline{Y_m}, \Omega^{n-1-q}_{(\ol{Y_m}, D)})] &= \sum_{0\le k\le q}(-1)^{k}[\tbinom{n}{n-q+k}(L^*((k+1)\Delta) - L^*(k\Delta))] \\ 
&= \sum_{1\le k\le q+1}(-1)^{k+1} \tbinom{n+1}{n-q+k}L(k\Delta). 
 \end{align}
 Re-indexing gives the desired result. 
\end{proof}

It is well-known that an explicit formula for the Eulerian numbers $A(n, k)$ is given by 
\[
A(n, k) = \sum_{i=0}^k(-1)^i\binom{n+1}{i}(k+1-i)^n.
\] 
Furthermore, note that each of the $m^n$ characters occurs in the decomposition of $L((q+1-i)m\Delta)$ approximately $(q+1-i)^n\Vol(\Delta)$ times.  Because 
\[
     [H^q(\overline{Y_m}, \Omega^{n-1-q}_{(\ol{Y_m}, D)})] = (-1)^q\left(\sum_{i\ge 0}(-1)^i[H^i(\overline{Y_m}, \Omega^{n-1-q}_{(\ol{Y_m}, D)})] + (-1)^{n-1}[H^{2n-1-q}_c(\G_m^n, \C)]\right)
    \]
we see that the characters occurring in the compact cohomology of $\G_m^n$ have an additional contribution.  But $H^{2n-1-q}_c(\G_m^n, \C)$ has dimension $\binom{n}{n-q-1}$ and decomposes into $\binom{n}{n-q-1}$ unique 1-dimensional characters.  Therefore we conclude that as $n$ and $\Delta$ get larger, the Hodge-Deligne weights generically limit towards $A(n, q)\Vol(\Delta)$.  We will now make this precise. 

\subsection{Verification of numerical conditions for the Eulerian distribution} 
The expression for the adjoint Hodge numbers in terms of the Hodge-Deligne numbers depends on $\textbf{H}$.  In the case of $\textbf{H} = GL_n$, we have 
\[
2h^p_a = \sum_{p_1+p_2 = p+n-1} h^{p_1, n-1-p_1}h^{p_2, n-1-p_2}. 
\]
In the case of $\textbf{H}=GO_n$, we have 
\begin{equation}
\label{eq: hp}
2h^p = \sum_{p_1+p_2 = p+n-1} h^{p_1, n-1-p_1}h^{p_2, n-1-p_2}\pm h^{(p+n-1)/2, (-p+n-1)/2}.
\end{equation}

We are particularly interested in these two cases because these are the possibilities for the monodromy group imposed by Theorem \ref{thm: gabber}.

We will first carry out the calculations in the case $\textbf{H} = GL_n$; as we will see the $SO_n$ case is not too different.  If we replaced $L(k\Delta)$ in the conclusion of Lemma \ref{lem: Laurent poly} with $k^n\Vol(\Delta)$, then the Hodge-Deligne numbers would be precisely given by $A(n, q)\Vol(\Delta)$.  In this case, the adjoint Hodge numbers follow a distribution given by the convolution of the Eulerian numbers.  In this case, the author has shown in prior work \cite{Ji} that the numerical conditions are satisfied for sufficiently large dimension.  We briefly review the results.  

Let $X_n$ be the random variable indicating $\frac{n-1}{2}$ less than the number of descents to a uniformly distributed random permutation of $\{1, 2, \ldots, n\}$.  This is precisely the normalized distribution of the Eulerian numbers.  By \cite{Stanley}, $P(X_n =p - \frac{n-1}{2})$ is the probability that a sum of $n$ uniform, independent and identically distributed variables between 0 and 1 is between $p$ and $p+1$.  

Now let $X'_n$ be the random variable given by a sum of two distinct copies of $X_n$.  Let $\beta_p = P(X'_n = p)$.  Then the $\beta_p$ give the distribution of the adjoint Hodge numbers $h^p_a$.  

We will need to have an upper bound on $\beta_0$ and a lower bound on $\sum_{p>0}\beta_p$.  

\begin{lemma} 
\cite[Lemma 3.1]{Ji}
\label{lem: upp}
    For all $n$, we have $\beta_0 \le \frac{\sqrt{3}}{\sqrt{n+4}}$. 
\end{lemma}

\begin{lemma} 
\cite[Lemma 3.4]{Ji}
\label{lem: low}
    For all $n$, we have $\sum_{p>0}p\beta_p > \frac{\sqrt{n}}{4\sqrt{3}}-\frac{1}{2}$.   
\end{lemma}

We recall the numerical conditions of Theorem \ref{thm: zariski non-density}. 
\begin{equation}
\label{condition 1}
    \sum_{q>0} h^q_a \ge \dim X + h^0
\end{equation}
        and 
\begin{equation}
\label{condition 2}
    \sum_{q>0} qh^q_a > T_G(\dim X + h^0_a) + T_G(\dim X + \frac{3}{2}h^0_a)
\end{equation}

In general $\dim X$ is very small compared to $h_a^0$, because $\dim X$ is approximated by $\Vol(\Delta)$ while $h^0_a = \sum_q \Vol(\Delta)^2A(n, q)^2 > 2\Vol(\Delta)$.  Thus it is generally sufficient to check that $\sum_{q>0} qh^q_a > 2T_G(2h_a^0)$.

\begin{proposition}
\label{prop: eulerian check}
    Take $h_a^q$ to be the adjoint Hodge numbers constructed from $h^q = A(n, q)V$ where $V=\Vol(\Delta)$, in the case of $\textbf{H} = GL_n$.  Then if $n\ge 13000$, we have 
    \[
    \sum_{q>0} qh^q_a > 2T_G(2h_a^0).
    \]
\end{proposition}
\begin{proof}
    Let $H = \sum_{q}h_a^q$.  Since $\beta_p = \frac{h_a^p}{H}$, by Lemma \ref{lem: low} we have $\sum_{q>0}qh_a^q > \frac{H\sqrt{n}}{4\sqrt{6}}$.  Thus  it suffices to show that 
    $
 T_G(2h_a^0) <  \frac{H\sqrt{n}}{4\sqrt{6}} 
    $ 
    By Lemma \ref{lem: upp} and since $n\ge 13000$, we have $2h_a^0< \frac{2\sqrt{3}H}{\sqrt{n}}$, so it suffices to show that 
    \[
T_G( \frac{2\sqrt{3}H}{\sqrt{n}}) \le \frac{H\sqrt{n}}{4\sqrt{6}}.
    \]
    Let $\eps = \frac{1}{48\sqrt{2}}$.  The contribution of the Hodge numbers below $\eps n$ is at most $(\eps n)(\frac{2\sqrt{3}H}{\sqrt{n}})= \frac{H\sqrt{n}}{16\sqrt{6}}$. 

    For the Hodge numbers above $\eps n$, we can use the variance bound $\sum_{p>0} p^2\beta_p = \frac{n+1}{12}$.  Then $\sum_{q > \eps n}qh_a^q < H(\eps n)^{-1}\frac{n+1}{12}\le \frac{4\sqrt{2}H(n+1)}{n}$.  We need this value to be at most $\frac{H\sqrt{n}}{16\sqrt{6}}$ for the desired bound to be satisfied. This occurs when $n\ge 13000$. 
\end{proof}

Now we consider the $GO_n$ case.  In reality, the numbers are close enough to the $GL_n$ case and our bounds are weak enough that the same results are likely to hold.  However, at the expense of the strength of the results we can instead use the following result. 

\begin{lemma}
    \cite[Lemma 4.9]{Ji} Let $h^j_a$ denote the adjoint Hodge numbers in the case of $\textbf{H} = SO_n$ with total sum $H$.  For $n\ge 40000$, we have  \[
h^0_a < \frac{\sqrt{13}H}{2\sqrt{n}}, \quad \sum_{p>0} ph^p_a > \frac{H\sqrt{n}}{11}. 
    \]
\end{lemma} 
This serves as the analogue of Lemmas \ref{lem: upp} and \ref{lem: low} in the $SO_n$ case.  By the proof of \cite[Proposition 4.10]{Ji}, which is the exact same argument as what we did in the proof of Proposition \ref{prop: eulerian check}, we reach the following result. 

\begin{proposition}
\label{prop: so check}
    Take $h_a^q$ to be the adjoint Hodge numbers constructed from $h^q = A(n, q)V$ where $V=\Vol(\Delta)$, in the case of $\textbf{H} = SO_n$.  Then if $n\ge 500000$, we have 
    \[
    \sum_{q>0} qh^q_a > 2T_G(2h_a^0).
    \]
\end{proposition} 

In conclusion, for $n\ge 500000$, if $\textbf{H} = GL_n$ or $SO_n$ and we make use the approximation of $A(n, q)\Vol(\Delta)$ for the Hodge-Deligne numbers, then the corresponding adjoint Hodge numbers satisfy the numerical conditions of Theorem \ref{thm: zariski non-density}.  In what follows, we will see that in some scenarios this approximation is close enough so that this conclusion does indeed hold.   

\subsection{Verification for the case of the truncated rectangular prism} 
When the polyhedron in question is a rectangular prism with a corner cut off with certain dimensions, we can estimate the Hodge-Deligne weights.   

\begin{proposition}
\label{prop: asy weights}
    Fix $n\ge 500000$ and let $\Delta$ be a rectangular prism with side lengths $a_1, a_2, \ldots, a_n$ with a corner with volume $o(\Vol(\Delta))$ cut off.  Then as $\min_i(a_i)\rightarrow \infty$, for sufficiently large $m$ and any choice of character, the Hodge-Deligne weights satisfy 
    \[
h^{p,q} \sim A(n, q)\Vol(\Delta). 
    \] 
    Furthermore, they satisfy the numerical conditions of Theorem \ref{thm: zariski non-density}. 
\end{proposition} 
\begin{proof}
Recall that Lemma \ref{lem: Laurent poly} gives
 \begin{equation}
 \label{eqn: m points}
     (-1)^q\sum_{i\ge 0}(-1)^i[H^i(\overline{Y_m}, \Omega^{n-1-q}_{(\ol{Y_m}, D)})] = \sum_{i=0}^q (-1)^i\dbinom{n+1}{i}L((q+1-i)m\Delta). 
 \end{equation}
 Then up to a function of $n$, the Hodge-Deligne weights for an order $m$ character $\chi$ of $\G_m^n$ given by the tuple $(a_1, \ldots, a_n)\in (\Z/m)^n$ are given by the integer points indicated on the RHS whose coordinates are equivalent to $(a_1, \ldots, a_n)\pmod{\Z/m}$.  Because we are scaling the original polyhedron $\Delta$ by $m$, by scaling back down by $m$ this is equivalent to looking at those points whose fractional residues are $\frac{a_i}{m}$ for each $i$.  

 If we let $\Delta$ be the entire rectangular prism without the removed corner, we note that the number of such points is bounded between $\prod_{i=1}^n (a_i-1)$ and $\prod_{i=1}^n a_i$.  By our choice of $a_i$ sufficiently large, we can ensure that this is within $o(\Vol(\Delta))$ of $\Vol(\Delta)$.  Now because the corner cut off has volume $o(\Vol(\Delta))$, this statement remains true when considering $\Delta$ with the corner cut off.  Then comparing the RHS of equation \ref{eqn: m points} with what it would be replacing $L((q+1-i)m\Delta)$ with $((q+1-i)m)^{n}\Vol(\Delta)$, we see that the difference is still $o(\Vol(\Delta))$.  Using the explicit formula for the Eulerian numbers as before, this implies   \[
h^{p,q} \sim A(n, q)\Vol(\Delta). 
    \] 
    as $\min_i(a_i)\rightarrow \infty$, as desired.  

Now to show the numerical conditions of Theorem \ref{thm: zariski non-density} are satisfied, in this scenario it suffices to show the conclusion of Proposition \ref{prop: eulerian check}.  The same argument holds, except now we must contend with additional factors of the form $\sum_{p}(h^p - H\beta_p)$ when proving the desired inequality  $T_G( \frac{2\sqrt{3}H}{\sqrt{n}}) \le \frac{H\sqrt{n}}{4\sqrt{6}}$.  Those additional factors are bounded above by $H\eps$ for any $\eps>0$ while the terms in the desired inequality are linear in $H$; thus the inequality still holds.  (This argument is worked out in detail in an essentially identically setting with different numbers in \cite[Proposition 3.5]{Ji}.)   
\end{proof} 

Next, we want to ensure that the big monodromy condition and the numerical conditions can be simultaneously satisfied.  The next proposition explains how to construct polyhedra for which these both hold. 
\begin{proposition}
\label{prop: exist}
    Let $\Delta$ be a rectangular prism with side lengths $a_1, a_2, \ldots, a_n$ with a corner with volume $o(\Vol(\Delta))$ cut off, as in Proposition \ref{prop: asy weights}.  Then the $a_i$ and the dimensions of the corner can be chosen in such a way that the conclusion of Proposition \ref{prop: asy weights} hold as well as the big monodromy condition. 
\end{proposition}
\begin{proof}
    Recall that for the big monodromy condition to hold, we must ensure that the $n!\Vol(\Delta)$ is a prime.  By Proposition \ref{prop: asy weights}, we can pick the $a_i$ such that their minimum is sufficiently large so that, after cutting off a corner with dimensions $b, 1, \ldots, 1$, the numerical conditions will be satisfied.  Choose $a_1$ to be sufficiently larger than the rest.  Then by the prime number theorem, we are able to choose some $1\le b < a_1$ for which $n!\Vol(\Delta)$ is a prime, as desired. 
\end{proof}

\section{Application to the Shafarevich conjecture} 
We will now put together the results of the previous sections to given application to the Shafarevich conjecture.  Given a fixed Newton polyhedron $\Delta$, we are concerned with the corresponding hypersurfaces nondegenerate with respect to $\Delta$ that are also primitive, as defined below. 
\begin{definition}
    Let $K$ be a number field and let $H\subset (\G_m^n)_K$ be a hypersurface.  We say that $H$ is primitive if it is not invariant under translation by any $x\in \G_m^n(\bar{\Q})$.
\end{definition}  

We want to apply Theorem \ref{thm: zariski non-density} to an appropriate parameter space of hypersurfaces.  Recall that one of the inputs to this theorem is a big monodromy statement.  We will say that $\Delta$ has big monodromy if its convolution monodromy group is large; i.e., $GL_N, SO_N$, or $Sp_N$.  (In such situations, $\textbf{H}$ in Theorem \ref{thm: zariski non-density} is taken to be $GL_N, GSp_N$, or $GO_N$ respectively.)  Recall that Corollary \ref{cor: big} is used to translate this big monodromy statement into the ``strongly $c$-balanced" assumption used in Theorem \ref{thm: zariski non-density}.   

\begin{theorem}
\label{thm: finiteness}
    Fix a Newton polyhedron $\Delta$ with big monodromy.  Let $S$ be a finite set of primes of a number field $K$.  Then assuming the numerical conditions are satisfied, up to translation there are only finitely many smooth primitive hypersurfaces in $(\G_m^n)_K$ nondegenerate with respect to their Newton polyhedron $\Delta$ with good reduction outside $S$. 
\end{theorem}

\begin{proof}
    [Proof of Theorem \ref{thm: finiteness}] 
    Working over $\O_{K, S}$, let $X'$ be the affine scheme parameterizing smooth hypersurfaces in $(\G_m^n)_K$ nondegenerate with respect to their Newton polyhedron $\Delta$.  Let $H_{\univ}\subseteq X'\times \G_m^n$ be the universal family. 
    Because Theorem \ref{thm: zariski non-density} only works over $\Q$, we will consider the Weil restriction of $X'$ as follows.  Let $S'$ be a finite set of primes of $\Z$ such that $\O_{K, S}$ is finite \'etale over $\Z[1/S']$.  Then we have a canonical bijection of sets between $X'(\O_{K, S})$ and $\Res_{\Z[1/S']}^{\O_{K, S}}X'(\Z[1/S'])$.  Furthermore, the universal family over $\Res_{\Z[1/S']}^{\O_{K, S}}X'$ is given by 
    \[
H_{\univ, \Q} = H_{\univ}\times_{X'}(\Res^{\O_{K, S}}_{\Z[1/S']}X'\times_{\Z[1/S']}\O_{K, S}). 
    \]

    Let $X$ be an irreducible component of the Zariski closure of $(\Res_{\Z[1/S']}^{\O_{K, S}} X')(\Z[1/S])^{prim}$, the subset corresponding to the primitive hypersurfaces, in $\Res_{\Z[1/S']}^{\O_{K, S}} X'$. 
    Let $Y = H_{\univ, \Q}\times_{\Res_{\Z[1/S']}^{\O_{K, S}} X'} X$ be the pullback of the universal family of hypersurfaces to $X$.  
     Possibly after enlarging $S$ and $S'$, the morphism $Y\rightarrow X$ spreads out to a smooth family $\mc{Y}\rightarrow \mc{X}$.  
          Since the $S$-integral points are Zariski dense in $X$ by construction, they are also dense in $\mc{X}$. 
    Our goal is to show that $Y\subset X_K\times_K \G_m^n$ is a translate of a constant hypersurface $H_0\subset \G_m^n$ by a section $s\in \G_m^n(X_K)$. 
 Indeed, since there are finitely many irreducible components in the Zariski closure, this will imply that every primitive hypersurface in the moduli space is given by a translate by one of finitely many.  We will achieve this goal by showing that otherwise we can apply Theorem \ref{thm: zariski non-density}, which will contradict Zariski-density of the $S$-integral points in $\mc{X}$.  

 Assume that $Y$ is not such a translate.  Let $G$ be the convolution monodromy group of the constant sheaf on $Y_{\overline{\eta}}$ and let $G^*$ be the commutator subgroup of the identity component of $G$.  By the big monodromy assumption, we may apply  Corollary \ref{cor: big} to obtain an embedding $\iota\colon K\rightarrow \C$ and a torsion character $\chi$ of $\pi_1((\G_m^n)_{\iota})$ satisfying the big monodromy condition in the condition of Lemma \ref{lem: gmon}.  Let $I$ be the corresponding orbit of characters and let $\msf{V}_I$ be the corresponding Hodge-Deligne system.   
 
 We now wish to apply Theorem \ref{thm: zariski non-density} to $\msf{V}_I$.  By Construction \ref{construction}, $V_I$ is a polarized, integral $E$-module with $\textbf{H}$-structure.  Furthermore, $V_I$ has integral Frobenius eigenvalues by the generalized Weil conjectures.  By Lemma \ref{lem: gmon}, the differential Galois group of $\msf{V}_I$ is a strongly 1-balanced subgroup of $G^0$, accounting for the big monodromy condition.  Finally, the numerical conditions hold by assumption.  Thus the conditions of Theorem \ref{thm: zariski non-density} are satisfied, so the $S$-integral points in $\mc{X}$ are not Zariski dense.  This gives the desired contradiction.   
\end{proof}

\begin{corollary}
    If we take $\Delta$ to be a Newton polyhedron satisfying the conditions of Proposition \ref{prop: asy weights} such that $n!\Vol(\delta)$ is a prime, then Theorem \ref{thm: finiteness} holds for $\Delta$.  
\end{corollary}
\begin{proof}
    To apply Theorem \ref{thm: finiteness}, we need to check the big monodromy condition and the numerical conditions.  The big monodromy condition was checked in Corollary \ref{cor: prism monodromy}, while the numerical conditions were checked in Proposition \ref{prop: asy weights}. 
\end{proof} 

We recall that, as checked in the Proof of Corollary \ref{cor: prism monodromy}, all hypersurfaces with such a Newton polyhedron are primitive, and furthermore that this represents an infinite family of possible Newton polyhedra $\Delta$, as shown in the proof of Proposition \ref{prop: exist}.

\appendix

\bibliographystyle{alpha}
\bibliography{main}

\end{document}